\documentclass[12pt]{amsart}
\usepackage{amsthm, amsfonts, amssymb, color}
 \usepackage{mathrsfs}
\usepackage{enumerate}

\usepackage{amsmath}
 \usepackage{amstext, amsxtra}
  \usepackage{txfonts}
 \usepackage[colorlinks, linkcolor=black, citecolor=blue, pagebackref, hypertexnames=false]{hyperref}
\usepackage{graphicx}
\usepackage[symbol]{footmisc}
\usepackage{cite}
\usepackage{geometry}
 \allowdisplaybreaks

\numberwithin{equation}{section}
\theoremstyle{definition}

\newtheorem{Theorem}{Theorem}[section]
\newtheorem{Proposition}{Proposition}[section]
\newtheorem{Lemma}{Lemma}[section]
\newtheorem{Corollary}{Corollary}[section]

\numberwithin{equation}{section}

\title[Global Schrodinger map flows]
{
Global Schr\"odinger map flows to K\"ahler manifolds with small data in critical Sobolev spaces: II. High dimensions. }

\author[Z. Li ]
{Ze Li}

\address{Ze Li
\newline\indent
School of Mathematics and Statistics, Ningbo University
\newline\indent
Ningbo, 315000, Zhejiang, P.R. China
}
\email{rikudosennin@163.com}


\begin{document}

\begin{abstract}
In this paper, we prove that the Schr\"odinger map flows from $\Bbb R^d$ with $d\ge 3$ to compact K\"ahler manifolds with small initial data in critical Sobolev spaces are global. This is a companion work of our previous paper  \cite{LIZE} where the energy critical case $d=2$ was solved.   In the first part of this paper, for heat flows from $\Bbb R^d$ ($d\ge 3$) to Riemannian manifolds with small data in critical Sobolev spaces, we prove the decay estimates of moving frame dependent quantities in the caloric gauge setting, which is of independent interest and may be applied to other problems. In the second part,  with a key bootstrap-iteration  scheme in our previous  work \cite{LIZE}, we apply these decay estimates to the study of Schr\"odinger map flows by choosing caloric gauge.  This work with our previous work solves the open problem raised by Tataru.
\end{abstract}
\maketitle

\section{Introduction }

Let $(\mathcal{M},g)$ be a Riemannian manifold, and $(\mathcal{N},J,h)$ be a K\"ahler manifold. Given a map $u:\Bbb R^d\to \mathcal{M}$, the Dirichlet energy $\mathcal{E}(u)$ is defined by
\begin{align}\label{hia1}
\mathcal{E}(u)=\frac{1}{2}\int_{\Bbb R^d}|du|^2dx.
\end{align}
The heat flow is the gradient flow of energy functional $\mathcal{E}(u)$.  A map $u(x,t):\Bbb R^d\times [0,\infty)  \to \mathcal{M}$ is called heat flow (of harmonic maps) if $u$ satisfies
\begin{align}\label{hia2}
\begin{cases}
u_t =\tau(u)\\
u\upharpoonright_{t=0} = u_0.
\end{cases}
\end{align}
Here, the tension field $\tau(u)$ is defined by
\begin{align*}
\tau(u)=\sum^{d}_{j=1}\nabla_j\partial_ju
\end{align*}
where $\nabla$ denotes the induced covariant derivative on the pullback bundle $u^*T\mathcal{M}$.

The Hamiltonian analogy of heat flows is the so-called Schr\"odinger map flow. A map $u(x,t):\Bbb R^d\times \Bbb R\to \mathcal{N}$ is called Schr\"odinger map flow (SL) if $u$ satisfies
\begin{align}\label{hia3}
\begin{cases}
u_t =J\tau(u)\\
u\upharpoonright_{t=0} = u_0.
\end{cases}
\end{align}

(\ref{hia2}) is related to the liquid crystal theory (see e.g. \cite{huangFLhuang} ), while (\ref{hia3}) plays a fundamental role in solid-state physics (\cite{huangLLhuang}).

In our previous work \cite{LIZE}, we proved the global well-posedness of 2D Schr\"odinger map flows into compact K\"ahler manifolds with small energy initial data.  We noticed that to prove global well-posedness it is essential to firstly establish the parabolic decay estimates of differential fields and connection coefficients associated with heat flows in the caloric gauge setting. Hence, in the first part of this paper, we prove
these decay estimates for heat flows with small data in critical Sobolev spaces with $d\ge 3$.  In the second part, we apply decay estimates to prove global existence of SL by bootstrap-iteration argument, dynamic separations involved in our previous work \cite{LIZE} and some adaptations to high dimensions.

We briefly recall the following non-exhaustive list of works on Cauchy problems and dynamic behaviors of SL. More details and references can be found in \cite{LIZE}. The local Cauchy theory of SL was developed by Sulem-Sulem-Bardos \cite{huangSSBhuang}, Ding-Wang \cite{huangDWhuang}, McGahagan \cite{huangMchuang}. The global theory for Cauchy problem was pioneered by Chang-Shatah-Uhlenbeck \cite{huangCSUhuang}, Bejenaru \cite{huangBhuang}, Ionescu-Kenig \cite{huangIK1huang,huangIKhuang}. The small data global theory in critical Sobolev spaces for target $\Bbb S^2$ was completed by Bejenaru-Ionescu-Kenig \cite{huangBIKhuang}($d\ge 4$) and Bejenaru-Ionescu-Kenig-Tataru \cite{huangBIKThuang} ($d\ge 2$). In the other direction, for large data in equivariant classes, global theories on stability/instability and threshold scattering were studied by works of Gustafson, Kang, Tsai, Nakanish \cite{huangGKT1huang,huangGNThuang} and Bejenaru, Ionescu, Kenig, Tataru \cite{huangBIKT2huang, huangBIKT3huang}.  The singularity formulation in equivariant class was achieved by  Merle-Raphael-Rodnianski \cite{huangMPRhuang} and Perelman \cite{huangPhuang}.  And for large data without equivariant assumptions, Rodnianski-Rubinstein-Staffilani \cite{huangRRShuang} studied global regularity for $d=1$ and Dodson, Smith \cite{huangDShuang,huangS2huang} studied the conditional global regularity for solutions with controlling dispersed norms for $d=2$.

One of the main part of this paper is the decay estimates of moving frame dependent quantities of heat flows in the caloric gauge setting. The $d=2$ case was established by Tao \cite{huangTaohuang} for the $\Bbb H^{n}$ target and by Smith \cite{huangSmithhuang} for general targets below threshold. Since $d\ge 3$ is energy supercritical, generally, these decay estimates could be expected only for small data. When $d\ge 3$ is even, the issue is relatively easy in the small data case. In fact, Bochner inequalities, bootstrap and Sobolev inequalities will suffice. The case when $d$ is odd requires much more efforts. On one side, the involved quantities such as connection coefficients, curvature terms, depend on both frames and the map itself, and geometric inequalities such as Bochner inequalities only provide bounds for covariant derivatives which are of integer orders. On the other side, the critical Sobolev spaces for odd dimensions are of fractional order. The conflict becomes prominent while bounding curvature dependent quantities. To solve this problem, we use parallel transport and difference characterization of Besov spaces. In fact, the difference characterization enables us to avoid directly apply fractional derivatives to geometric quantities. And the parallel transport enables us to compare geometric quantities at different points of the manifold.

The core result of  this paper is the global existence of solutions to SL with small data in critical Sobolev spaces. In order to state it, we introduce several notations.
Suppose that $\mathcal{N}$ is isometrically embedded into $\Bbb R^{N}$. Denote the embedding map by $\mathcal{P}$. Given a point $Q\in \mathcal{N}$, define the extrinsic Sobolev space $H^{k}_Q$ by
\begin{align*}
{ {H}^{k }_Q}:=\{u:\Bbb R^d\to \Bbb R^M\mid u(x)\in \mathcal{N}\mbox{ } a.e. {\mbox{  }}{x\in\Bbb R^d}, \| u-Q\|_{H^{k}(\Bbb R^d)}<\infty\},
\end{align*}
with the metric $d_{Q}(f,g)=\|f-g\|_{H^k}$. Let
\begin{align*}
\mathcal{Q}(\Bbb R^d,\mathcal{N}):=\bigcap^{\infty}_{k=1}{ {H}^{k }_Q}.
\end{align*}

Our main theorems are as follows:
\begin{Theorem}\label{XS2}
Let $\mathcal{N}$ be a compact K\"ahler manifold which is isometrically embedded into $\Bbb R^{N}$. Let $Q\in\mathcal{N}$ be a fixed given point.
Let $u_0\in  \mathcal{{Q}}(\Bbb R^d,\mathcal{N})$ with $d\ge 3$. There exists a sufficiently small constant $\epsilon_* >0$ depending only on $d$ such that if $u_0$  satisfies
\begin{align}
\|u_0\|_{\dot{H}^{\frac{d}{2}}_{x}}\le \epsilon_*,
\end{align}
then (\ref{hia3}) with initial data $u_0$ evolves into a global unique solution $u\in  C(\Bbb R;\mathcal{Q}(\Bbb R^d,\mathcal{N}))$. Moreover, as $t\to\infty$ the solution
$u(t)$ converges to the constant map $Q$ in the sense that
\begin{align}\label{FFF}
\lim\limits_{t\to\infty}\|u(t)-Q\|_{L^{\infty}_{x}}=0.
\end{align}
\end{Theorem}

The following theorem  proves the uniform bounds and well-posedness results analogous to that of \cite{huangBIKThuang}.
\begin{Theorem}\label{Two}
Let $d\ge 3$, $\sigma_1\ge \frac{d}{2}$. Let $\mathcal{N}$ be a compact K\"ahler manifold which is isometrically embedded into $\Bbb R^N$, and let $Q\in\mathcal{N}$ be a given point.
 There exists a sufficiently small constant $\epsilon_{\sigma_1,d} >0$ depending only on $\sigma_1,d$ such that for any initial data $u_0\in \mathcal{{Q}}(\Bbb R^d,\mathcal{N})$ with $\|u_0-Q\|_{\dot{H}^{\frac{d}{2}}}\le \epsilon_{\sigma_1,d}$, the global solution $u=S_{Q}(t)u_0\in  C(\Bbb R;\mathcal{Q}(\Bbb R^d, \mathcal{N}))$  constructed in Theorem  \ref{XS2}  satisfies the uniform bounds
\begin{align}\label{XcDfg}
\sup_{t\in \Bbb R} \| u(t)-Q \|_{ {\dot H}^{\sigma}_x \cap {\dot H}^1_x}\le C_{\sigma}(\|u_0-Q\|_{H^{\sigma}_x}), \mbox{ }\forall \sigma\in [\frac{d}{2},\sigma_1].
\end{align}
In addition, for any $\sigma\in[\frac{d}{2},\sigma_1]$, the operator $S_{Q}$ admits a continuous extension
\begin{align*}
S_{Q}:\mathfrak{B}^{\sigma}_{\epsilon_{\sigma_1,d}}\to    C(\Bbb R; {\dot{H}}^{\sigma}_Q\cap {\dot{H}}^{\frac{d}{2}-1}_Q),
\end{align*}
where we denote
\begin{align*}
\mathfrak{B}^{\sigma}_{\epsilon}:=\{f\in \dot{H}^{\frac{d}{2}-1}_Q\cap {\dot{H}}^{\sigma}_Q: \|f-Q\|_ {\dot{H}^{\frac{d}{2}}}\le \epsilon\}.
\end{align*}
\end{Theorem}

{\bf{Remark 1.1}}
Tataru raised the proof  of global well-posedness for small initial data in the critical Sobolev spaces for general compact K\"ahler targets as an open  problem in the survey report \cite{Khuanghuang}. Our previous work \cite{LIZE} solved the case $d=2$. Here, Theorem \ref{XS2} solves the case $d\ge 3$.

{\bf{Remark 1.2}}
We also prove the  global regularity of heat flows for $d\ge 3$  with data of small $\dot{H}^{\frac{d}{2}}$ norms.
For the global regularity of heat flows with small data in energy supercritical dimensions, we recall the remarkable results obtained by Struwe \cite{huangstruwehuang}. Struwe \cite{huangstruwehuang} proved in $d\ge 3$ that for any initial data $v_0$ satisfying $\|d v_0\|_{L^{\infty}_x(\Bbb R^{d})}\le K$ there exists $\epsilon(K)$ such that if $\mathcal{E}(v_0)\le \epsilon (K)$ then the solution of heat flow equation is global and converges to constant map as time goes to infinity. In our work, we impose no smallness condition on the energy but require critical Sobolev norms to be small. These two settings are two different styles of giving small data.

{\bf{Notations.}} We fix two constants $\vartheta=1-\frac{1}{10^{10}}$, $\delta=\frac{1}{d10^{100}}$. The notation $A\lesssim B$ means there exists some $C>0$ such that $A\le CB$. We denote $P_{k}$ with $k\in\Bbb Z$ the Littlewood-Paley projection with Fourier multiplier supported in the frequency annual $\{2^{k-1}\le |\eta| \le 2^{k+1}\}$.

$\mathcal{N}$ denotes the compact target K\"ahler manifold. And
we will also use $\mathcal{M}$  to denote a closed Riemannian manifold in the heat flow part.
The connections of $T\mathcal{M}$  and  $u^{*}T\mathcal{M}$ are denoted by $\widetilde{\nabla}$ and ${\nabla}$  respectively. Without confusions, we also denote connections of $T\mathcal{N}$  and  $u^{*}T\mathcal{N}$ by $\widetilde{\nabla}$ and ${\nabla}$  respectively. Let $\bf{R}$ denote the curvature tensor of $\mathcal{M}$ or $\mathcal{N}$.

\subsection{Caloric Gauges}

We recall some background materials on gauges.
\subsubsection{Moving frames}
In this subsection, we make the convention that Roman indices range in $\{1,...,m\}$ or $\{1,...,d\}$ according to the context. Let $\Bbb I=[0,\infty)\times [-T,T]$.
Let $\mathcal{M}$ be an m-dimensional Riemannian manifold and $v:\Bbb I\times \Bbb R^d\to \mathcal{M}$ be a smooth map. Let $\{e_i(s,t,x)\}^{m}_{i=1}$ be global orthonormal frames  for $v^*(T\mathcal{M})$. Then $v$ induces scaler fields $\{ \psi_j\}^{d+1}_{j=0}$ which are defined on $\Bbb I\times \Bbb R^d$ and take values in $\Bbb R^{m}$:
\begin{align}
\psi_j^{l} = \left\langle {\partial _jv,{e_l}} \right\rangle,\mbox{  }l=1,...,m.
\end{align}
where and in the following we make the convention that $j=0$ refers to $t\in[-T,T]$, $j_{d+1}$  refers to $s\in [0,\infty)$, and $j=1,...,d$ refers to $x_{j}\in \Bbb R$ respectively. Conversely, sections of the trivial vector bundle $(\Bbb I\times\Bbb R^d;\Bbb R^m)$ yield sections of $v^{*}T\mathcal{M}$:
\begin{align*}
\varphi\in \Gamma((\Bbb I\times\Bbb R^d;\Bbb R^m)) \dashrightarrow \varphi{\bf e}:=\varphi^{l}e_{l}\in \Gamma (v^{*}T\mathcal{M})\\
\varphi:= ( \langle X,e_{1}\rangle,..., \langle X,e_{m}\rangle) \in \Gamma((\Bbb I\times\Bbb R^d;\Bbb R^m)) \dashleftarrow X\in \Gamma (v^{*}T\mathcal{M}).
\end{align*}
The map $v$ induces a covariant derivative on the trivial vector bundle $([0,T]\times\Bbb R^d;\Bbb R^m)$ by
\begin{align*}
D_i\psi^{l}=\partial_i \psi^{l}+\sum^{m}_{q=1}\left([A_i]^l_q\right)\psi^q,
\end{align*}
where the induced connection coefficient matrices are defined by $[A_i]^q_p= \left\langle \nabla _ie_p,{e_q} \right\rangle$. The following identities are very often used throughout the paper:
\begin{align}
&\emph{torsion free identity} \mbox{  } \mbox{  } \mbox{  }D_{\mu}\psi_{\nu}=D_{\nu}\psi_{\mu}\\
&\emph{commutator identity}\mbox{  }([D_i,D_j]\varphi){\bf e} = {\bf R}(\partial_{i} v,\partial_{j} v)(\varphi{\bf e}).\label{b780}
\end{align}
(\ref{b780}) will be schematically written as
\begin{align*}
[D_i,D_j] = \partial_i  A_j-\partial_j A_i+[ A_i, A_j] =\mathcal{R}(\psi_i,\psi_j).
\end{align*}

The choices for gauges are very important for geometric dispersive PDEs. In the study of wave maps, Yang-Mills, Schr\"odigner map flows, several gauges are often used: Coulomb gauge (see  e.g. \cite{huangNSU2huang,huangKrhuang} for wave maps), caloric gauge (see \cite{huangTaohuang,huangSmithhuang,huangLawriehuang} for wave maps and \cite{huangOhhuang,huangOThuang,huangOT1huang} for hyperbolic Yang-Mills), microlocal gauge (see \cite{huangTaohuang,huangTataruhuang} for wave maps).

\subsubsection{Gauges for Schr\"odinger map flows }

If the target manifold is K\"ahler say $2n$-dimensional manifold $\mathcal{N}$ with complex structure $J$ and metric $h$, it is convenient to work with trivial vector bundle over $\Bbb I\times \Bbb R^d$ with fiber $\Bbb C^{n}$ instead of $\Bbb R^{2n}$. In this case, the moving frame is chosen to be $\{e_{i},Je_{i}\}^{n}_{i=1}$ and the induced scaled valued fields on $(\Bbb I\times \Bbb R^d;\Bbb C^n)$ are
\begin{align*}
\phi^{\gamma}_{i}=\psi^{\gamma}_{i}+\sqrt{-1}\psi^{\gamma+n}_i,\mbox{  }i=0,...,d; \mbox{ }\gamma=1,...,n.
\end{align*}
Let the $\Bbb C^{n}$ valued function $\varphi$ be a section of  $(\Bbb I\times \Bbb R^d;\Bbb C^n)$, then it induces a section of $v^{*}T\mathcal{N}$ via
\begin{align*}
\varphi {\bf e}:=\varphi^{\gamma}e_{\gamma}+\varphi^{\gamma+n}Je_{\gamma}.
\end{align*}
The induced derivative on the complex vector bundle $([0,T]\times \Bbb R^d;\Bbb C^n)$  is $D_i=\partial_i+{A_i}$,
where $\{A_i\}$ denote the induced connection coefficient matrices which are defined by
\begin{align*}
{A_{i}}^{\gamma}_{\beta}=[A_{i}]^{\gamma}_{\beta}+\sqrt{-1}[A_i]^{\gamma+n}_{\beta},\mbox{  }i=0,...,d; \mbox{ }\gamma,\beta=1,...,n.
\end{align*}
The torsion free identity and the commutator identity are recalled as follows:
\begin{align}
D_i\phi_j&=D_j\phi_i\label{ccccccccbb1}\\
([D_i,D_j]\varphi){\bf e}&=\left[\left(\partial_i  A_j-\partial_j A_i+[ A_i, A_j]\right)\varphi\right]{\bf e}=\mathbf{R}(\partial_iu, \partial_j u)(\varphi {\bf e}).\label{ccccccccxxxxbbbb2}
\end{align}

The heat flow equation shows the heat tension filed $\phi_s$ satisfies $\phi_s=\sum^{d}_{i=1}D_i\phi_i$. And in the heat flow direction, the differential fields $\{\phi_j\}^d_{j=1}$ satisfy
\begin{align}\label{b3}
\partial_s\phi_j= \sum^{d}_{i=1}D_iD_i\phi_j + \sum^{d}_{i=1}\mathcal{R}(\phi_j,\phi_i)\phi_i.
\end{align}
At the heat initial time $s=0$, under the Schr\"odinger map flow evolution the differential fields $\{\phi_j\}^d_{j=1}$ satisfy
\begin{align}
-\sqrt{-1}D_t\phi_j&=  \sum^{d}_{i=1}D_iD_i\phi_j+\sum^{d}_{i=1} \mathcal{R}(\phi_j,\phi_i)\phi_i.\label{b4}
\end{align}
And for $i=0,1,...,d$, $s>0$, the connection coefficients can be written as
\begin{align}
&[{A_i}]^{p}_{q}(s,t,x)= \int_s^\infty\langle {\mathbf{R}(v(s'))\left( {{\partial _s}v ,{\partial _i}v } \right)} e_{p},e_{q}\rangle ds'.\label{edf}.
\end{align}

Let $E$ be a manifold. Let $D$ be the equipped connection. Suppose that $\mathbb{T}$ is a type $(0,r)$ tensor on $E$. The $k$-th covariant derivative   of $\Bbb T$ is of $(0,r+k)$ type. And we denote
\begin{align*}
(D^{1}\Bbb T)&(X_1;Y_1,...,Y_r):=(D_{X_1}\Bbb T)(Y_1,...,Y_r)\\
(D^{k}\Bbb T)&(X_1,...,X_k;Y_1,...,Y_r)
:=\left[D_{X_k}(D^{k-1}\Bbb T)\right](X_1,...,X_{k-1};Y_1,...,Y_r)
\end{align*}
where $\{X_j\},$ $\{Y_{i}\}^{r}_{i=1}$ are vector fields on $E$.

\subsection {Road map for the proof of Theorem \ref{XS2}}

Let us outline the proof for Theorem 1.1.
The start point is the decay estimates of moving frame dependent quantities under Tao's caloric gauge.
\begin{Proposition}\label{1XS}
Given $d\ge 3$, let $L\ge (d+100)2^{d+100}$ be a fixed integer. Let $\mathcal{M}$ be an m-dimensional closed Riemannian manifold. And let $\mathcal{P}:\mathcal{M}\to \Bbb R^{M}$ be an isometric embedding. Let $v(s,x)$ be the solution of heat flow (\ref{hia2}) with initial data $v_0\in \mathcal{Q}(\Bbb R^d,\mathcal{M})$.
There exists a sufficiently small  constant $\epsilon_1>0$ depending only on $L,d$  such that
if
\begin{align}\label{ccBOOT}
\|v_0\|_{\dot{H}^{\frac{d}{2}}_{x}}\le \epsilon_1,
\end{align}
then $v$ is global with respect to $s\in\Bbb R^+$ and converges uniformly to $Q$ as $s\to\infty$. And there exists a unique Tao's caloric gauge $\{e_{l}\}^{m}_{l=1}$ for which
\begin{align}\label{FD}
\nabla_s e_{l}=0,\mbox{  }\lim_{s\to\infty} e_{l}=e^{\infty}_l, \mbox{ }l=1,...,m,
\end{align}
where $\{e^{\infty}_l\}$ are the given frames for $T_{Q}\mathcal{M}$.
 Denote the connection coefficients and differential fields under the caloric gauge by $\{A_{i}\}^{d}_{i=1}$ and $\{\psi_i\}^{d}_{i=1}$ respectively. Then we have
\begin{align}
\|\partial^{j}_xv\|_{\dot{H}^{\frac{d}{2}}_{x}}&\lesssim s^{-\frac{j}{2}}\epsilon_1\label{zzFD}\\
\|\partial^{j'}_xv\|_{L^{\infty}_{x}}&\lesssim s^{-\frac{j}{2}}\epsilon_1\label{1.8a}\\
\|\partial^{j}_x(d\mathcal{P}(e_{l})-\chi^{\infty}_{l})\|_{\dot{H}^{\frac{d}{2}}_{x}}&\lesssim s^{-\frac{j}{2}}\epsilon_1;\mbox{ }
\sum^{d}_{i=1}\|\partial^{j}_xA_i\|_{\dot{H}^{\frac{d}{2}-1}_{x}} \lesssim s^{-\frac{j}{2}}\epsilon_1; \mbox{ } \nonumber \\
\sum^{d}_{i=1}\|\partial^{j}_x\phi_i\|_{\dot{H}^{\frac{d}{2}-1}_{x}}&\lesssim s^{-\frac{j}{2}}\epsilon_1; \mbox{  }
\sum^{d}_{i=1}\|\partial^{j}_x\phi_i\|_{L^{\infty}_{x}} \lesssim s^{-\frac{j+1}{2}}\epsilon_1\nonumber
\end{align}
for any $0\le j\le L$,  $1\le j'\le L$, where $\chi^{\infty}_l=\lim_{s\to\infty}d\mathcal{P}(e_{l})$ for $l=1,...,m.$
\end{Proposition}

Proposition \ref{1XS} serves as the cornerstone for our proof of Theorem 1.1. In fact, it is the starting point and engine of our iteration scheme.

\subsubsection{ {\bf Outline  of Proof}}\label{DFvv}

Suppose that we have a solution $u\in C([-T,T];\mathcal{Q}(\Bbb R^d,\mathcal{N}))$ for SL with initial data $u_0$. Let $\{c_{k}(\sigma)\}$, $\{c_{k}\}$ be frequency envelopes associated with $u_0$:
\begin{align}
2^{\frac{d}{2}k}\|P_{k}u_0\|_{L^2_x}&\le c_{k}\\
2^{\frac{d}{2}k+\sigma k}\|P_{k}u_0\|_{L^2_x}&\le c_{k}(\sigma).
\end{align}
Let $v(s,t):\Bbb R^d\to \mathcal{N}$ be the heat flow initiated from $u(t)$.
Let $\{\phi_i\}^{d}_{i=0}$ and $\{A_{i}\}^{d}_{i=0}$ be differential fields and connection coefficients associated with $v$ under the caloric gauge. (The index 0 refers to $t$)
Assume that $u$ satisfies
\begin{align}
2^{\frac{d}{2}k}\|P_{k}u\|_{L^{\infty}_tL^2_x}&\le \epsilon^{-\frac{1}{2}}c_{k}\label{no1}\\
2^{\frac{d-2}{2}k }\|P_{k}\phi_x(\upharpoonright_{s=0})\|_{G_{k}(T)}&\le \epsilon^{-\frac{1}{2}}c_{k} \label{no2}.
\end{align}
{\bf Step 1. Before iteration.}{ In Step 1, we assume $\sigma \in[0,\vartheta]$, where $\vartheta=1-\frac{1}{10^{10}}$ is fixed throughout the paper.} \\

{\bf Step 1.1. Parabolic estimates along the heat direction. }
Let $\{b_{k}(\sigma)\}$ and $\{b_{k}\}$ be frequency envelopes of $\{\phi_i\upharpoonright_{s=0}\}^{m}_{i=1}$ in $G_{k}(T)$ norm:
\begin{align}
b_{k}(\sigma):=\sum^{d}_{i=1}\sup_{k'\in\Bbb Z}2^{-\delta|k-k'|}2^{\frac{d}{2}k'-k'}2^{\sigma k'}\|P_{k'}\phi_{i}\upharpoonright_{s=0}\|_{G_{k'}(T)},
\end{align}
and $b_{k}:=b_k(0)$.
Then the connection coefficients $\{A_i\}^{d}_{i=1}$ satisfy
\begin{align}\label{noc3}
\sum^{d}_{i=1}2^{\frac{d}{2}k-k}2^{\sigma k}\|P_{k}A_{i}(s)\|_{F_{k}(T)\bigcap S^{\omega}_{k}(T)}\lesssim (1+s2^{2k})^{-4}b_{k,s}(\sigma)
\end{align}
(\ref{noc3}) is essential to derive parabolic estimates for all other differential fields, especially it implies
\begin{align}
2^{\frac{d}{2}k-k}2^{\sigma k}\|P_{k}A_{t}\upharpoonright_{s=0}\|_{L^2(T)}&\lesssim \epsilon b_{k}(\sigma)\label{no4}\\
\sum^{d}_{i=1}2^{\frac{d}{2}k-k}2^{\sigma k}\|P_{k}A_{i}\upharpoonright_{s=0}\|_{L^{p_{d}}(T)}&\lesssim \epsilon b_{k}(\sigma).\label{no5}
\end{align}

{\bf Step 1.2. Estimates along the Schr\"odinger  direction. } By studying (\ref{b4}) and linear estimates in $G_{k} \mbox{ }v.s.\mbox{ } N_{k}$ spaces for linear Schr\"odinger equation established by \cite{huangBIKThuang}, (\ref{no4}) and (\ref{no5}) give
\begin{align}\label{nco5}
b_{k}(\sigma)\lesssim c_{k}(\sigma)
\end{align}
for all $\sigma \in[0,\vartheta]$. \\
{\bf Step 2. Iteration.}{ In Step 2, we assume $\sigma \in[0,2\vartheta]$.}

{\bf Step 2.1. Parabolic estimates along the heat direction. }
Let $\{b^{(1)}_{k}(\sigma)\}$ and $\{b^{(1)}_{k}\}$ be
\begin{align}\label{no3}
b^{(1)}_{k}(\sigma)=\left\{
                      \begin{array}{ll}
                        b_{k}(\sigma), & \hbox{ }{\rm if}\mbox{ }\sigma\in [0,\vartheta]\\
                        b_{k}(\sigma)+c_{k}( {\sigma}-\vartheta )c_{k}(\vartheta), & \hbox{ }{\rm if}\mbox{ }\sigma\in (1,2\vartheta]
                      \end{array}
                    \right.
\end{align}
Then the connection coefficients $\{A_i\}^{d}_{i=1}$ satisfy
\begin{align}\label{ano3}
\sum^{d}_{i=1}2^{\frac{d}{2}k-k}2^{\sigma k}\|P_{k}A_{i}(s)\|_{F_{k}(T)\bigcap S^{\omega}_{k}(T)}\lesssim (1+s2^{2k})^{-4}b^{(1)}_{k,s}(\sigma),
\end{align}
and  it implies
\begin{align}
2^{\frac{d}{2}k-k}2^{\sigma k}\|P_{k}A_{t}\upharpoonright_{s=0}\|_{L^2(T)}&\lesssim \epsilon b^{(1)}_{k}(\sigma)\label{ano4}\\
\sum^{d}_{i=1}2^{\frac{d}{2}k-k}2^{\sigma k}\|P_{k}A_{i}\upharpoonright_{s=0}\|_{L^{p_{d}}(T)}&\lesssim \epsilon b^{(1)}_{k}(\sigma).\label{ano5}
\end{align}
{\bf Step 2.2. Estimates along the Schr\"odinger  direction. } By (\ref{b4}), (\ref{ano4}) and (\ref{ano5}) show
\begin{align}
b_{k}(\sigma)\lesssim c_{k}(\sigma)+c_{k}(\vartheta)c_{k}(\sigma-\vartheta)
\end{align}
as well for arbitrary $\sigma \in[\vartheta,2\vartheta]$.
Thus, as a corollary of embedding $G_{k}(T)\hookrightarrow L^{\infty}_tL^2_x$, we get
\begin{align}\label{oicj1}
\sum^{d}_{i=1}2^{\sigma k}2^{{\frac{d}{2}k-k}}\|P_{k}\phi_{i}\|_{L^{\infty}_tL^2_x}\lesssim c^{(1)}_{k}(\sigma).
\end{align}
{\bf Step 3. Global regularity.} Doing iteration for $K$ times gives
\begin{align}\label{oipcj1}
\sum^{d}_{i=1}2^{\sigma k}2^{{\frac{d}{2}k-k}}\|P_{k}\phi_{i}\|_{L^{\infty}_tL^2_x}\lesssim c^{(K)}_{k}(\sigma).
\end{align}
with $\sigma\in[0,K\vartheta]$.
And transforming (\ref {oipcj1}) to bounds of $u$ gives
\begin{align}\label{oicj2}
\|u(t)\|_{L^{\infty}_t{\dot H}^{1}_x\bigcap{\dot H}^{L}_x}\lesssim C(\|u_0\|_{H^{L}_Q}).
\end{align}
for $L=\frac{d}{2}+K$.
By the local Cauchy theory of \cite{huangDWhuang,huangMchuang},  by taking $K=2$ we see $u$ is globally smooth provided $u_0\in \mathcal{Q}(\Bbb R^d,\mathcal{N})$.

\subsection{ Reduction to decay of heat flows }
   {\it Main ideas and overview of first time iteration.}
 The new difficulty in the general targets case is that the curvature term in the gauged equation depends not only on the differential fields but also the map $u$ itself. And the connection coefficients also depend on curvatures. This difficulty  was overcome by synthetically using kind of dynamic separation and iteration argument. Here, the dynamic separation may be seen as the ``freezing coefficient method" originally developed in the elliptic PDEs.

The curvature terms emerge in (\ref{edf}), (\ref{b3}), (\ref{b4}). The equation (\ref{edf}) is used to control connection coefficients (see (\ref{no3})), while  (\ref{b3}) and  (\ref{b4}) are used to track the evolution of differential fields along heat and Schr\"odinger direction respectively.

We will follow the framework of our previous work \cite{LIZE}, especially the bootstrap-iteration scheme. But in  dimensions $d\ge 3$, arguments of \cite{LIZE} can be largely simplified. The following scheme is a simplification of  [Section 4,5,6, \cite{LIZE}].

{\it Refined dynamical separation.} The curvature term $\mathcal{R}(\phi_s,\phi_i)$ in (\ref{edf}) can be schematically written as
\begin{align}\label{HHJnm}
\sum (\phi_x\diamond\phi_s)\langle {\bf R}(e_{j_0},e_{j_1})e_{j_2},e_{j_3}\rangle,
\end{align}
where the notation $f\diamond g$ for two $\Bbb C^n$ valued functions $f,g$ means a $\Bbb C^n$ valued function whose components are linear combinations of  $f^{\beta}g^{\alpha} $,  $f^{\alpha}\overline{g^{\beta}}$, $ \overline{f^{\alpha}}\overline{g^{\beta}}$, $\overline{f^{\alpha}} {g^{\beta}}$  for some $\alpha,\beta=1,..,n$.
The $ \{\phi_i\}^{d+1}_{i=1}$ part in (\ref{HHJnm}) will be controlled by bootstrap assumption. Denote the remainder part by
\begin{align}
\mathcal{G}(s)=\sum\langle {\bf R}(e_{l_0},e_{l_1})e_{l_2},e_{l_3}\rangle(s).
\end{align}
By caloric gauge condition, $\mathcal{G}$ can be expanded as
\begin{align*}
&\langle {\bf R}(e_{l_0},e_{l_1})e_{l_2},e_{l_3}\rangle(s)=\Gamma^{\infty}+\mathcal{U}_{0}+\mathcal{U}_{1}
\end{align*}
where we denote
\begin{align*}
\mathcal{U}_{0}&:=\Xi^{\infty}_l\int^{\infty}_s\sum^{2}_{i=1}(\partial_i\phi_i)^{l}ds'\\
\mathcal{U}_{1}&:=\Xi^{\infty}_l\int^{\infty}_s
\sum^{2}_{i=1}(A_i\phi_i)^{l}ds'+\int^{\infty}_s\sum^{2}_{i=1}(\partial_i\phi_i)^{l}\left(\mathcal{G}'\right)_lds'
+\sum\int^{\infty}_s(A_i\phi_i)^{l}(\widetilde{s})\left(\mathcal{G}'\right)_ld\widetilde{s}\\
(\mathcal{G}')_l&:=(\widetilde{\nabla}{\bf R})({e_l};e_{l_0},e_{l_1},e_{l_2},e_{l_3})-\Xi^{\infty}_l.
\end{align*}
and $\Gamma^{\infty}, \{\Xi^{\infty}_l\}$ are constant vectors.

Denote
\begin{align*}
h_{k}(\sigma):=\sup_{i\in\{1,...,d\},k'\in\Bbb Z}2^{-\delta|k-k'|}(1+s2^{2k'})^{4} 2^{\frac{d}{2}k'}2^{\sigma k'}\|P_{k'}\phi_i\|_{F_{k}}.
\end{align*}

In Step 1 of Section \ref{DFvv}, we additionally assume
\begin{align}\label{KKey}
2^{\frac{d}{2}k}\|P_k\mathcal{U}_1\|_{F_{k}}\lesssim (1+2^{2k}s)^{-5}h_{k}.
\end{align}
The key for Step 1 is to prove (\ref{noc3}).

{\it{Proof of  (\ref{noc3})}.}
With (\ref{KKey}) in hand, we observe  that for (\ref{noc3}) it suffices to prove the parabolic estimates of $P_{k}\mathcal{G}'$:
\begin{align}
2^{\frac{d}{2}k}\|P_{k}\mathcal{G}'\|_{L^{\infty}_tL^{2}_x}&\lesssim_{L} \|u\|_{L^{\infty}_t\dot{H}^{\frac{d}{2}}_x}(1+s2^{2k})^{-20}\label{ji00}\\
2^{\frac{d}{2}k}\|P_{k}\mathcal{G}'\|_{L^{\infty}_tL^2\cap L^{p_{d}}_{t,x}}&\lesssim_{L} 2^{-\sigma k} h_{k}(\sigma) (1+s2^{2k})^{-20} \label{D1}\\
2^{\frac{d}{2}k}\|P_{k}\partial_t\mathcal{G}'\|_{L^{p_d}_{t,x}}&\lesssim 2^{-\sigma k+2k} h_{k}(\sigma)\label{D2}
\end{align}
for $\sigma\in [0,\vartheta]$.  (\ref{ji00})  will be proved in Section 3 and Section 4 by geodesic parallel transport and difference characterization of Besov spaces. (\ref{D1}) and (\ref{D2}) will be proved in Section 5. Among them, (\ref{D2}) is the most difficult.

At the end of Step 1 of Section \ref{DFvv}, (\ref{KKey})  will be improved to
\begin{align}
2^{\frac{d}{2}k}\|P_k\mathcal{U}_1\|_{F_{k}}\lesssim \varepsilon (1+2^{2k}s)^{-4}2^{-\sigma k}h_{k}(\sigma),\mbox{ }\forall \sigma\in[0,\vartheta].\label{1.37a}
\end{align}
which by bootstrap shows (\ref{KKey}) holds. And thus (\ref{noc3}) follows.

{\it{Proof of  (\ref{D1}), (\ref{D2})}.}   By caloric gauge condition, $\mathcal{G}'$ can be further decomposed as
\begin{align*}
(\mathcal{G}')_{l}(s)&=\sum\left(\int^{\infty}_{s}\phi^{p}_s(s')(\widetilde{\nabla}^2{\bf R})({e_l}, e_{p};e_{l_0},e_{l_1},e_{l_2},e_{l_3})(s')ds'\right)d{s}'\\
&=\sum\left(\int^{\infty}_{s}\phi^{p}_s(s')d{s}'\right)\Omega^{\infty}_{lp}+\sum \int^{\infty}_{s}\phi^{p}_s(s')(({\mathcal{G}}^{''})_{lp}-\Omega^{\infty}_{lp})d{s}',
\end{align*}
where we denote
\begin{align*}
(\mathcal{G}^{''})_{lp}= (\widetilde{\nabla}^2{\bf R})({e_l},{e_{p}};e_{l_0},e_{l_1},e_{l_3},e_{l_4})-\Omega^{\infty}_{lp}
\end{align*}
Applying bilinear Littlewood-Paley decompositions reduces the proof of  (\ref{D1})-(\ref{D2}) to verify
\begin{align}
2^{\frac{d}{2}k-2k }\|P_{k}\phi_s\|_{L^{\infty}_tL^{2}_x\bigcap L^{p_{d}}_{t,x}}&\lesssim (1+s2^{2k})^{-M}2^{-\sigma k}h_{k}(\sigma)\label{ji1}\\
2^{\frac{d}{2}k }\|P_{k}({\mathcal{G}}^{''})\|_{L^{\infty}_tL^{2}_x\bigcap L^{p_{d}}_{t,x}}&\lesssim (1+s2^{2k})^{-M}\|u\|_{L^{\infty}_t\dot{H}^{\frac{d}{2}}_x}\label{ji2}\\
2^{\frac{d}{2}k-2k }(\|P_{k}\phi_t\|_{L^{p_d}_{t,x}}+\|P_{k}A_t\|_{L^{p_{d}}_{t,x}})&\lesssim 2^{-\sigma k}h_{k}(\sigma). \label{ji3}
\end{align}

{\it{Proof of  (\ref{oicj2}).}} Let $K=0$ in (\ref{oicj2}). To pass from the bounds for moving frame dependent quantities stated in (\ref{oicj1}) to the bounds of $u$ itself stated in (\ref{oicj2}), the key is to deduce frequency bounds for frames:
\begin{align}
2^{ {\frac{d}{2}k+\sigma k} }\left\|P_{k}\left((d\mathcal{P})(e_l)-\chi^{\infty}_l\right)\right\|_{L^{\infty}_tL^2_x}&\lesssim_{M}(1+s 2^{2k})^{-M}h_{k}(\sigma).\label{jii3}
\end{align}
for $\sigma\in[0,\vartheta]$.
Moreover, (\ref{ji1}) is a corollary of (\ref{jii3}) and
\begin{align}
2^{\frac{dk}{2}+\sigma k}\|P_{k}\partial_sv\|_{L^{\infty}_tL^2_x\bigcap L^{p_{d}}_{t,x}}\lesssim 2^{2k}(1+s 2^{2k})^{-L}h_{k}(\sigma)\label{ji4}.
\end{align}
Furthermore, using dynamic separation and bilinear Littlewood-Paley decomposition, for $\sigma\in[0,\vartheta]$, (\ref{jii3}) reduces to (\ref{ji4}) and
\begin{align}
&2^{ {\frac{dk}{2}} }\left\|P_{k}\left(({\bf D}d\mathcal{P})(e_{p};e_l)\right)\right\|_{L^{\infty}_tL^2_x\bigcap L^{p_{d}}_{t,x}}\lesssim_{M}  (1+s 2^{2k})^{-M}\|u\|_{L^{\infty}_t\dot{H}^{\frac{d}{2}}_x}.\label{ji5}
\end{align}
To prove (\ref{ji4}), except for using the heat flow equation, one also needs
\begin{align}
&2^{ {\frac{dk}{2}} }\left\|P_{k} DS^{l}_{ij}(v) \right\|_{L^{\infty}_tL^2_x\bigcap L^{p_{d}}_{t,x}}\lesssim (1+s 2^{2k})^{-M}\|u\|_{L^{\infty}_t\dot{H}^{\frac{d}{2}}_x}\label{ji6}.
\end{align}

{\it{Logic graph.}} We summarize the above reduction process as the following graph for convenience:
\begin{align*}
{\bf I.} &(\ref{KKey})\xrightarrow{+Eq. (\ref{edf})}  (\ref{noc3}) \xrightarrow{+Eq. (\ref{b3})} (\ref{no4}),(\ref{no5})\xrightarrow{+Eq. (\ref{b4})}(\ref{nco5})\\
{\bf II.}&\begin{tabular}{|c|}
  \hline
  (\ref{ji1})\\
(\ref{ji2})\\
  (\ref{ji3}) \\
  \hline
\end{tabular}
\rightarrow\begin{tabular}{|c|}
  \hline
  (\ref{D1})\\
(\ref{D2})\\
  \hline
\end{tabular}
\xrightarrow{+(\ref{noc3})}
  (\ref{1.37a})\\
{\bf III.} &\mbox{ } {\rm Results\mbox{ } in \mbox{ }Step\mbox{ } 1} +(\ref{1.37a})\xrightarrow{+Eq. (\ref{edf})}
\begin{tabular}{|c|}
  \hline
  (\ref{ano4})\\
  (\ref{ano5}) \\
  \hline
\end{tabular}
\xrightarrow[{+(\ref{nco5} )}]{{Eq. (\ref{b4})}}{{(\ref{oicj1})}}\\
{\bf IV.}&(\ref{ji6})\xrightarrow{+{\rm HF\mbox{ } Eq.}}(\ref{ji4})\xrightarrow{ {+ (\ref{ji5})}}(\ref{jii3})\left\{
                                                           \begin{array}{ll}
                                                             \xrightarrow{ {+ (\ref{ji4})}}(\ref{ji1}), & \hbox{ } \\
                                                              \xrightarrow{+(\ref{oicj1})}(\ref{oicj2})&\hbox{ }
                                                           \end{array}
                                                         \right.
\end{align*}
Therefore, it suffices to prove (\ref{ji00}), (\ref{ji2})-(\ref{ji3}),  (\ref{ji5})-(\ref{ji6}). Here, I and II are parts of the zeroth SL iteration, III belongs to the first time SL iteration, IV describes heat flow iteration.

{\it Remarks on iteration.} The above is just a toy model for iteration scheme up to once time. We emphasize that the key and the engine for heat flow iteration is to improve estimates of $\partial_s v$ step by step, and the key for SL iteration is to improve  (\ref{KKey}) after each time iteration. The true scheme is a sophisticated combination of bootstrap and the above iteration.

\noindent{\bf Remarks on frequency envelopes.} The notion of frequency envelopes introduced by Tao is very convenient in doing frequency estimates and becomes standard in the study of dispersive PDEs.
Due to the iteration argument used here, we need emphasize the ``order" of envelopes applied in our previous work \cite{LIZE}: We say a positive $\ell^2$ summable sequence  $\{a_{k}\}_{k\in\Bbb Z}$ is  a frequency envelope
of $\delta$ order if
\begin{align}
a_{k}\le a_{j}2^{\delta|k-j|}, \mbox{  }\forall k,j\in\Bbb Z.
\end{align}
Throughout the paper, all the frequency envelopes are assumed to be of $\delta$ order except $\{c^{(j)}_{k},c_{k},c_{k}(\sigma)\}_{j\in\Bbb N,k\in\Bbb Z}$. If we want to reach $\sigma=[K\vartheta]+1$ in Theorem 1.1, then  $\{c^{(j)}_{k}\}$, $\{c_{k}\}$, $\{c_{k}(\sigma)\}$, with ${j\in [0,K+1]}$ are defined to be $\frac{1}{2^{K+1}}\delta$ order.

\section{Preliminaries  on Function Spaces}

We recall the spaces developed by \cite{huangBIKThuang,huangLPhuang,huangIKhuang}.
Given a unite vector ${\vec e}\in \Bbb S^{d-1}$ we denote its orthogonal complement of $\Bbb R^d$ by $ \vec{e}^\bot$. The lateral space $L^{p,q}_{\vec e}$ is defined by
\begin{align}
\|g\|_{L^{p,q}_{ \vec{ e}}}=\left(\int_{\Bbb R}\left(\int_{{  \vec{e}}^{\bot}\times\Bbb R}\left|g(t,x_1{\vec e }+x')\right|^qdx'dt\right)^{\frac{p}{q}}dx_1\right)^{\frac{1}{p}},
\end{align}
with modifications if $p=\infty$ or $q=\infty$.
Given $T\in \Bbb R$, $k\in \Bbb Z$, define $I_k:=\{\eta\in \Bbb R^d: 2^{k-1}\le  |\eta|\le 2^{k+1}\}$ and
\begin{align*}
L^2_k(T):=\{g\in L^2([-T,T]\times \Bbb R^d): \widehat{g}(t,\eta) =0,\mbox{  }{\rm if} \mbox{  }(t,\eta)\in \Bbb R\times I_{k}\backslash [-T,T]\times \Bbb R^d\}.
\end{align*}
The main work function spaces designed by \cite{huangBIKThuang} are $F_{k}(T),G_k(T),N_k(T),S^{\omega}_{k}(T)$.
In our work, the concrete definitions of $N_{k},G_{k}$ are not needed. We just recall $F_{k},S^{\omega}_{k}$ here, see \cite{huangBIKThuang} for the original detailed construction. Let $p_d=\frac{2d+4}{d}$. Define the exponents $2_{\omega} $ and $p_{d,\omega}$ via
\begin{align}
\frac{1}{2_{\omega}}-\frac{1}{2}=\frac{1}{p_{d,\omega}}-\frac{1}{p_d}=\frac{\omega}{d}.
\end{align}
Let $S^{\omega}_k(T)$, $F_k(T)$ denote the normed space of functions in $L^2_k(T)$ for which the corresponding norm
\begin{align}
\|g\|_{S^{\omega}_k(T)}&:=2^{k\omega}\left(\|g\|_{L^{\infty}_tL^{2_\omega}_x}+\|g\|_{L^{p_d}_tL^{p_{d,\omega}}_x}
+2^{-\frac{kd}{d+2}}\|g\|_{L^{p_{d,\omega}}_xL^{\infty}_t}\right)\\
\|\psi\|_{F_k(T)}&:=\|\psi\|_{L^{\infty}_tL^2_x}+\|\psi\|_{L^{p_d}_{t,x}}+2^{-\frac{kd}{d+2}}\|\psi\|_{L^{p_d}_xL^{\infty}_t}+2^{-\frac{k(d-1)}{2}}\sup_{\vec{e} \in \Bbb S^{d-1}}\|\psi\|_{L^{2,\infty}_{ \vec{e}}}
\end{align}
is finite.

The following bilinear estimates are used to control the curvature terms.
\begin{Lemma}\label{1bi1}
If $|k_1-k|\le 4$, then
\begin{align}\label{2bi2}
\|P_k( {P_{\le k-4}g }P_{k_1} f)\|_{F_k(T)}\lesssim \|P_{\le k-4}g \|_{L^{\infty}}\|P_{k_1} f\|_{F_{k_1}(T)}.
\end{align}
If $|k_2-k_1|\le 8$, $k_1,k_2\ge k-4$, then
\begin{align}\label{3bi3}
\|P_k( P_{k_1}f P_{k_2} g)\|_{F_k(T)}\lesssim \left(2^{\frac{1}{2}(d -1)(k_1-k)}+2^{\frac{d}{d+2}(k_1-k)}\right)\|P_{k_2}g \|_{L^{\infty}}\|P_{k_1} f\|_{F_{k_1}(T)}
\end{align}
If $|k_2-k|\le 4$, $k_1\le k-4$, then
\begin{align}
&\|P_k( P_{k_1}f P_{k_2} g)\|_{F_k(T)}\lesssim 2^{\frac{d-1}{2}(k_1-k)} \|P_{k_1}f \|_{F_{k_1}(T)}\|P_{k_2} g\|_{L^{\infty}}\nonumber\\
&+2^{-\frac{d}{d+2}k}\|P_{k_1}f \|_{L^{\infty}}\|P_{k_2} g\|_{L^{p_d}_xL^{\infty}_t}+ \|P_{k_1}f \|_{L^{\infty}}\|P_{k_2} g\|_{L^{p_d}\bigcap L^{\infty}_tL^2_x}.\label{4bi4}
\end{align}
\end{Lemma}
\begin{proof}
(\ref{2bi2})-(\ref{3bi3}) follow directly by definition of $F_{k}(T)$.
For (\ref{4bi4}), we used
\begin{align*}
\|P_k( P_{k_1}f P_{k_2} g)\|_{L^{2,\infty}_{\vec e}}&\le \|P_{k_1}f \|_{L^{2,\infty}_{\vec{e}}}\|P_{k_2} g\|_{L^{\infty}}\\
\|P_k( P_{k_1}f P_{k_2} g)\|_{L^{p_d}_{x}L^{\infty}_t}&\le \|P_{k_1}f \|_{L^{\infty}}\|P_{k_2} g\|_{L^{p_d}_{x}L^{\infty}_t}\\
\|P_k( P_{k_1}f P_{k_2} g)\|_{L^{p_d}\cap L^{\infty}_tL^2_x}&\le \|P_{k_1}f \|_{L^{\infty}}\|P_{k_2} g\|_{L^{p_d}\cap L^{\infty}_tL^2_x}.
\end{align*}
\end{proof}

The following two trilinear estimates are used to study the extrinsic formulation of heat flow equations.
\begin{Lemma}[\cite{huangBIKThuang,LIZE}]\label{EH}
Let $F:\Bbb R^M\to \Bbb R$ be a smooth function and $v: \Bbb R^{d}\times [-T,T]\to \Bbb R^M$ be a smooth map.
Define
\begin{align*}
\beta_k&=\sum_{|k'-k|\le 20}2^{\frac{d}{2}k'}\|P_{k'}v\|_{L^{\infty}_t{L^2_x}}\\
\alpha_k&=\sum_{|k'-k|\le 20}2^{\frac{d}{2}k'}\| P_{k'}(F(v))\|_{L^{\infty}_tL^2_x}.
\end{align*}
Assume that $\|v\|_{L^{\infty}_x}\lesssim 1$ and $\sup_{k\in \Bbb Z}\beta_k\le 1$. Then
\begin{align}
&2^{\frac{d}{2}k}\|P_{k }F(v)(\partial_x v,\partial_x v)\|_{L^{\infty}_tL^2_x}\lesssim 2^{k}\beta_k \sum_{k_1\le k}\beta_{k_1}2^{k_1}+\sum_{k_2\ge k}2^{-d|k-k_2|}2^{2k_2}\beta^2_{k_2}\nonumber\\
&+ \alpha_{k}(\sum_{k_1\le k}2^{k_1}\beta_{k_1})^2+\sum_{k_2\ge k}2^{d(k-k_2)}2^{k_2}\alpha_{k_2}\beta_{k_2}(\sum_{k_1\le k_2}\beta_{k_1}2^{k_1}).
\end{align}
\end{Lemma}

\begin{Lemma}\label{EH2}
Let $F:\Bbb R^M\to \Bbb R$ be a smooth function and $v: \Bbb R^{d}\times [-T,T]\to \Bbb R^M$ be a smooth map.
Define
\begin{align*}
\widetilde{\beta}_k&=\sum_{|k'-k|\le 30}2^{\frac{d}{2}k'}\|P_{k'}v\|_{L^{\infty}_t{L^2_x}\bigcap L^{p_{d}}_{t,x}}\\
\widetilde{\alpha}_k&=\sum_{|k'-k|\le 30}2^{\frac{d}{2}k'}\|  P_{k'}(F(v))\|_{L^{\infty}_t{L^2_x}\bigcap L^{p_{d}}_{t,x}}.
\end{align*}
Assume that $\|v\|_{L^{\infty}_x}\lesssim 1$. Then
\begin{align}
&2^{\frac{d}{2}k}\|P_{k }F(v)(\partial_x v,\partial_x v)\|_{L^{\infty}_tL^2_x\bigcap L^{p_{d}}_{t,x}}\lesssim 2^{k}\widetilde{\beta}_k \sum_{k_1\le k}\widetilde{\beta}_{k_1}2^{k_1}+\sum_{k_2\ge k}2^{-d|k-k_2|}2^{2k_2}\widetilde{\beta}^2_{k_2}\nonumber\\
&+ \widetilde{\alpha}_{k}(\sum_{k_1\le k}2^{k_1}\widetilde{\beta}_{k_1})^2+\sum_{k_2\ge k}2^{d(k-k_2)}2^{k_2}\widetilde{\alpha}_{k_2}\widetilde{\beta}_{k_2}\sum_{k_1\le k_2}2^{k_1}\widetilde{\beta}_{k_1}.
\end{align}
\end{Lemma}

We also recall the general form of fractional Leibnitz rule (Kato-Ponce inequality), see \cite{Grafakos} and the reference therein.
\begin{Lemma}
Let $\frac{1}{2}<r< \infty$, and $1<p_1,q_1,p_2,q_2\le \infty $  with $\frac{1}{r}=\frac{1}{p_1}+\frac{1}{q_1}=\frac{1}{p_2}+\frac{1}{q_2}$. Given $s>\max(0,\frac{d}{r}-d)$ or $s\in 2\Bbb N$, for $ f,g\in \mathcal{S}(\Bbb R^d)$  we have
\begin{align}
\| |\nabla|^s(fg)\|_{ L^r_x}\lesssim \|g\|_{L^{p_1}_x}\||\nabla|^{s}f\|_{L^{q_1}_x}+\|f\|_{L^{p_2}_x}\||\nabla|^{s}g\|_{L^{q_2}_x}
\end{align}
\end{Lemma}

\section{Decay estimates of heat flows}

We have seen in Section 1.2 that the whole proof reduces to prove decay estimates of heat flows such as (\ref{ji00}), (\ref{ji2})-(\ref{ji3}),  (\ref{ji5})-(\ref{ji6}). The $L^{p_{d}}$ part especially (\ref{ji3}) requires more efforts than $L^{\infty}_tL^{2}_x$. Thus we leave them into later sections.

In this section, we prove decay estimates for quantities related to caloric gauge for small data heat flows from $\Bbb R^{d}$ with $d\ge 3$ in critical Sobolev spaces. This part is of independent interest and can be applied to other problems.

\subsection{Global evolution of Heat flows}

In this subsection we prove that $v$ is global and satisfies (\ref{zzFD}) stated in Theorem \ref{1XS}.
Suppose that the target manifold $\mathcal{M}$ is isometrically embedded into $\Bbb R^m$. Let  $\{S^{q}_{kj}\}$ denote the second fundamental form of the embedding $\mathcal{M}\hookrightarrow \Bbb R^{m}$. Then the heat flow equation is written as
\begin{align}
\partial_s v^{q}-\Delta_{\Bbb R^d} v^{q}=\sum^{d}_{i=1}S^{q}_{ab}\partial_{i}v^{a}\partial_{i}v^{b},
\end{align}
where $\{q,a,b\}$ run over $\{1,...,m\}$.

\begin{Lemma}\label{j1}
Let $d\ge 3$, and $L'\ge (d+200)2^{d+100}$ be a fixed integer. Assume that $v_0\in \mathcal{Q}(\Bbb R^d,\mathcal{M})$. Denote $\{\gamma_k(\sigma)\}$ the frequency envelope of $v_0$:
\begin{align}
\gamma_{j}(\sigma):=\sum_{j_1\in\Bbb Z}2^{-\delta|j-j_1|}2^{(\frac{d}{2}+\sigma )j_1}\|P_{j_1}v_0\|_{L^2_x},
\end{align}
and denote $\gamma_k(0)$ by $\gamma_k$.
There exists $0<\epsilon_1\ll 1$ depending only on $L,d$  such that if
\begin{align}\label{77}
\|v_0\|_{\dot{H}^{\frac{d}{2}}}\le\epsilon_1,
\end{align}
then the heat flow initiated from $v_0$ is global and satisfies
\begin{align}
\sup_{s\in[0,\infty)}2^{\frac{d}{2}k}(1+s^{\frac{1}{2}}2^{k})^{L'}\|P_{k}v\|_{L^2_x}&\lesssim_{L'}\gamma_k\label{3.4a}\\
\sup_{s\in[0,\infty)}2^{(\frac{d}{2}+\sigma)k}(1+s2^{2k})^{l}\|P_{k}v\|_{L^2_x}&\lesssim_{L'}\gamma_k(\sigma),\mbox{ }0\le l\le \frac{1}{2}(L'-1)\label{FuC}
\end{align}
provided that $\sigma\in [0,\vartheta]$, $s\ge 0$, $k\in\Bbb Z$.
Particularly, (\ref{zzFD}) and (\ref{1.8a})  hold.
\end{Lemma}

\begin{proof}
Let $\bar{s}>0$ be the maximal time such that for all $s\in[0,\bar{s})$, $0\le   j\in L'$, $1\le j'\le L'$, there holds
\begin{align}
s^{\frac{j}{2}}\|\partial^{j}_xv\|_{\dot{H}^{\frac{d}{2}}_{x}}&\le C\epsilon^{\frac{1}{2}}_1\label{FD2}\\
s^{\frac{j'}{2}}\|\partial^{j'}_xv\|_{L^{\infty}_{x}}&\le C\epsilon^{\frac{1}{2}}_2.\label{aFD2}
\end{align}
{\bf Step 1.}
We first verify that
\begin{align}\label{FD1}
\| \partial^{j}_x\left(S^{q}_{ab}(v)-S^{q}_{ab}(Q) \right)\|_{\dot{H}^{\frac{d}{2}}_{x}}\le C_j s^{-\frac{j}{2}}\epsilon^{\frac{1}{2}}_1.
\end{align}
If $d$ is even, (\ref{FD1}) follows by chain rules, Sobolev embedding and (\ref{FD2}), (\ref{aFD2}). The case when $d$ is odd requires slightly more efforts. Let $d=2d_0+1$ with $d_0\in\Bbb N$. Then by chain rule we have
\begin{align}\label{FD3}
\| \partial^{j}_x\left(S^{q}_{jl}(v)-S^{q}_{jl}(Q) \right)\|_{\dot{H}^{\frac{d}{2}}_{x}}\le
\sum_{0\le l, l'\le j+d_0}\sum_{|\alpha_1|+...+|\alpha_{l}|=j+d_0}
\|\mathcal{S}^{(l')}(v)\partial^{\alpha_{1}}_{x}v...\partial^{\alpha_{l}}_{x}v\|_{\dot{H}^{\frac{1}{2}}_{x}}
\end{align}
where we denote the $l'$ order derivatives of $\{S^{q}_{jl}\}$ by $\mathcal{S}^{(l')}(v)$ for simplicity. Then fractional Leibnitz formula shows
\begin{align}
&\|\mathcal{S}^{(l')}(v)\partial^{\alpha_{1}}_{x}v...\partial^{\alpha_{l}}_{x}v\|_{\dot{H}^{\frac{1}{2}}_{x}}
\nonumber\\
&\lesssim \|\partial^{\alpha_{1}}_{x}v...
\partial^{\alpha^{l}}_{x}v\|_{L^{r_1}_{x}}\|\mathcal{S}^{(l')}(v)\|_{\dot{B}^{\frac{1}{2},r_2}_{x}}
+\|\partial^{\alpha_{1}}_{x}v...
\partial^{\alpha_{l}}_{x}v\|_{\dot{H}^{\frac{1}{2}}_{x}}\|\mathcal{S}^{(l')}(v)\|_{L^{\infty}_{x}},\label{hbiu}
\end{align}
where $r_1,r_2\in(2,\infty)$ are taken as
\begin{align} \label{vfcr54}
\frac{1}{d}\left(\frac{d}{2}-\frac{1}{2}\right)=\frac{1}{2}-\frac{1}{r_2},\mbox{ }r_1=\frac{2d}{d-1}.
\end{align}  For $0< |\beta|\le \frac{d}{2}$, by Sobolev embedding
\begin{align}\label{iujn}
\||\nabla_{x}|^{\beta }v\|_{L^{\beta^*}_{x}}\lesssim \| v\|_{\dot{H}^{\frac{d}{2}}_{x}}, \mbox{ }\beta^*=\frac{d}{\beta}.
\end{align}
Similarly, for $\alpha=d_{0}+n$ with $n\in\Bbb N$, by Sobolev embedding
\begin{align}\label{oij}
\|\partial^{d_0+n}_{x}v\|_{L^{\frac{2d}{d-1}}_{x}}\lesssim \|\partial^{ n }_{x}v\|_{\dot{H}^{{\frac{d}{2}}}_{x}}.
\end{align}
{\bf Case 1.}
Assume that among all $\{\alpha_k\}^{l}_{k=1}$ there exists some $|\alpha_{k'}|\ge\frac{d}{2}=\frac{1}{2}+d_0$,  then by H\"older,  (\ref{aFD2}), (\ref{iujn}) and (\ref{oij}) we have
\begin{align}
\|\partial^{\alpha_{1}}_{x}v...
\partial^{\alpha_{l}}_{x}v\|_{L^{\frac{2d}{d-1}}_{x}}
&\lesssim  \left(\prod_{k\in\{1,...,l\}\setminus \{k'\}}\|\partial^{ \alpha_{k}}_{x}v\|_{L^{\infty}_x}\right)\|\partial^{\alpha_{k'}}_xv\|_{L^{\frac{2d}{d-1}}_{x}} \nonumber\\
&\lesssim \epsilon^{\frac{1}{2}} s^{-\frac{1}{2}
(\alpha_{1}+...+\alpha_{l}-\alpha_{k'})}\|\partial^{\alpha_{k'}-d_0}_xv\|_{{\dot{H}}^{\frac{d}{2}}_{x}}\nonumber\\
&\lesssim   \epsilon^{\frac{1}{2}}s^{-\frac{1}{2}(\alpha^{1}+...+\alpha_{l}-\alpha_{k'})}s^{-\frac{1}{2}(\alpha_{k'}-d_0)}.\label{2hbvg}
\end{align}
 Hence, (\ref{2hbvg}) gives
\begin{align}\label{3hbvg}
\|\partial^{\alpha_{1}}_{x}v...
\partial^{\alpha_{l}}_{x}v\|_{L^{\frac{2d}{d-1}}_{x}}\lesssim  \epsilon^{\frac{1}{2}}_1s^{-j}.
\end{align}
{\bf Case 2.} Assume that all $\{\alpha_k\}^{l}_{k=1}$ satisfy  $1\le  \alpha_{k} \le\frac{d}{2}$.  Then by H\"older,  (\ref{aFD2}), (\ref{iujn}) and interpolation we have
\begin{align}\label{6hbvg}
\|\partial^{\alpha_{k}}_{x}v\|_{L^{p}_{x}}\lesssim  \epsilon^{\frac{1}{2}}_1s^{-\frac{1}{2}\alpha^{k}(1-\frac{\alpha^*_{k}}{p})},
\end{align}
for all $p\in [\alpha^*_{k},\infty]$. Since by definition $\alpha^*_k=\frac{d}{\alpha_k}$, we get a tight form of (\ref{6hbvg}) as
\begin{align}\label{7hbvg}
\|\partial^{\alpha_{k}}_{x}v\|_{L^{p}_{x}}\lesssim  \epsilon^{\frac{1}{2}}_1s^{-\frac{1}{2}\alpha_{k}}s^{\frac{d}{2p}}.
\end{align}
Thus for $\sum^{l}_{k=1}\frac{1}{p_{k}}=\frac{1}{r_1}$, (recall $ r_1=\frac {2d}{d-1}$), we conclude
\begin{align}
\|\partial^{\alpha_{1}}_{x}v...
\partial^{\alpha_{l}}_{x}v\|_{L^{r_1}_{x}}&\lesssim \Pi_{1\le k\le l}\|\partial^{\alpha_{k}}_{x}v\|_{L^{p_k}}\lesssim \epsilon^{\frac{1}{2}}_1s^{-\frac{1}{2}\sum^{l}_{k=1}\alpha^{k}}s^{\sum^{l}_{k=1}\frac{d}{2p_{k}}}\nonumber\\
&\lesssim \epsilon^{\frac{1}{2}}_1s^{-\frac{1}{2}j}.\label{7hbvg}
\end{align}
Meanwhile, since $\mathcal{S}^{(l')}$ is Lipchitz, we have
\begin{align}\label{4hbvg}
\|\mathcal{S}^{(l')}(v)\|_{\dot{B}^{\frac{1}{2},r_2}_{x}}&\lesssim \|v\|_{\dot{B}^{\frac{1}{2},r_2}_{x}}\lesssim \|v\|_{\dot{H}^{\frac{d}{2}}_{x}},
\end{align}
where we used Sobolev embedding and (\ref{vfcr54}) in the last inequality. Therefore, (\ref{3hbvg}) and (\ref{7hbvg}) imply the first term in the RHS of  (\ref{hbiu}) is dominated by
$\epsilon^{\frac{1}{2}}_1s^{-j}$ up to  constants $C_{j}$.

Now we turn to the second term in the RHS of (\ref{hbiu}). As before, we consider two subcases. {\bf Case 1'.}  Assume that all $\{\alpha_k\}^{l}_{k=1}$ satisfy  $1\le \alpha_{k}\le\frac{d}{2}$.  And especially we have $\alpha_{k}+\frac{1}{2}\le\frac{d}{2}$ since $d$ is odd. Then, by  (\ref{7hbvg}) and fractional Leibnitz formula, we obtain for  $\sum^{l}_{k=1}\frac{1}{p_k}=\frac{1}{2}$ and $p_k\in [\alpha^{*}_k,\infty]$ (Recall $\alpha^*_{k}=\frac{d}{\alpha_k}$) that
\begin{align}
\|\partial^{\alpha_{1}}_{x}v...
\partial^{\alpha_{l}}_{x}v\|_{\dot{H}^{\frac{1}{2}}_{x}}&\lesssim  \epsilon^{\frac{1}{2}}_1\sum_{1\le i\le l}s^{-\frac{1}{2}(\alpha_{i}+\frac{1}{2})}s^{\frac{d}{2p_i}}
\Pi_{1\le k\le l,k\neq i}s^{-\frac{1}{2}\alpha_{k}}s^{\frac{d}{2p_k}}\nonumber\\
&\lesssim  \epsilon^{\frac{1}{2}}_1s^{-\frac{1}{2}j}.\label{9hbvg}
\end{align}
{\bf Case 2'.} Assume that among $\{\alpha_k\}^{l}_{k=1}$ there exists some $k'$ such that  $\alpha_{k'}\ge\frac{d}{2}$.  Then, by  (\ref{FD2}), (\ref{aFD2}) (\ref{iujn}), (\ref{oij}) and fractional Leibnitz formula, we obtain
\begin{align}
\|\partial^{\alpha_{1}}_{x}v...
\partial^{\alpha_{l}}_{x}v\|_{\dot{H}^{\frac{1}{2}}_{x}}&\lesssim
\|\partial^{\alpha_{k'}}_{x}v\|_{\dot{H}^{\frac{1}{2}}_{x}}
\Pi_{ k\in\{l,...,k\}\setminus\{k'\}}\|\partial^{\alpha_{l}}_{x}v\|_{L^{\infty}_{x}}\nonumber\\
&+
\|\partial^{\alpha_{k'}}_{x}v\|_{L^{2}_{x}}
\||\nabla|^{\alpha_{i}{+\frac{1}{2}}}_{x}v\|_{L^{\infty}_{x}}\Pi_{ k\in\{l,...,k\}\setminus\{k',i\}}\|\partial^{\alpha_{k}}_{x}v\|_{L^{\infty}_{x}}\nonumber\\
&\lesssim\epsilon^{\frac{1}{2}}_1s^{-\frac{1}{2}(\alpha_{k'}-\frac{d}{2}+\frac{1}{2})}
\Pi_{k\in\{l,...,k\}\setminus\{k'\}}
s^{-\frac{1}{2}\alpha_{k}}\nonumber\\
&+\epsilon^{\frac{1}{2}}_1s^{-\frac{1}{2}(\alpha_{k'}-\frac{d}{2})}s^{-\frac{1}{2}(\alpha_{i}+\frac{1}{2})}\Pi_{1\le k\le l,k\neq i,k'}s^{-\frac{1}{2}\alpha_{k}}\nonumber\\
&\lesssim  \epsilon^{\frac{1}{2}}_1s^{-\frac{1}{2}j}.\label{10hbvg}
\end{align}
Therefore, (\ref{10hbvg}) and (\ref{9hbvg}) show the second term in the RHS of (\ref{hbiu}) is dominated by
$\epsilon^{\frac{1}{2}}_1s^{-j}$ up to  constants $C_{j}$. And thus (\ref{FD1}) follows.

{\bf Step 2.}  Applying (\ref{FD1}) and following the lines of our previous paper [\cite{LIZE}, Proposition 3.2, Step 1], one obtains by Lemma \ref{EH} that (\ref{3.4a}) and (\ref{FuC}) holds within $s\in[0,\bar{s})$, i.e.
\begin{align}
\sup_{s\in[0,\bar{s})}2^{\frac{d}{2}k}(1+s^{\frac{1}{2}}2^{k})^{L'}\|P_{k}v\|_{L^2_x}&\lesssim_{L'}\gamma_k\label{FuC0}\\
\sup_{s\in[0,\bar{s})}2^{\frac{d}{2}k+\sigma k}(1+s^{\frac{1}{2}}2^{k})^{l}\|P_{k}v\|_{L^2_x}&\lesssim_{L'}\gamma_k(\sigma), \mbox{ }\sigma\in [0,\vartheta],0\le l\le L'-1\label{FuC00}
\end{align}
by bilinear Littlewood-Paley decomposition, provided that $\epsilon_1>0$ is sufficiently small depending on $L,d$. Since this part is routine, we leave the details for readers.

{\bf Step 3.1} (\ref{FuC0}) shows
\begin{align*}
2^{\frac{d}{2}k} (s^{\frac{1}{2}}2^{k})^{l}\|P_{k}v\|_{L^2_x}\lesssim_{L}\gamma_k, \mbox{ } \forall \mbox{ }0\le l\le L'.
\end{align*}
Then by Bernstein inequality we get
\begin{align*}
2^{\frac{d}{2}k}\|P_{k}(|\nabla|^{j}v)\|_{L^{2}_x}&\lesssim_{j} s^{-\frac{j}{2}}\gamma_k.
\end{align*}
and thus for any $0\le j\le L'$
\begin{align}\label{saut}
\|\partial^{j}_x v\|_{{\dot H}^{\frac{d}{2}}_x}&\lesssim_{j} s^{-\frac{j}{2}}\epsilon_1.
\end{align}
Then  (\ref{FD2})  holds with  $\epsilon^{\frac{1}{2}}_1$ replaced by $\epsilon_1$:

{\bf Step 3.2} Let us improve (\ref{aFD2}).
By Gagliardo-Nirenberg  inequality we get from Step 3.1 that
\begin{align}\label{dgkn}
\|\partial_x (v-Q)\|_{{L}^{\infty}_x}&\lesssim \| (v -Q)\|^{\frac{1}{2}}_{{\dot H}^{\frac{d}{2}}_x}\|\partial^{\frac{d}{2}+2}_x(v-Q)\|^{\frac{1}{2}}_{L^{2}_x}\lesssim \epsilon_1 s^{-\frac{1}{2}}.
\end{align}
In order to prove the following improved version of  (\ref{aFD2}),
\begin{align}\label{YhubnL}
\|\partial^{j}_x v\|_{L^{\infty}_x}&\lesssim_{j} s^{-\frac{j}{2}}\epsilon_1,\mbox{ }\forall 1\le j\le L',
\end{align}
it suffices to apply
\begin{align}
\|\partial_x v\|_{L^{\infty}_x}&\lesssim \epsilon_1 s^{-\frac{1}{2}}  \label{1Ggkn} \\
\|\partial^{l+1}_x v\|_{L^{\infty}_x}&\lesssim s^{-\frac{1}{2}}\|\partial^{l}_x v(\frac{s}{2})\|_{L^{\infty}_x}+\int^{s}_{\frac{s}{2}}(s-\tau)^{-\frac{1}{2}} \|\partial^{l+1}_x v\|_{L^{\infty}_x}\|\partial_x v\|_{L^{\infty}_x}d\tau\nonumber\\
& +\sum^{l}_{q=2}\sum_{l_1+...+l_q=l+2, l_1,...,l_q\le l}\int^{s}_{\frac{s}{2}}(s-\tau)^{-\frac{1}{2}}  \|\partial^{l_1}_x v\|_{L^{\infty}_x}...\|\partial^{l_q}_x v\|_{L^{\infty}_x}d\tau.\label{Ggkn}
\end{align}
In fact, (\ref{1Ggkn}) follows from (\ref{dgkn}), and (\ref{Ggkn}) follows by Duhamel principle and smoothing estimates $\|\partial_xe^{s\Delta}f\|_{L^{\infty}_x}\lesssim s^{-\frac{1}{2}}\|f\|_{L^{\infty}_x}$.
To obtain (\ref{YhubnL}), we begin with (\ref{1Ggkn}) and do iteration by (\ref{Ggkn}). The smallness of $s^{\frac{1}{2}}\|\partial_x v\|_{L^{\infty}_x}$ enables us to absorb the first term in the RHS of (\ref{Ggkn}) to the LHS, and the left terms in RHS of   (\ref{Ggkn})  are lower   derivative order terms.

Therefore, one obtains (\ref{FD2}), (\ref{aFD2}) hold with  $\epsilon^{\frac{1}{2}}_1$ replaced by $\epsilon_1$. Hence, $\bar{s}=\infty$.

Then (\ref{FuC0}) and (\ref{FuC00}) yield (\ref{3.4a}) and (\ref{FuC}) respectively, while (\ref{1.8a}) and (\ref{zzFD})  follow from (\ref{YhubnL}) and (\ref{saut}) respectively.

\end{proof}

\begin{Lemma}(Space-time estimates)\label{8xiao}
Let $v$ be the global heat flow in Lemma \ref{j1} with initial data $v_0\in \mathcal{Q}(\Bbb R^{d},\mathcal{M})$. Then we have
\begin{align}
\| \partial_xv\|_{L^{2}_s{\dot H}^{\frac{d}{2}}_x}&\lesssim \epsilon_1\label{1mkn}.
\end{align}
\end{Lemma}
\begin{proof}
The proof is based an energy argument and trilinear Littlewood-Paley decomposition.
By the heat flow equation,
\begin{align}\label{zxxxxxv7}
\frac{d}{ds}\| v\|^2_{ {\dot H}^{\frac{d}{2}}_x}=-\| \partial_xv\|^2_{ {\dot H}^{\frac{d}{2}}_x}+\langle|\nabla|^{\frac{d}{2}+1}v,|\nabla|^{\frac{d}{2}-1}[S(v)(\partial_xv,\partial_xv)]\rangle_{L^2_ x}
\end{align}
We have seen in Lemma \ref{j1} that
\begin{align}\label{nizzj}
 \| S(v)\|^2_{ {\dot H}^{\frac{d}{2}}_x}\lesssim \epsilon_1.
\end{align}
Let $\{\alpha_k\},\{\beta_k\}$ be defined in Lemma \ref{EH}, and define
\begin{align*}
\zeta_k(\sigma)&=\sup_{k_1\in\Bbb Z}2^{-\delta|k-k_1|}\sum_{|k'-k_1|\le 20}2^{\frac{d}{2}k'+\sigma k'}\|P_{k'}v\|_{L^{\infty}_t{L^2_x}}\\
\check{\alpha}_k &=\sup_{k'\in\Bbb Z}2^{-\delta|k-k'|}{\alpha_{k'}}
\end{align*}
Since frequency envelopes are of slow variation,
Lemma \ref{EH} shows
\begin{align*}
&\||\nabla|^{\frac{d}{2}}P_{k}[S(v)(\partial_xv,\partial_xv)]\|_{L^{\infty}_tL^2_x}
\lesssim \zeta_k(1) \left(\sum_{k_1\le k}\zeta_{k_1}(0)2^{k_1}\right)+\sum_{k_2\ge k}2^{-d|k-k_2|} \zeta_{k_2}(1)\zeta_{k_2}(0)\nonumber\\
&+ \check{\alpha}_{k}\left(\sum_{k_1\le k}2^{\frac{1}{2}k_1}\zeta_{k_1}(\frac{1}{2})\right)^2+\sum_{k_2\ge k}2^{d(k-k_2)}2^{\frac{1}{2}k_2}\breve{\alpha}_{k_2}\zeta_{k_2}(\frac{1}{2})\left[\sum_{k_1\le k_2}\zeta_{k_1}(\frac{1}{2})2^{\frac{1}{2}k_1}\right]\\
&\lesssim  2^{k}\zeta_k(1)\zeta_{k_1}(0) + \check{\alpha}_{k} 2^{k}\zeta_{k}(\frac{1}{2})^2.
\end{align*}
Therefore, by (\ref{nizzj}) and Lemma \ref{j1}, we get from $\zeta_k(1/2)\le \sqrt{\zeta_k(0)\zeta_k(1)}$ that
\begin{align*}
\sum_{k\in\Bbb Z}\||\nabla|^{\frac{d}{2}-1}P_{k}[S(v)(\partial_xv,\partial_xv)]\|^2_{L^{\infty}_tL^2_x}
&\lesssim \sum_{k\in\Bbb Z}\epsilon\|\zeta_k(1)\|^2+\epsilon\|\zeta_{k}(\frac{1}{2})\|^4\\
&\lesssim \epsilon_1\|\partial_xv\|^2_{L^{\infty}_t{\dot H}^{\frac{d}{2}}_x}.
\end{align*}
Thus (\ref{zxxxxxv7}) reduces to
\begin{align}\label{zvzB7}
\frac{d}{ds}\| v\|^2_{ {\dot H}^{\frac{d}{2}}_x}+(1-\epsilon_1)\| \partial_xv\|^2_{ {\dot H}^{\frac{d}{2}}_x}\le0
\end{align}
Integrating (\ref{zvzB7}) in $s\in [0,\infty)$ yields (\ref{1mkn}) since $\|v\|_{\dot{H}^{\frac{d}{2}}}\lesssim \epsilon_1$ by Lemma \ref{j1}.
\end{proof}

\begin{Corollary}
Let $v$ be the global heat flow in Lemma \ref{j1} with initial data $v_0\in \mathcal{Q}(\Bbb R^{d},\mathcal{M})$. Then for all $0\le a\le L'-1$,  $0\le b\le L'-2$and $0\le j\le [\frac{d}{2}-1]$, there holds
\begin{align}
\left\|\nabla^{a}_x\partial_xv(s)\right\|_{L^{\infty}_x}&\lesssim \epsilon_1 {s}^{-\frac{a+1}{{2}}}\label{70}\\
\left\|\nabla^{b}_x\partial_sv(s)\right\|_{L^{\infty}_x}&\lesssim \epsilon_1 {s}^{-\frac{a+2}{{2}}}\label{79}\\
\left\|\nabla^{j}_x\partial_xv(s)\right\|_{L^{{d}/({1+j})}_x}&\lesssim \epsilon_1 \label{270}
\end{align}
Moreover, if $ \frac{d}{2}-1\le k\le L'-1$, $p\in[2,\infty]$, we have
\begin{align}\label{370}
\left\|\nabla^{k}_x\partial_xv(s)\right\|_{L^{p}_x}&\lesssim \epsilon_1s^{-\frac{k+1}{2}+\frac{d}{2p}}
\end{align}
\end{Corollary}
\begin{proof}
Basic theories of embedded sub-manifolds show the following inequality
\begin{align}\label{p2}
\left|{\nabla^{a}_x\partial_{i}v}\right|\lesssim \sum^{a+1}_{j=1}\sum _{\sum^{j}_{l}\beta_l=a+1,{\beta_l}\in\Bbb Z_+}|\partial^{\beta_1}_xv|...|\partial^{\beta_j}_xv|.
\end{align}
Then (\ref{70}) follows by (\ref{1.8a}). (\ref{79}) follows by (\ref{70}) and the identity $\partial_s v=\sum^{d}_{i=1}\nabla_j\partial_jv$. And (\ref{270})
follows by Sobolev embedding inequalities and H\"older inequalities. Lastly, we prove (\ref{370}) by interpolating (\ref{70}) with
\begin{align}\label{pVB2}
\left\|{\nabla^{a}_x\partial_{i}v}\right\|_{L^2_x}\lesssim \epsilon_1s^{-\frac{k+1}{2}+\frac{d}{4}}.
\end{align}
In order to prove (\ref{pVB2}), we consider two subcases: {\emph{Case 1.}} All $\{\beta_l\}^{j}_{l=1}$ in
(\ref{p2}) satisfy $\beta_l<\frac{d}{2}-1$; {\emph{Case 2.}} There exists some $1\le l_*\le j$ such that $\beta_{l_*}\ge \frac{d}{2}-1$. Then (\ref{pVB2}) follows as Step 1 of Lemma \ref{j1}.

\end{proof}

\subsection{Non-critical theory for heat flows}

This subsection involves some estimates which depend  on both $\||\nabla|^{d/2}v_0\|_{L^2_x}$ and $\|dv_0\|_{L^2_x}.$ Thus all theses estimates are not in the critical level. But they are necessary for setting up our bootstrap in the next subsection. Most of the techniques in this subsection are classical and we present them in detail just for reader's convenience.

\begin{Lemma}\label{Q}
Let $v$ be the global heat flow in Lemma \ref{j1} with initial data $v_0\in \mathcal{Q}(\Bbb R^{d},\mathcal{M})$.
Then the heat flow $v$ will uniformly converge to $Q$ as $s\to\infty$.
\end{Lemma}
\begin{proof}
The Bochner-Weitzenb\"ock identity for $|\partial_s v|^2$ is
\begin{align}
(\partial_s-\Delta) |\partial_s v|^2+2|\nabla \partial_s v|^2=\sum^{d}_{i=1}\langle{\bf{R}}(\partial_s v,\partial_iv)\partial_sv,\partial_i v\rangle.
\end{align}
And we claim that
\begin{align}\label{npk}
\|\partial_s v\|_{L^2_x}\lesssim s^{-\frac{1}{2}}\|dv_0\|_{L^2_x}.
\end{align}
Then, by smoothing effect of heat equations and(\ref{zzFD}), we get
\begin{align}
\|\partial_s v(s)\|_{L^{\infty}_x}&\lesssim s^{-\frac{d}{2}}\left\|\partial_s v(\frac{s}{2})\right\|^2_{L^{2}_x} +\int^{s}_{\frac{s}{2}}\|dv\|^2_{L^{\infty}_x}\|\partial_s v\|^2_{L^{2}_x}\tau^{-\frac{d}{2}}d\tau\nonumber\\
&\lesssim s^{-\frac{d}{2}-1}\|\nabla v_0\|^2_{L^{2}_x}.\label{nop}
\end{align}
Hence, we conclude
\begin{align}\label{gcfdrer}
\left\|v(s_1,\cdot)-v(s_2,\cdot)\right\|_{L^{\infty}_x}&\le \int^{s_2}_{s_1}\|\partial_s v(s,\cdot)\|_{L^{\infty}_x}ds\lesssim  {s_1}^{-\frac{{d+2}}{{4}}+1}.
\end{align}
which implies $v$ converges uniformly as $s\to\infty$ since $d\ge 3$. Denote the limit map of $v$ by $\Theta:\Bbb R^d\to \mathcal{M}$. Then by $\|dv\|_{L^{\infty}_x}\lesssim s^{-\frac{1}{2}}$, we have $\Theta$ is a constant map.
(\ref{gcfdrer}) now reads as
\begin{align} \label{xhgv78}
\sup_{x\in\Bbb R^d}\left|v(s,x)-\Theta\right|\lesssim s^{-\frac{2+d}{4}+1}\|dv_0\|_{L^2_x}
\end{align}
Since $v\in \mathcal{Q}(\Bbb R^d,\mathcal{M})$ implies $\lim_{|x|\to\infty} v=Q$, (\ref{xhgv78}) shows $\Theta=Q$  by contradiction argument.

Therefore, it suffices to verify the claim (\ref{npk}).  By Duhamel principle and smoothing effect of linear heat equation,
\begin{align*}
\|\Delta v(s)\|_{L^{2}_x}&\lesssim \|e^{\frac{s}{2}\Delta}\nabla v(\frac{s}{2})\|_{L^2_x}+\int^{s}_{\frac{s}{2}} \|e^{(s-\tau)\Delta}(\nabla(S(v)|\nabla v|^2)(\tau))\|_{L^2_x}d\tau\\
&\lesssim s^{-\frac{1}{2}}\|\nabla v_0\|_{L^2_x}+\int^{s}_{\frac{s}{2}}(s-\tau)^{-\frac{1}{2}}\left(\|\nabla v\|^2_{L^{\infty}_x}\|\nabla v\|_{L^2_x}+\|\Delta v\|_{L^2_x}\|\nabla v\|_{L^{\infty}_x}\right)\\
&\lesssim s^{-\frac{1}{2}}\|\nabla v_0\|_{L^2_x}+\int^{s}_{\frac{s}{2}}(s-\tau)^{-\frac{1}{2}}\|\nabla v_0\|_{L^2_x}(\epsilon^2_1\tau^{-1}+\epsilon_1\tau^{-\frac{1}{2}}\|\Delta v\|_{L^2_x})d\tau\\
&\lesssim s^{-\frac{1}{2}}\|\nabla v_0\|_{L^2_x}+\epsilon_1\int^{s}_{\frac{s}{2}}(s-\tau)^{-\frac{1}{2}} \tau^{-\frac{1}{2}}\|\Delta v\|_{L^2_x}d\tau,
\end{align*}
where in the third line we applied
\begin{align*}
\|\nabla v\|_{L^2_x}&\lesssim \|v_0\|_{L^2_x}\\
\|\nabla v\|_{L^{\infty}_x}&\lesssim \epsilon_1 s^{-\frac{1}{2}}.
\end{align*}
Let $X(s)=\sup_{\tilde{s}\in[0,s]}\tilde{s}^{\frac{1}{2}}\|\Delta v(\tilde{s})\|_{L^2_x}$, thus
\begin{align*}
X(s)\lesssim \|d v_0\|_{L^2_x}+\epsilon_1 X(x),
\end{align*}
which shows
\begin{align*}
\|\Delta v\|_{L^2_x}\lesssim s^{-\frac{1}{2}}\|d v_0\|_{L^2_x}.
\end{align*}
And thus
\begin{align*}
\|\partial_sv\|_{L^2_x}\lesssim \|\Delta v\|_{L^2_x} +\|\nabla v\|_{L^{\infty}_x}\|\nabla v\|_{L^2_x}\lesssim s^{-\frac{1}{2}}\|d v_0\|_{L^2_x},
\end{align*}
from which (\ref{npk}) follows. So the proof has been completed.
\end{proof}

The proof of Lemma \ref{Q} shows $\|\Delta v\|_{L^{\infty}_x}$ indeed decays faster than that stated in Lemma \ref{j1} if one takes $\|\nabla v_0\|_{L^2_x}$ into consideration. These faster rates will be useful in the set up of bootstrap. And in fact decay estimates of higher order derivatives of $v$ can be obtained similarly by induction.

\begin{Lemma}\label{Q}
Let $v$ be the global heat flow in Lemma \ref{j1} with initial data $v_0\in \mathcal{Q}(\Bbb R^{d};\mathcal{M})$. Then for all $0\le j\le L' -1$, $0\le l\le L'  -2$, one has
\begin{align}
\|\nabla^{j}_x\partial_x v\|_{L^2_x}&\lesssim s^{-\frac{j}{2}}\|dv_0\|_{L^2_x}\label{08}\\
\|\nabla^{j}_x\partial_x v\|_{L^{\infty}_x}&\lesssim s^{-\frac{2j+d}{4}}\|dv_0\|_{L^2_x}\label{081}\\
\|s^{\frac{1}{2}{(j-1)}}\nabla^{L}_x\partial_x v\|_{L^{2}_sL^{2}_x}&\lesssim \|dv_0\|_{L^2_x}.\label{082}\\
\|\nabla^{l}_x\partial_s v\|_{L^2_x}&\lesssim s^{-\frac{l+1}{2}}\|dv_0\|_{L^2_x}\label{08s}\\
\|\nabla^{l}_x\partial_s v\|_{L^{\infty}_x}&\lesssim s^{-\frac{l+1}{2}-\frac{d}{4}}\|dv_0\|_{L^2_x}.\label{081s}
\end{align}
Moreover, let $\{e_l\}^{m}_{l=1}$ be an orthonormal frame for the pullback bundle $v^*T\mathcal{M}$ and $\{\psi_i\}^{d}_{i=1},\psi_s$ be the sections of trivial bundle $[0,\infty)\times \Bbb R^d$ with fiver $\Bbb R^{m}$ induced by $v$ via:
\begin{align*}
\psi^{l}_i:=\langle \partial_i v,e_l\rangle,\mbox{  }\psi^{l}_s:=\langle \partial_s v,e_l\rangle.
\end{align*}
Denote $\{D_i, D_s\}^{d}_{i=1}$ the induced covariant derivatives on the bundle $([0,\infty)\times \Bbb R^d, \Bbb R^{m})$.
Then  for any  $0\le j\le L'-1$, $0\le l\le L'-2$, we also have
\begin{align}
\|D^{j}_x\psi_x\|_{L^{2}_x}&\lesssim s^{-\frac{j}{2}}\|dv_0\|_{L^2_x}\label{09}\\
\|D^{j}_x\psi_x\|_{L^{\infty}_x}&\lesssim s^{-\frac{2j+d}{4}}\|dv_0\|_{L^2_x},\label{091}\\
\|s^{\frac{1}{2}{(j-1)}}D^{j}_x\psi_x\|_{L^{2}_sL^{2}_x}&\lesssim \|dv_0\|_{L^2_x}\label{092}\\
\|D^{l}_x\psi_s\|_{L^{2}_x}&\lesssim s^{-\frac{l+1}{2}}\|dv_0\|_{L^2_x}\label{09s}\\
\|D^{l}_x\psi_s\|_{L^{\infty}_x}&\lesssim s^{-\frac{l+1}{2}-\frac{d}{4}}\|dv_0\|_{L^2_x},\label{092s}\\
\|D^{l}_x\psi_s\|_{L^{\infty}_x}&\lesssim \epsilon s^{\frac{l+2}{2}}\label{s00}\\
\|D^{l}_x\psi_x\|_{L^{\infty}_x}&\lesssim \epsilon s^{\frac{l+1}{2}}\label{x00}.
\end{align}
where the simplified notations $\psi_x$, $D^{j}_x$ refer to differential fields $\{\psi_i\}^{d}_{i=1}$ and  respectively  various combinations of $\{D_i\}^{d}_{i=1}$ of order $j$.
\end{Lemma}
\begin{proof}
The proof is based on well-known techniques. We sketch it for reader's convenience.
(\ref{08s})-(\ref{081s}) follow from (\ref{08})-(\ref{081}) by the identity $\partial_sv=\sum^{d}_{i=1}\nabla_i\partial_iv.$
(\ref{09}) and (\ref{091}) follow from (\ref{08})-(\ref{081}) since  $|\nabla^{j}_x \partial_xv| $ controls $|D^{j}_x \psi_x |$ point-wisely. And by the same reason, (\ref{s00}), (\ref{x00}) follow by  (\ref{70}), (\ref{79}). Meanwhile, (\ref{09s}), (\ref{092s})  follow from  (\ref{09}), (\ref{091}).

Therefore, it suffices to prove (\ref{08})-(\ref{082}).
We denote
\begin{align*}
X_{j,\infty}(s)&:=\sup_{\tilde{s}\in[0,s]}
\tilde{s}^{\frac{d+2j}{4}}\|\nabla^{j}_x\partial_x v(\tilde{s})\|_{L^{\infty}_x}\\
X_{j,2}(s)&:=\sup_{\tilde{s}\in[0,s]}
\tilde{s}^{\frac{j}{2}}\|\nabla^{j}_x\partial_x v(\tilde{s})\|_{L^{2}_x}\\
Y_{j,2}(s)&:=(\int^s_0\int_{\Bbb R^d}{\tilde{s}}^{{j-1}}|\nabla^{j}_x\partial_x v(\tilde{s})|^2dxd\tilde{s})^{\frac{1}{2}}.
\end{align*}
Recall the Bochner inequality (see e.g. \cite{huangSmithhuang,huangTaohuang}):
\begin{align}\label{hgyu}
(\partial_s -\Delta)|\nabla^{j}_x\partial_x v|^2+2|\nabla^{j+1}_x\partial_x v|^2\lesssim \sum^{j+3}_{z=3}\sum_{(1+n_{1})+...+(1+n_{z})=j+3}|\nabla^{n_1}_x\partial_{x}v|...|\nabla^{n_z}_x\partial_{x}v||\nabla^{j}_x\partial_x v|.
\end{align}
We notice that the RHS of (\ref{hgyu}) can be further expanded as
\begin{align}
&(\partial_s -\Delta)|\nabla^{j}_x\partial_x v|^2+2|\nabla^{j+1}_x\partial_x v|^2\nonumber\\
&\lesssim |dv|^2|\nabla^{j}_x\partial_x v |^2+ \sum^{j+3}_{z=3}\sum_{\sum^{z}_{i=1}(1+n_{i})=j+3, \forall i, |n_i|<j}|\nabla^{n_1}_x\partial_{x}v|...|\nabla^{n_z}_x\partial_{x}v||\nabla^{j}_x\partial_x v |.\label{xiao9}
\end{align}
Then it is easy to see
\begin{align*}
&X^2_{j,2}(s)+2Y^2_{j+1,2}(s)\\
&\lesssim jY^2_{j,2}(s)+\epsilon_1Y^{2}_{j,2}(s)\\
&+\sum^{j+3}_{z=3}\sum_{\sum^{z}_{i=1}(1+n_{i})=j+3, \forall i, |n_i|<j}\int^{s}_0\tilde{s}^{{j}}\|\nabla^{j}_x\partial_x v|\|_{L^2_x}\|\nabla^{n_1}_x\partial_{x}v\|_{L^{2}_x}\|\cdot\|_{L^{\infty}_x}...
\|\nabla^{n_z}_x\partial_{x}v\|_{L^{\infty}_x}d\tilde{s}\nonumber\\
&\lesssim  Y^2_{j,2}(s)+\epsilon_1\sum^{j-1}_{a=0}Y_{j,2}(s)Y_{a,2}(s)\nonumber
\end{align*}
where we used (\ref{zzFD}) in the first line and (\ref{70}) in the last line.
Thus we have seen
\begin{align}
Y_{l,2}(s)\lesssim \|dv_0\|_{L^2_x}, \forall 1\le l\le j&\Longrightarrow X_{j,2}(s)\lesssim \|dv_0\|_{L^2_x}\\
X_{j,2}(s)+Y_{l,2}(s)\lesssim \|dv_0\|_{L^2_x}, \forall 1\le l\le j&\Longrightarrow Y_{j+1,2}(s)\lesssim \|dv_0\|_{L^2_x}.
\end{align}
These two induction relations show that  for (\ref{08}), (\ref{082}) it suffices to verify  $Y_{1,2}(s)+X_{0,2}(s)\lesssim \|dv_0\|_{L^2_x}$. By the energy identity we see
\begin{align}\label{nzzzbv}
\int^{s}_0\|\tau(v)\|^2_{L^2_x}ds'+X_{0,2}(s)\le \|dv_0\|_{L^2_x}.
\end{align}
Integration by parts gives
\begin{align*}
\|\nabla dv\|^2_{L^2_x}\lesssim \|\tau(v)\|^2_{L^2_x}+\|dv \|^{4}_{L^4_x}.
\end{align*}
By Gagliardo-Nirenberg inequality
\begin{align*}
\|dv \|^{4}_{L^4_x}\lesssim \|dv\|^2_{L^2_x}\|dv\|^2_{\dot{H}^{\frac{d}{2}}_x}.
\end{align*}
Then (\ref{nzzzbv}), (\ref{1mkn}) yield
\begin{align*}
Y_{1,2}(s)+X_{0,2}(s)\lesssim \|dv_0\|_{L^2_x}.
\end{align*}
Thus  (\ref{08}), (\ref{082}) are done.

It remains to prove (\ref{081}).
(\ref{xiao9}) and Kato's inequality show
\begin{align*}
(\partial_s -\Delta)|\nabla^{j}_x\partial_x v|\lesssim \|dv\|^2_{L^{\infty}_x}|\nabla^{j}_x\partial_x v |+
 \sum^{j+3}_{z=3}\sum_{\sum^{z}_{i=1}(1+n_{i})=j+3, \forall i, |n_i|<j}|\nabla^{n_1}_x\partial_{x}v|...|\nabla^{n_z}_x\partial_{x}v|.
\end{align*}
We prove  (\ref{081})  by induction.
Suppose that  (\ref{081}) hold for all $j'<j$. Then by Duhamel principle and smoothing effect of heat equation one has
\begin{align*}
&\|\nabla^{j}_x\partial_x v(s)\|_{L^{\infty}_x}\\
&\lesssim s^{-\frac{d}{4}}\|\nabla^{j}_x\partial_x v(s/2)\|_{L^{2}_x}
+\int^{s}_{s/2}\|dv\|^2_{L^{\infty}_x}\|\nabla^{j}_x\partial_x v\|_{L^{\infty}_x}d\tau\\
&+\sum^{j+3}_{z=3}\sum_{\sum^{z}_{i=1}(1+n_{i})=j+3, \forall i, |n_i|<j}\int^{s}_{s/2}\|\nabla^{n_1}_x\partial_{x}v\|_{L^{\infty}_x}...\|\nabla^{n_z}_x\partial_{x}v\|_{L^{\infty}_x}.  d\tau\\
&\lesssim s^{-\frac{d}{4}}\|\nabla^{j}_x\partial_x v(s/2)\|_{L^{2}_x}
+s^{-\frac{d+2j}{4}}(\int^{s}_{s/2}\|dv\|^2_{L^{\infty}_x}d\tau)X_{j,\infty}(s)\\
&+\epsilon^2_1\|dv_0\|_{L^2_x}\sum^{j+3}_{z=3}\sum_{\sum^{z}_{i=1}(1+n_{i})=j+3, \forall i, |n_i|<j}\int^{s}_{s/2}\tau^{-\frac{d+2n_1}{4}}\tau^{-\frac{n_2+1}{2}}...\tau^{-\frac{n_z+1}{2}} d\tau,
\end{align*}
where in the last line we applied induction assumption  to $\|\nabla^{n_1}_x\partial_x v\|_{L^{\infty}_x}$ and  (\ref{70}) to $\|\nabla^{n_i}_x\partial_x v\|_{L^{\infty}_x}$, $i=2,...,z$.
Thus by (\ref{zzFD}), $X_{j,\infty}(s)$ satisfies
\begin{align*}
X_{j,\infty}(s)\lesssim X_{j,2}(s)+\epsilon_1X_{j,\infty}(s)+\|dv_0\|_{L^2_x}.
\end{align*}
Hence (\ref{08}) shows $X_{j,\infty}(s)\lesssim \|dv_0\|_{L^2_x}$ and thereby our lemma follows.

\end{proof}

\begin{Lemma}\label{QQ}
Let $v$ be the global heat flow in Lemma \ref{j1} with initial data $v_0\in \mathcal{Q}(\Bbb R^{d},\mathcal{M})$.
Given limit orthonormal frames $\{e^{\infty}_l\}^{m}_{l=1}$, there exists a unique gauge $\{e_{l}\}^{m}_{l=1}$ for $v^{*}T\mathcal{M}$ such that
\begin{align}
\nabla_s e_{l}&=0 \label{hebei1}\\
\lim_{s\to\infty}e_{l}&=e^{\infty}_l, \mbox{  }\forall l=1,...,m.\label{hebei2}
\end{align}
Frames satisfying (\ref{hebei1}), (\ref{hebei2}) are called Tao's caloric gauge.
Moreover, the connection coefficients $A_x$ satisfy
\begin{align}\label{1poke}
A_i=\int^{\infty}_s\mathcal{R}(\psi_s,\psi_i)ds'
\end{align}
and the estimates
\begin{align}\label{xvb67}
\|\partial^{l}_x A_x(s)\|_{L^2_x}&\lesssim s^{-\frac{l}{2}-\frac{d}{4}+\frac{1}{2}}\|dv_0\|_{L^2_x},
\end{align}
for any $s\ge 1$, $0\le l\le L'-3+[\frac{d}{2}]$. And the frames $\{e_l\}^{m}_{l=1}$ satisfy
\begin{align}\label{fcav}
\|\partial^{j}_x(d\mathcal{P}(e_{l})-d\mathcal{P}(e^{\infty}_{l}))\|_{{\dot H}^{\frac{d}{2}}_x}&\lesssim s^{-\frac{2(j-1)+d}{4}}\|dv_0\|_{L^2_x}.
\end{align}
for any $s\ge 1$, $0\le j\le L'-3$.
\end{Lemma}
\begin{proof}
The existence of caloric gauge follows by the standard line: (i) Take arbitrary $\{\tilde{e}_l\}^{m}_{l=1}$
as the initial data of (\ref{hebei1}); (ii) Suppose that the solution to (\ref{hebei1}) with initial data $\{\tilde{e}_l\}^{m}_{l=1}$ is $\{\widetilde{e}_l(s,x)\}^{m}_{l=1}$. Prove that $d\mathcal{P}\widetilde{e}_l(s,x)$ converges uniformly to some $d\mathcal{P}\widetilde{e}^{\infty}_l(x)$ as $s\to\infty $; (iii) Apply an $s$ independent gauge transformation $\Lambda(x)\in O(m)$ to $\{\widetilde{e}^{\infty}_l(x)\}$ such that $\Lambda(x)\widetilde{e}^{\infty}_l(x)=e^{\infty}_l$. Then $\{\Lambda(x)\widetilde{e}_l(s,x)\}^{m}_{l=1}$ is the desired caloric gauge satisfying (\ref{hebei1}), (\ref{hebei2}).

Therefore, to prove the existence of caloric gauge, it suffices to prove the convergence in Step ({ii}).
The uniqueness is standard by the boundary condition  (\ref{hebei2}).

For the simplicity of notations, we denote $\{e_l\}$ instead of $\{\widetilde{e}_l\}$.
By caloric condition $\nabla_s e_l=0$, one has
\begin{align*}
\partial_sd\mathcal{P}(e_{l})=({\bf D}d\mathcal{P})(\partial_sv; e_{l}).
\end{align*}
where ${\bf D}$ denotes the induced connection on the bundle $\mathcal{P}^*T\Bbb R^{m}$. Thus (\ref{081s}) shows
\begin{align*}
\|d\mathcal{P}(e_{l})(s_2)-d\mathcal{P}(e_{l})(s_1)\|_{L^{\infty}_x}&\lesssim \int^{s_2}_{s_1}\|\partial_s v\|_{L^{\infty}_x}ds\lesssim \|dv_0\|_{L^2_x}\int^{s_2}_{s_1}s^{-\frac{2+d}{4}}ds\\
&\lesssim s^{-\frac{d}{4}+\frac{1}{2}}_1,
\end{align*}
which shows $d\mathcal{P}(e_{l})$ converges uniformly in $\Bbb R^d$ as $s\to\infty$. Denote the limit of
$d\mathcal{P}(e_{l})$ by $\chi^{\infty}_l$. Hence, the convergence in Step ({ii}) has been verified. And
\begin{align}\label{zijkb}
\chi^{\infty}_l=\lim_{s\to\infty}d\mathcal{P}(e_{l})(s,x)=\lim_{s\to\infty}d\mathcal{P}(e^{\infty}_{l}(Q))
\end{align}
is constant in $x$. In the rest we prove (\ref{1poke})-(\ref{fcav}).

Since $\partial_sA_i=\mathcal{R}(\psi_s,\psi_i)$, one has for $s_2>s_1\ge 1$
\begin{align*}
&\|\partial^{j}_x\left(A_x(s_2)-A_x(s_1)\right)\|_{L^2_x}\\
&\lesssim \int^{s_2}_{s_1}\|\partial^{j}_x\mathcal{R}(\psi_s,\psi_x)\|_{L^2_x}ds\\
&\lesssim \sum^{j+1}_{z=1}\sum_{j_0+j_1+...+j_z=j}\int^{s_2}_{s_1}\|D^{j_0}_x\psi_s D^{j_1}_x\psi_x...D^{j_z}_x\psi_x\|_{L^2_x}ds\\
&\lesssim \|dv_0\|_{L^2_x}\sum^{j+1}_{z=1}\sum_{j_0+j_1+...+j_z=j}
\int^{s_2}_{s_1}s^{-\frac{j_0+1}{2}}s^{-\frac{j_1}{2}-\frac{d}{4}}
s^{-\frac{j_2}{2}}...s^{-\frac{j_z}{2}}ds\\
&\lesssim s^{-\frac{j}{2}-\frac{d}{4}+\frac{1}{2}}_1\|dv_0\|_{L^2_x}
\end{align*} where in the forth line we applied (\ref{09s}) and the bounds (\ref{x00}), (\ref{s00}).
Thus $A_x(s)$ converges in $H^{k}$ for all $0\le k\le L'-2$ as $s\to\infty$. Denote $A^{\infty}_x$ the limit of $\lim_{s\to\infty}A_x(s,x)$. Then we summarize that
\begin{align}\label{poke}
&\|\partial^{j}_x\left(A_x(s)-A^{\infty}_x\right)\|_{L^2_x}\lesssim s^{-\frac{j}{2}-\frac{d}{4}+\frac{1}{2}}\|dv_0\|_{L^2_x}
\end{align}
for $s\ge 1$.

And by the identity $|\partial_x\partial_sd\mathcal{P}(e_{l})(s_2)|\lesssim |\nabla_x\partial_s v|+ |\partial_s v||A_x|+|\partial_x v||\partial_x v|$, we obtain
\begin{align*}
\|\partial_x[d\mathcal{P}(e_{l})(s_2)-d\mathcal{P}(e_{l})(s_1)]\|_{L^{\infty}_x}&\lesssim  \|dv_0\|_{L^2_x} s^{-\frac{d}{4}+\frac{1}{2}}_1,
\end{align*}
for $s_2\ge s_1\ge 1$, where we applied the bound $\|A_x(s)\|_{L^{\infty}_x}\lesssim 1$ for $s\ge 1$  given by  (\ref{poke}). Thus combining this with (\ref{zijkb}), we see $ d\mathcal{P}(e_l)$ converges to the constant vector  $\chi_{l}$ in $C^1$ topology. Particularly, there holds
\begin{align}\label{zxzijkb}
\lim\limits_{s\to \infty} \partial_xd\mathcal{P}(e_l)=0,\mbox{ } \forall \mbox{ }l=1,...,m.
\end{align}

To prove (\ref{1poke}), it suffices to verify
\begin{align}\label{0k}
A^{\infty}_x=0.
\end{align}
And  (\ref{poke}) gives  (\ref{xvb67})
if we have shown $A^{\infty}_x=0$. Hence it only remains to check (\ref{0k}). By the identity
\begin{align*}
\partial_i d\mathcal{P}(e_l)=d\mathcal{P}(\nabla_i e_{l})+({\bf D}d\mathcal{P})(\partial_iv;  e_{l}),
\end{align*}
and the isometry of $d\mathcal{P}$, we see
\begin{align}\label{pozijk}
|\nabla_i e_{l}|\lesssim |\partial_i d\mathcal{P}(e_l)|+|\partial_iv|
\end{align}
By (\ref{zxzijkb}), $ |\partial_i d\mathcal{P}(e_l)|\to 0$ as $s\to\infty$. Meanwhile, $|\partial_iv|\to 0$ as $s\to\infty$ by Lemma \ref{j1}. Thus (\ref{pozijk}) shows
\begin{align*}
\lim_{s\to\infty}|\nabla_i e_{l}|=0.
\end{align*}
Thus $A^{\infty}_{x}=0$, and (\ref{1poke}) follows. And then (\ref{xvb67}) follows by (\ref{poke}).

Point-wisely one has for $k\ge 1$,
\begin{align*}
|\partial^{k}_x(d\mathcal{P}(e_{l}))|\lesssim \sum^{k}_{z,z'=1} \sum\limits_{\sum^{z}_{l=1}(i_l+1)+\sum^{z'}_{p=1}(j_p+1)=k}  |\nabla^{i_1}_x\partial_x v| ...|\nabla^{i_z}_x\partial_x v||\partial^{j_1}_xA_x|...|\partial^{j_{z'}}_{x}A_x|.
\end{align*}
Then (\ref{xvb67}) and the decay of $|\nabla^{i}_x\partial_x v|$ previously obtained give (\ref{fcav}). Notice that when $d$ is odd we use Gagliardo-Nirenberg  inequality  to prove (\ref{fcav}).
Now,  the whole proof is completed.
\end{proof}

\section{Proof of Proposition 1.1}

Let $L'$ be the given integer in Section 3.
Set $L=L'-5$.

We begin with a simple $L^{\infty}$ bound Lemma for connections and frames.
\begin{Lemma}
For caloric gauge in Lemma \ref{QQ}, the connection coefficients  $A_x$ and frames $\{e_{l}\}$ satisfy
\begin{align}
\|\partial^{j}_xA_{x}\|_{L^{\infty}_x}&\lesssim_{L} \epsilon_1  s^{-\frac{j+1}{2}}\label{X2}\\
\|\partial^{j}_x(d\mathcal{P}(e)-\chi^{\infty})\|_{L^{\infty}_x}&\lesssim \epsilon_1 s^{-\frac{j}{2}}\label{XX2}
\end{align}
for any $s>0$, $0\le j\le L$
\end{Lemma}
\begin{proof}
By (\ref{x00})-(\ref{s00}), (\ref{X2}) follows by direct calculations as (\ref{xvb67}). Then (\ref{XX2}) follows from (\ref{X2}) and (\ref{x00}).
\end{proof}

\subsection{Setting of Bootstrap}
Let $s_*\ge 0$ be the smallest time such that for all $s_*\le s<\infty$, $0\le j\le L$, $0\le j'\le L$, there holds
\begin{align}
 s^{\frac{j}{2}}\|\partial^{j'}_xA_{x}\|_{\dot{H}^{\frac{d}{2}-1}_x}&\lesssim  \epsilon_1 \label{X1}\\
s^{\frac{j}{2}}\|\partial^{j}_x(d\mathcal{P}e_{l}-\chi^{\infty}_l)\|_{\dot{H}^{\frac{d}{2}}_x}&\lesssim 1 \label{X3}.
\end{align}
By (\ref{xvb67}), (\ref{fcav}), for $s$ sufficiently large depending on $\|dv_0\|_{L^2_x}$, (\ref{X1})-(\ref{X3}) hold.  Our aim is to prove $s_*=0$.

First, we improve the bounds for frames.
\begin{Lemma}\label{j2}
Let $v: [0,\infty)\times \Bbb R^d\to \mathcal{M}$ be heat flow with initial data $v_0\in \mathcal{Q}(\Bbb R^d,\mathcal{M})$. Assume also that (\ref{77}) holds.
Let $\{{e}_{l}\}^{m}_{l=1}$ be the corresponding caloric gauge with given limit $\{e^{\infty}_{l}\}^{m}_{l=1}$.  Recall that the isometric embedding $ \mathcal{M}\hookrightarrow \Bbb R^{M}$ is denoted by $\mathcal{P}$ and $\mathop {\lim }\limits_{s \to \infty } (d{\mathcal{P}})({e_p})=\chi^{\infty}_{l}$. Then if (\ref{X1})-(\ref{X3}) hold in $s\in [s_*,\infty)$, one has the improved bound
\begin{align}\label{3.4}
\|\partial^{j}_x\left((d\mathcal{P} e_{l})-\chi^{\infty}_{l}\right)\|_{\dot{H}^{\frac{d}{2}}_x}\lesssim_{j} \epsilon_1s^{-\frac{j}{2}},
\end{align}
for all $s\in[s_*,\infty)$, $0\le j\le L$,
\end{Lemma}
\begin{proof}
As before, the case when $d$ is odd requires more efforts. From now on assume that $d=2d_0+1$ with $d_0\in\Bbb N_+$. Denote the connection on the bundle $\mathcal{P}^*T\Bbb R^{m}$ by ${\mathbf{D}}$. Denote the induced covariant derivatives on the bundle $v^{*}(\mathcal{P}^*T\Bbb R^{m})$ by $\{{\mathbb D}_{i}\}^{d}_{i=1}$. Then direct calculations show
\begin{align*}
\partial_{x_i}\left((d\mathcal{P} e_{p})-\chi^{\infty}_{p}\right)&=({\mathbb D}_{i}(d\mathcal{P}))(e_p)+d\mathcal{P}(\nabla_{i}e_{p})\\
\partial^2_{x_ix_j}\left((d\mathcal{P} e_{p})-\chi^{\infty}_{p}\right)&= \left({\bf D}^2(d\mathcal{P})\right)(\partial_{i}v, \partial_{j}v;e_p)+\left({\bf D}(d\mathcal{P})\right)(\nabla_j \partial_i v;e_p) \\
&+\left({\bf D}(d\mathcal{P})\right)(\partial_i v;\nabla_{j}e_{p})+\left({\bf D}(d\mathcal{P})\right)(\partial_j v;\nabla_{i}e_{p})+d\mathcal{P}(\nabla_j\nabla_i e_p).
\end{align*}
And schematically we write
\begin{align}\label{huasi}
\partial^{\alpha}_{x}\left((d\mathcal{P} e_{p})-\chi^{\infty}_{p}\right)=\sum^{\alpha}_{k=0}\sum_{a_0+\sum^{k}_{l=1}(a_l+1)=\alpha}\left({\bf D}^{k}(d\mathcal{P})\right)(\nabla^{a_1}_x\partial_{x}v,...,\nabla^{a_k}_x\partial_{x}v;\nabla^{a_0}_xe_p)
\end{align}
In order to estimate the $\dot{H}^{\frac{1}{2}}$ norm, it is convenient to use the difference characterization of $\dot{H}^{\frac{1}{2}}$ and  the geodesic parallel transport. Given $h\in\Bbb R^+$, for fixed $(s,x)\in[0,\infty)\times \Bbb R^d$, let $\gamma(\zeta)$ be the shortest geodesic connecting $v(s,x+h)$ and $v(s,x)$. There may exist more than one shortest geodesic, it suffices to pick up one of them.  Suppose that $\zeta$ is normalized to be the arclength parameter. For any given vector field $V$ on $v^*\mathcal{M}$, denote the parallel transport of $V$ along $\gamma(\zeta)$ by $\widetilde{V}(\gamma(\zeta))$, i.e.
\begin{align}\label{87gt}
\left\{
  \begin{array}{ll}
    \nabla_{\dot{\gamma}(\zeta)}\widetilde{V}(\gamma(\zeta))=0,& \hbox{ } \\
    \widetilde{V}\upharpoonright_{\zeta=0}=V(\gamma(0)), & \hbox{ }
  \end{array}
\right.
\end{align}
for $\zeta\in[0,{\rm dist}(v(s,x),v(s,x+h))]$. Since $\mathcal{P}$ is an isometric embedding, we see $ {\rm dist}(v(s,x),v(s,x+h))=|v(s,x)-v(s,x+h)|$. Introduce the difference operator
\begin{align}
\triangle_{h}f=f(x+h)-f(x).
\end{align}
Denote
\begin{align*}
I_1&=\left({\bf D}^{k}(d\mathcal{P})(v(s,x))\right)(\nabla^{a_1}_x\partial_{x}v,...,\nabla^{a_k}_x\partial_{x}v;\nabla^{a_0}_xe_p)\\
I_2&=\left({\bf D}^{k}(d\mathcal{P})(v(s,x+h))\right)(\widetilde{\nabla^{a_1}_x\partial_{x}v},
...,\widetilde{\nabla^{a_k}_x\partial_{x}v};\widetilde{\nabla^{a_0}_xe_p})
\end{align*}
Then (\ref{87gt}) gives (Recall that $I_1,I_2$ now take values in $\Bbb R^M$)
\begin{align*}
I_2-I_1&=\int^{|\triangle_{h}v(s)|}_0\partial_{\zeta}\left[\left({\bf D}^{k}(d\mathcal{P})(\gamma(\zeta))\right)(\widetilde{\nabla^{a_1}_x\partial_{i_1}v},...,\widetilde{\nabla^{a_k}_x\partial_{i_k}v}
;\widetilde{\nabla^{a_0}_xe_p})\right]d\zeta\\
&=\int^{|\triangle_{h}v(s)|}_0\left({\bf D}^{k+1}(d\mathcal{P})(\gamma(\zeta))\right)(\widetilde{\nabla^{a_1}_x\partial_{x}v},...,\widetilde{\nabla^{a_k}_x\partial_{x}v}
,\dot\gamma;\widetilde{\nabla^{a_0}_xe_p})d\zeta.
\end{align*}
Hence, we have point-wisely that
\begin{align*}
|I_1-I_2|&=|\triangle_{h}v(s)| \sup_{y\in\gamma}
\left|\widetilde{\nabla^{a_1}_x\partial_{x}v}\right|...\left|\widetilde{\nabla^{a_k}_x\partial_{x}v}
\right|\left|\widetilde{\nabla^{a_0}_xe_p}\right|(y).
\end{align*}
By (\ref{87gt}), we observe that $|\widetilde{V}(\gamma(\zeta))|=|{V}(\gamma(0))|$. Thus we arrive at
\begin{align}\label{oiu07}
|I_1-I_2|&\le |\triangle_{h}v(s)| \left|{\nabla^{a_1}_x\partial_{x}v}\right(x)|...\left|{\nabla^{a_k}_x\partial_{x}v}(x)
\right|\left|{\nabla^{a_0}_xe_p}(x)\right|.
\end{align}
Then, by (\ref{p2}), (\ref{oiu07}) gives
\begin{align*}
|I_1-I_2|&\lesssim |\triangle_{h}v(s)||{\nabla^{a_0}_xe_p}|\Pi_{1\le l\le k}\left(\sum^{a_l+1}_{j=1}\sum _{\sum^{j}_{i=1}|\beta_i|=a_l+1}|\partial^{\beta_1}_xv|...|\partial^{\beta_j}_xv|\right).
\end{align*}
Denote
\begin{align*}
I_3&=\left({\bf D}^{k}(d\mathcal{P})(v(s,x+h))\right)({\nabla^{a_1}_x\partial_{x}v},...,
{\nabla^{a_k}_x\partial_{x}v};{\nabla^{a_0}_xe_p})
\end{align*}
Then it is easy to see
\begin{align*}
\left|I_3-I_2\right|&\lesssim \left({\bf D}^{k}(d\mathcal{P})(v(s,x+h))\right)({\nabla^{a_1}_x\partial_{x}v}-\widetilde{{\nabla^{a_1}_x\partial_{x}v}},...,{\nabla^{a_k}_x\partial_{i_k}v};{\nabla^{a_0}_xe_p})
+...\\
&+\left({\bf D}^{k}(d\mathcal{P})(v(s,x+h))\right)({\nabla^{a_1}_x\partial_{x}v},...,
{\nabla^{a_k}_x\partial_{x}v}-\widetilde{{\nabla^{a_k}_x\partial_{x}v}};{\nabla^{a_0}_xe_p})\\
&+\left({\bf D}^{k}(d\mathcal{P})(v(s,x+h))\right)({\nabla^{a_1}_x\partial_{x}v},...,
{\nabla^{a_k}_x\partial_{x}v};{\nabla^{a_0}_xe_p}-\widetilde{{\nabla^{a_0}_xe_{p}}})
\end{align*}
To estimate this formula, we firstly estimate $|\widetilde{V}-V|$ for $V=\nabla^{k}_x\partial_xv$, $\nabla^{j}_x e_p$.  It is convenient to bound $|d\mathcal{P}(\widetilde{V})-d\mathcal{P}(V)|$ instead, because the later takes value in $\Bbb R^{M}$ and equals the former due to the isometric embedding.

{\bf Step 2.} Before bounding $|d\mathcal{P}(\widetilde{V}-V)|$ for $V=\nabla^{k}_x\partial_xv,\nabla^{j}_x e_p$, we use the extrinsic quantities $\{\partial^{j}_xv\}$ to express the intrinsic ones $\nabla^{k}_x\partial_xv$. It is easy to check (we adopt the same notation $v$ to denote both $\mathcal{P}\circ v$ and the map $v$ itself without confusion)
\begin{align}
d\mathcal{P}(\partial_i v)&=\partial_i v;\nonumber\\
d\mathcal{P}(\nabla_j\partial_i v)&=\partial_j [d\mathcal{P}(\partial_i v)]-({\bf D}d\mathcal{P})(\partial_jv;\partial_i v)\nonumber\\
&=\partial^2_{ij}v-({\bf D}d\mathcal{P})(\partial_jv;\partial_i v).\label{ok}\\
&....\nonumber
\end{align}
In Step 1, we have seen that using parallel transport
\begin{align}
&\left|({\bf D}^{k}d\mathcal{P})(V_1,...,V_k;V_0)(x+h)-({\bf D}^{k}d\mathcal{P})(V_1,...,V_k;V_0)(x)\right|\label{2gfcv}\\
&\le\sum^{k}_{i=1}\left|({\bf D}^{k}d\mathcal{P})(\widetilde{{...}},\widetilde{V_{i-1}},V_i-
\widetilde{V_i},\widetilde{V}_{i+1},...;\widetilde{V}_0)(x+h)\right|\label{3.11b}\\
&+
\left|({\bf D}^{k}d\mathcal{P})(\widetilde{V_1},...,
\widetilde{V_k};V_0-\widetilde{V_0})(x+h)\right|\label{3.11a}\\
&+\left|({\bf D}^{k}d\mathcal{P})(V_1,...,V_k;V_0)(x)-({\bf D}^{k}d\mathcal{P})(\widetilde{V_1},...,\widetilde{V_k};\widetilde{V_0})(x+h)\right|.\label{4gfcv}
\end{align}
Moreover, (\ref{4gfcv}) is dominated by
\begin{align}\label{h6}
\left(\sum^{k}_{i=0}\max_{y\in\{x,x+h\}}|{V_i}(y)|\right){\triangle_h}{v(s)}.
\end{align}
We also recall the inequality
\begin{align}\label{h7}
\left|d\mathcal{P}V(x)-d\mathcal{P}\widetilde{V}(x+h)\right|\lesssim \left(\max_{y\in\{x,x+h\}}|{V}(y)|\right)\triangle_{h}v(s).
\end{align}
Since $\mathcal{P}$ is isometric, (\ref{h7}) further yields
\begin{align}
&\left|V(x+h)-\widetilde{V}(x+h)\right|\lesssim \left(\max_{y\in\{x,x+h\}}|{V}(y)|\right)\triangle_{h}v(s)+\left|\triangle_{h}d\mathcal{P}V\right|.\label{h8}
\end{align}
Thus one has by (\ref{h6}), (\ref{h8}), (\ref{2gfcv}) that
\begin{align}
&\left|({\bf D}^{k}d\mathcal{P})(V_1,...,V_k;V_0)(x+h)-({\bf D}^{k}d\mathcal{P})(V_1,...,V_k;V_0)(x)\right|\nonumber\\
&=\sum^{k}_{i=0}\left(\prod^{k}_{l=0,l\neq i}\max_{y\in\{x,x+h\}}|{V_{l}(y)}|\right)\left[|\triangle_{h}v(s)|\max_{y\in\{x,x+h\}}|V_i(y)|+
\left|\triangle_{h}d\mathcal{P}V_i\right|\right]\nonumber\\
&+\left(\prod^{k}_{i=0}\max_{y\in\{x,x+h\}}|{V_i}(y)|\right)|\triangle_{h}v(s)|.\label{NH}
\end{align}
Therefore, applying (\ref{NH}) to (\ref{ok}) yields
\begin{align*}
&\left|({\bf D}d\mathcal{P})(\partial_jv;\partial_i v)(x+h)-({\bf D}d\mathcal{P})({\partial_jv;\partial_i v})(x)\right|\\
&\lesssim |\triangle_{h}v(s)|C^2_{ij}+\left|\partial_jv-\widetilde{\partial_jv}\right||\partial_i v|(x+h)+\left|\partial_iv-\widetilde{\partial_iv}\right||\partial_jv|(x+h)\\
&\lesssim C^2_{ij}|\triangle_{h}v|+\left(D_{ij}+C_{ij}|\triangle_{h}v|\right)C_{ij}
\end{align*}
where we denote
\begin{align*}
&C_{ij}:=\max_{y\in\{x,x+h\}}|\partial_jv(y)|+\max_{y\in\{x,x+h\}}|\partial_iv(y)|\\
&D_{ij}:=\left|\triangle_{h}\partial_jv\right|
+\left|\triangle_{h}\partial_iv\right|.
\end{align*}
We conclude for the second order intrinsic derivatives $\nabla_x\partial_x v$ that
\begin{align*}
&\left|\triangle_{h}d\mathcal{P}(\nabla_j\partial_i v)\right|\lesssim \triangle_{h}\partial^2_{ij}v+C^{2}_{ij}|\triangle_{h}(v)|+|\triangle_{h}(\partial_xv)|C_{ij}.
\end{align*}
By induction, we summarize that
\begin{align}\label{zzx1}
\left| {{\triangle_h}d\mathcal{P}(\nabla _x^k{\partial _x}v)} \right| \lesssim \mathop \sum \limits_{p = 0}^k \sum\limits_{l + \sum\limits_{\mu  = 1}^p {{j_\mu }} (1+{i_\mu })= k + 1,{j_\mu },{i_\mu },l \in\Bbb N} {C_{({i_1})}^{{j_1}}...C_{({i_p})}^{{j_p}}{\triangle_h}\partial _x^lv}.
\end{align}
where we adopt the notation
\begin{align*}
&C_{(i)}:=\sum_{i}\max_{y\in\{x,x+h\}}|\nabla^{i}_x\partial_{x}v(y)|, \mbox{ }{\rm{if}}\mbox{  }i\ge 0.
\end{align*}
Thus by (\ref{h7}), (\ref{h8}), (\ref{zzx1}), the RHS of  (\ref{3.11b}), (\ref{4gfcv})  is dominated by
\begin{align}
 \mathop \sum \limits_{p = 0}^k \sum\limits_{l + \sum\limits_{\mu  = 1}^p {{j_\mu }} (1+{i_\mu } )= k + 1,{j_\mu },{i_\mu },l \in N} {C_{({i_1})}^{{j_1}}...C_{({i_p})}^{{j_p}}{\triangle_h}\partial _x^lv}+\triangle_{h}(v)C_{(k)}.\label{z1}
\end{align}
We now turn to estimate
$\triangle_{h}d\mathcal{P}(\nabla _x^k e_{p})$. Different from the above, we express $\nabla _x^k e_{p}$ by connection coefficients $\{\partial^j_xA_i\}$ rather than by extrinsic quantities $\{\partial^{i}_xd\mathcal{P}e_{p}\}$.
Schematically, we write
\begin{align}
d\mathcal{P}(\nabla_x e)&= A_xd\mathcal{P}(e);\nonumber\\
d\mathcal{P}(\nabla^{\alpha}_xe)&
=\mathop \prod \limits_{_{\sum\limits_l {{j_l}} ({i_l} + 1) =\alpha,{i_l},{j_l} \in \mathbb{N}}}
(\partial^{i_l}A_x)^{j_l}d\mathcal{P}(e).\label{KU}
\end{align}
Then, we deduce that
\begin{align}
&\left|\triangle_{h}d\mathcal{P}(\nabla^k_x e)\right|\lesssim  \sum^{k}_{q=1} \mbox{ } \sum_{\sum^{q}_{\mu=1} j_{\mu}(i_{\mu}+1)=k, i_{\mu},j_{\mu}\in\Bbb N}D^{j_{1}}_{(i_{1})}...
D^{j_{q}}_{(i_{q})}|\triangle_{h} d\mathcal{P}(e)|\nonumber\\
&+\sum_{1\le b\le z\le k} \mbox{ } \sum_{\sum^{z}_{\nu=1} (1+n_{\nu})m_{\nu}=k-1-m_b,n_{\mu},m_{\mu}\in\Bbb N} D^{n_{1}}_{(m_{1})}...D^{n_{z}}_{(m_{z})}\left|\triangle_{h} \partial^{m_b}_xA_x\right|,\label{y1}
\end{align}
where
\begin{align*}
&D_{(j)}:=\sum^{d}_{l=1}\sum_{|\alpha|=j}\max_{y\in\{x,x+h\}}|\partial^{\alpha}_x A_l(y)|, \mbox{ }{\rm{for}}\mbox{  }j\ge 0.
\end{align*}

{\bf Step 3.} We bound $\|C_{(i)}\|_{L^p}$ and $\|D_{(i)}\|_{L^p}$  in this step.
(\ref{70}) and (\ref{270}) show for all $0\le k\le L-1$, $0\le j\le [\frac{d}{2}-1]$
\begin{align}
\|C_{(k)}\|_{L^{\infty}_x}&\lesssim \epsilon_1s^{-\frac{k+1}{2}}\label{Lp1}\\
\|C_{(j)}\|_{L^{d/(j+1)}_x}&\lesssim \epsilon_1.\label{Lp2}
\end{align}
Meanwhile, (\ref{X2}) and (\ref{X1}) show  for any   all $0\le k\le L-1$, $0\le j\le [\frac{d}{2}-1]$
\begin{align}
\|D_{(k)}\|_{L^{\infty}_x}&\lesssim \epsilon_1s^{-\frac{k+1}{2}}\label{LpA1}\\
\|D_{(j)}\|_{L^{d/(j+1)}_x}&\lesssim \epsilon_1.\label{LpA2}
\end{align}

{\bf Step 4.}
Recall $d=2d_0+1$. Inserting the bounds  (\ref{y1}) and (\ref{z1}) to (\ref{NH}) with $V_i=\nabla^{a_i}_x\partial_x v$, $i=1,...,k$, $V_0= \nabla^{a_0}_xe$, we obtain by (\ref{huasi}) that
\begin{align}
&\left|\triangle_{h}\partial^{k}_x\left((d\mathcal{P} e)-\chi^{\infty}\right)\right|\nonumber\\
&\lesssim
\sum^{k}_{l=1} \mbox{ } \sum_{\Omega_1}C^{j_{1}}_{(i_{1})}...
C^{j_{l}}_{(i_{l})}D^{j'_{1}}_{(i'_{1})}...
D^{j'_{l'}}_{(i'_{l'})}|\triangle_{h} d\mathcal{P}(e_p)|\label{1ahuasi}\\
&+\sum^{k}_{1\le b\le z}  {\sum_{\Omega_2}} C^{n_{1}}_{(m_{1})}...C^{n_{z}}_{(m_{z})}D^{n'_{1}}_{(m'_{1})}...D^{n'_{z'}}_{(m'_{z'})}\left|\triangle_{h} \partial^{m_b}_xA_x\right|
\label{2ahuasi}\\
&+\sum_{1\le b\le z\le k} \mbox{ } \sum_{\Omega_3} C^{p_{1}}_{(q_{1})}...C^{p_{z}}_{(q_{z})}D^{p'_{1}}_{(q'_{1})}...D^{p'_{z'}}_{(q'_{z'})}\left|\triangle_{h} \partial^{m_{b}}_x v\right|\label{ahuasi}
\end{align}
where the index sets $\Omega_1,\Omega_2,\Omega_3$ are defined by
\begin{align*}
\Omega_1:    &\sum^{l}_{\mu=1} j_{\mu}(1+i_{\mu})+\sum^{l'}_{\nu=1} j'_{\nu}(i'_{\nu}+1)=k; i_{\mu},j_{\mu}, i'_{\nu},j'_{\nu}\in\Bbb N\\
\Omega_2: &\sum^z_{\mu=1}(1+ m_\mu) n_\mu+ \sum ^{z'}_{\nu  = 1}(1 + m'_\nu ){n'_\nu } = k - {m_b} - 1; \mbox{  }{n_\mu },{m_\mu },{n'_\nu },{m'_\nu } \in \Bbb N\\
\Omega_3: &\sum^z_{\mu=1}(1+ q_\mu) p_\mu+ \sum ^{z'}_{\nu  = 1}(1 + q'_\nu ){p'_\nu } = k - {m_b}; \mbox{  }{p_\mu },{q_\mu },{p'_\nu },{q'_\nu } \in \Bbb N\\
\end{align*}

Then the lemma follows by the difference characterization of $\dot{H}^{\frac{1}{2}}$, Sobolev embeddings and H\"older. We present the details of this part in the following Lemma.
\end{proof}

\noindent
{\bf Remark 4.1.}{ \bf One can also prove Lemma \ref{j2} without using parallel transport.  But the following curvature bounds seem to rely heavily on parallel transport.} We will deal with these geometric quantities in the unified way set up in the proof of Lemma \ref{j2}. The parallel transport technics were previously used by Shatah \cite{Sh} and later McGahagan \cite{huangMchuang} to prove uniqueness of low regularity solutions of wave maps and Schr\"odinger map flows respectively. But using  parallel transport  to estimate fractional Sobolev norms of geometric quantities is original in this paper as far as we know.

\begin{Lemma}\label{kx1}
Let $v\in C( [0,\infty);\mathcal{Q}(\Bbb R^d,\mathcal{M}))$ be heat flow with initial data $v_0$.  Assume also that (\ref{77}) holds and (\ref{X1})-(\ref{X3}) hold in  $s_*\le s<\infty$.
Let $\{{e}_{l}\}^{m}_{l=1}$ be the corresponding caloric gauge with given limit $\{e^{\infty}_{l}\}^{m}_{l=1}$. Then
\begin{align}
\|\partial^{j}_x\mathcal{G}'\|_{L^{\infty}_t{\dot H}^{\frac{d}{2}}_x}&\lesssim_{L} \epsilon s^{-\frac{j}{2}} \label{pzji1}\\
\|\partial^{j}_x\mathcal{G}''\|_{L^{\infty}_t{\dot H}^{\frac{d}{2}}_x}&\lesssim_{L} \epsilon s^{-\frac{j}{2}}   \label{pzji2}
\end{align}
for $0\le j\le L$.
\end{Lemma}
\begin{proof}
We only consider odd $d$. Recall the definition of $\mathcal{G}'$,$\mathcal{G}''$:
\begin{align*}
(\mathcal{G}')_{l}
&=(\widetilde{\nabla}{\bf R})({e_{l}};e_{l_0},e_{l_1},e_{l_2},e_{l_3})-\Gamma^{\infty}_l\\
(\mathcal{G}^{''})_{pl}&=(\widetilde{\nabla}^2{\bf R})\left({e_{p}},{e_{l}};e_{l_0},e_{l_1},e_{l_2},e_{l_3}\right)-\Omega^{\infty}_{pl}.
\end{align*}
where we view $\mathbf{R}$ as a $(0,4)$ tensor. The same arguments as Lemma \ref{j2}
show
\begin{align}
&\partial^{\alpha}_{x}(\mathcal{G}' )=\sum^{\alpha}_{k=0}\sum_{a_0+\sum^{k}_{l=1}(a_{l}+1)+\sum^{3}_{\mu=0}b_{\mu}=j}(\widetilde{\nabla}^{k+1}{\bf R})(\nabla^{a_0}_x e, \nabla^{a_1}_x\partial_x v,...,\nabla^{a_k}_x \partial_x v; \nabla^{b_0}_xe_{l_0},...,\nabla^{b_3}_xe_{l_3})\label{huakai1}\\
&\partial^{\alpha}_{x}(\mathcal{G}^{''} )=\sum^{\alpha}_{k=0}\sum_{\Omega}(\widetilde{\nabla}^{k+2}{\bf R})(\nabla^{a_0}_x e,\nabla^{a'_0}_x e,\nabla^{a_1}_x \partial_xv...,\nabla^{a_k}_x \partial_xv; \nabla^{b_0}_xe_{l_0},...,\nabla^{b_3}_xe_{l_3}),
\label{huakai2}
\end{align}
where the index set $\Omega$ is defined by
\begin{align}
\Omega:  a'_{0}+a_{0}+\sum^{k}_{l=1}(1+a_{l})+\sum^{3}_{\mu=0}b_{\mu}=\alpha.
\end{align}
Recall $d=2d_0+1$. By difference characterization of Besov spaces
and (\ref{huakai1}), to prove (\ref{pzji1}) it suffices to verify
\begin{align*}
&\left\|\tau^{-\frac{1}{2}}\sup_{|h|\le \tau}\|\triangle_{h}(\widetilde{\nabla}^{k+1}{\bf R})(\nabla^{a_0}_x e,\nabla^{a_1}_x\partial_x v,...,\nabla^{a_k}_x\partial_x v; \nabla^{b_0}_xe_{l_0},...,\nabla^{b_3}_xe_{l_3})\|_{L^{2}_{x}}\right\|_{{L^{\infty}_t}L^{2}(\tau^{-1}d\tau)}\\
&\lesssim_{L} \epsilon s^{-\frac{\beta}{2}}
\end{align*}
provided that
\begin{align*}
{a_0+\sum^{k}_{l=1}(a_{l}+1)+\sum^{3}_{\mu=0}b_{\mu}=\beta+d_0}, 0\le k\le \beta+d_0.
\end{align*}
Using the (\ref{NH}) type estimates and (\ref{y1}), we obtain
\begin{align}
&\left|\triangle_{h}(\widetilde{\nabla}^{k+1}{\bf R})(\nabla^{a_0}_x e,...,\nabla^{a_k}_x \partial_xv; \nabla^{b_0}_xe_{l_0},...,\nabla^{b_3}_xe_{l_3})\right|\\
&\lesssim \sum^{\beta+d_0}_{z,z'=1} \sum_{\Omega_1}C^{j_{1}}_{(i_{1})}...
C^{j_{z}}_{(i_{z})}D^{j'_{1}}_{(i'_{1})}...
D^{j'_{z'}}_{(i'_{z'})}|\triangle_{h} d\mathcal{P}(e)|\label{a99}\\
&+\sum^{ \beta+d_0}_{1\le b\le z', z\ge 1}  \sum_{\Omega_2} C^{n_{1}}_{(m_{1})}...C^{n_{z}}_{(m_{z})}D^{n'_{1}}_{(m'_{1})}...D^{n'_{z'}}_{(m'_{z'})}\left|\triangle_{h} \partial^{m'_b}_xA_x\right|\label{ka99}\\
&+\sum^{\beta+d_0}_{1\le b\le  z, z'\ge 1} \sum_{\Omega_3} C^{p_{1}}_{(q_{1})}...C^{p_{z}}_{(q_{z})}D^{p'_{1}}_{(q'_{1})}...D^{p'_{z'}}_{(q'_{z'})}\left|\triangle_{h}\partial^{m_{b}}_x v\right|\label{kka99}
\end{align}
where the index sets $\Omega_1,\Omega_2,\Omega_3$ are defined by
\begin{align}
\Omega_1: &\sum^{z}_{\nu=1} j_{\nu}(1+i_{\nu})+\sum^{z'}_{\mu=1} (i'_{\mu}+1)j'_{\mu}=\beta+d_0, \mbox{  } j_{\nu},i'_{\mu},j'_{\mu},i_{\nu}\in\Bbb N\label{AP1}\\
\Omega_2: &\sum^{z}_{\nu=1} (1+m_{\nu})n_{\nu}+\sum^{z'}_{\mu=1} (1+m'_{\mu})n'_{\mu}=\beta+d_0-(m'_b+1),\mbox{  }n_{\nu},n'_{\mu},m'_{\mu},m_{\nu}\in\Bbb N\nonumber\\
\Omega_3: &\sum^{z}_{\nu=1} (1+q_{\nu})p_{\nu}+\sum^{z'}_{\mu=1} (1+q'_{\mu})p'_{\mu}=\beta+d_0-m_b,\mbox{  }p_{\nu},p'_{\mu},q'_{\mu},q_{\nu}\in\Bbb N.\nonumber
\end{align}
As before, we consider two subcases.
 {\bf{Case 1.}}{ Assume that all $\{i'_{\mu}\}$, $\{i_{\nu}\}$ in (\ref{a99}) satisfy}
\begin{align}
0\le i_{\nu}\le \frac{d}{2}-1,\mbox{  }0\le i'_{\mu}\le \frac{d}{2}-1\label{a991}.
\end{align}
Then (\ref{Lp2}), (\ref{LpA2}) show for
 \begin{align}
\frac{1}{\widehat{r}_{\nu}}&=\frac{1}{2}-\frac{1}{d}(\frac{d}{2}-i_{\nu}-1)\\
\frac{1}{\widetilde{r}_{\mu}}&=\frac{1}{2}-\frac{1}{d}(\frac{d}{2}-i'_{\mu}-1)\label{995}
\end{align}
there hold
\begin{align}
\|C_{(i_{\nu})}\|_{L^{\infty}_tL^{\widehat{r}_{\nu}}_x}&\lesssim \epsilon \label{qcau}\\
\|D_{(i'_{\mu})}\|_{L^{\infty}_tL^{\widetilde{r}_{\mu}}_x}&\lesssim \epsilon \label{qau}
\end{align}
And interpolating (\ref{qcau}), (\ref{qau}) with $L^{\infty}$ bounds given by (\ref{Lp1})-(\ref{LpA1}), we have
\begin{align}
\|C_{(i_{\nu})}\|_{L^{\infty}_tL^{\overline{r}}_x}&\lesssim \epsilon s^{-\frac{1+i_{\nu}}{2}+\frac{d}{2\overline{r}}}\label{qua23}\\
\|D_{(i'_{\mu})}\|_{L^{\infty}_tL^{\underline{r}}_x}&\lesssim \epsilon s^{-\frac{1+i'_{\mu}}{2}+\frac{d}{2\underline{r}}}\label{qua2}
\end{align}
for all $\underline{r}\in[\widetilde{r}_{\mu},\infty]$ and
for all $\overline{r}\in[\widehat{r}_{\nu},\infty]$.  Without loss of generality we assume  {\bf Case 1a.} $j_1>0$ or  {\bf Case 1b.} $j'_1>0$.

In the {\bf Case 1a}, let $\{{\overline{r}}_{\nu}\}$, $\{{\underline{r}}_{\mu}\}$ be fixed exponents such that $\overline{r}_{\nu}\in[\widehat{r}_{\nu},\infty]$, $\underline{r}_{\mu}\in[\widetilde{r}_{\mu},\infty]$ and
\begin{align}\label{uhbg}
\sum^{z}_{\nu=1}\frac{d}{\overline{r}_{\nu}}j_{\nu}+\sum^{z'}_{\mu=1}
\frac{d}{\underline{r}_{\mu}}j'_{\mu}=d_0.
\end{align}
The exponents $\{\underline{r}_{\mu}\}$, $\{\overline{r}_{\nu}\}$ in (\ref{uhbg}) do exist. In fact, since $j_{1}>0$ the LHS of (\ref{uhbg}) is a continuous deceasing function with respect to any of $\overline{r}_{\nu}\in [\widehat{r}_{\nu},\infty]$, $\underline{r}_{\mu}\in[\widetilde{r}_{\mu},\infty]$ . Hence the LHS of  (\ref{uhbg}) ranges over $[0,\beta+d_0]$. Thus (\ref{uhbg}) holds with appropriate $\{\overline{r}_{\nu},\underline{r}_{\mu}\}$. Then, in Case 1a one obtains
\begin{align}
&\left\|\frac{1}{\tau^{\frac{1}{2}}}\sup_{|h|\le \tau}|C^{j_{1}}_{(i_{1})}...
D^{j_{z}}_{(i_{z})}D^{j'_{1}}_{(i'_{1})}...
D^{j'_{z'}}_{(i'_{z'})}|\triangle_{h} d\mathcal{P}(e_p)|\right\|_{L^{\infty}_tL^{2}(\tau^{-1}d\tau)L^{2}_x}\label{qqt}\\
&\lesssim
\|C_{(i_{1})}\|^{j_1}_{L^{\infty}_{t}L^{\overline{r}_1}_x}
\|C_{(i_2)}\|^{j_{2}}_{L^{\infty}_{t}L^{\overline{r}_{2}}_x}...
\|C_{(i_{z})}\|^{j_{z}}_{L^{\infty}_{t}L^{\overline{{r}}_{z}}_x}\\
&\times
\|D_{(i'_{1})}\|^{j'_1}_{L^{\infty}_{t}L^{\underline{r}_1}_x}...
\|D_{(i'_{z'})}\|^{j'_{z'}}_{L^{\infty}_{t}L^{\underline{r}_{z'}}_x}
\left\|\frac{1}{\tau^{\frac{1}{2}}}\sup_{|h|\le \tau}|\triangle_{h} d\mathcal{P}(e_p)|\right\|_{L^{\infty}_tL^{\infty}(\tau^{-1}d\tau)L^{2d}_x}\nonumber
\end{align}
where in the second line to apply H\"older inequality we used the following equality
\begin{align*}
\frac{1}{2}-\frac{1}{2d}=
\sum^{z}_{\nu=1}\frac{j_{\nu}}{\overline{r}_{\nu}}+\sum^{z'}_{\mu=1}\frac{j'_{\mu}}{\underline{r}_{\mu}}
\end{align*}
which follows from (\ref{uhbg}). Thus (\ref{qqt}), (\ref{qua2}) and (\ref{qua23}) imply
\begin{align*}
&\left\|\frac{1}{\tau}\sup_{|h|\le \tau}|C^{j_{1}}_{(i_{1})}...
C^{j_{z}}_{(i_{z})}D^{j'_{1}}_{(i'_{1})}...
D^{j'_{z'}}_{(i'_{z'})}|\triangle_{h} d\mathcal{P}(e_p)|\right\|_{L^{\infty}_tL^2_{\tau}(\Bbb R^+)L^{2}_x}
\\
&\lesssim \epsilon s^{-\frac{\beta}{2}}\left\|\frac{1}{\tau}\sup_{|h|\le \tau}|\triangle_{h} d\mathcal{P}(e_p)|\right\|_{L^{\infty}_tL^2_{\tau}(\Bbb R^+)L^{2d}_x}\\
&\lesssim \epsilon s^{-\frac{\beta}{2}}\left\| d\mathcal{P}(e_p)\right\|_{L^{\infty}_t{\dot H}^{\frac{d}{2}}_x}.
\end{align*}
where the power $\frac{\beta}{2}$ in $s^{-\frac{\beta}{2}}$ of the second line results from (\ref{uhbg}), (\ref{AP1}).
Therefore, we get by (\ref{X3})  that
\begin{align*}
\left\|\frac{1}{\tau}\sup_{|h|\le \tau}|C^{j_{1}}_{(i_{1})}...
C^{j_{z}}_{(i_{q})}D^{j'_{1}}_{(i'_{1})}...
D^{j'_{z'}}_{(i'_{z'})}|\triangle_{h} d\mathcal{P}(e_p)|\right\|_{L^{\infty}_tL^2_{\tau}(\Bbb R^+)L^{2}_x}
\lesssim \epsilon s^{-\beta/2}.
\end{align*}
The Case 1.b follows by the same way. Thus (\ref{a99}) has been done.
The left two terms (\ref{ka99})-(\ref{kka99}) can be dominated as  (\ref{a99}) by (\ref{X1}) and Lemma \ref{j1}.

 {\bf Case 2.a}{ Assume that among $\{i'_{\mu}:j'_{\mu}>0, \mu=1,...,z'\}$ in (\ref{a99}) there exists an $i'_{\mu'}$ such that }
\begin{align}
i'_{\mu'}>\frac{d}{2}-1\label{aa991}.
\end{align}
Since $\|g\|_{L^{\rho}_x}\lesssim \|g\|_{{\dot H}^{\frac{1}{2}}_{x}}$ with $$\frac{1}{\rho}=\frac{1}{2}-\frac{1}{2d},$$
we deduce by interpolation and (\ref{X1}) that
\begin{align}\label{qqt2}
\|D_{(i'_{\mu'})}\|_{L^{\infty}_{t}L^{\rho}_x}\lesssim \epsilon s^{-\frac{1}{2}(i'_{\mu'}+\frac{3}{2}-\frac{d}{2})}.
\end{align}
Without loss of generality assume that $\mu'=1$. Then in Case 2a, there holds
\begin{align}
&\left\|\frac{1}{\tau}\sup_{|h|\le \tau}|C^{j_{1}}_{(i_{1})}...
C^{j_{z}}_{(i_{z})}D^{j'_{1}}_{(i'_{1})}...
D^{j'_{z'}}_{(i'_{z'})}||\triangle_{h} d\mathcal{P}(e)|\right\|_{L^{\infty}_tL^2_{\tau}(\Bbb R^+)L^{2}_x}\nonumber\\
&\lesssim
\|C_{(i_{1})}\|^{j_1}_{L^{\infty}_{t_x}}...\|C_{(i_{z})}\|^{j_{z}}_{L^{\infty}_{t,x}}
\|D_{(i'_{1})}\|_{L^{\infty}_{t}L^{\rho}_x}\|D_{(i'_{1})}\|^{j'_{1}-1}_{L^{\infty}_{t,x}}
\left(\prod^{z'}_{\mu=2}\|D_{(i'_{\mu})}\|^{j'_{\mu}}_{L^{\infty}_{t,x}}\right)
\nonumber\\
&\left\|\frac{1}{\tau}\sup_{|h|\le \tau}|\triangle_{h} d\mathcal{P}(e_p)|\right\|_{L^{\infty}_tL^2_{\tau}(\Bbb R^+)L^{2d}_x}\nonumber\\
&\lesssim  \epsilon s^{-\frac{1}{2}(i_{1}+\frac{3}{2}-\frac{d}{2})}
s^{-\frac{j_{1}-1}{2}(i_{1}+1)}\left(\prod^{z'}_{\mu=2}s^{-\frac{j'_{\mu}}{2}(i'_{\mu}+1)}\right)
\label{9900}
\end{align}
where we applied  $L^{\infty}$ bounds given by  (\ref{X1}) and Lemma \ref{j2} in the last line.
It is easy to check that the RHS of (\ref{9900}) is exactly $\epsilon s^{-\frac{\beta}{2}}$.

{\bf Case 2.b}{ Assume that among $\{i_{\nu}:j_{\nu}>0, \nu=1,...,z\}$ in (\ref{a99}) there exists an $i_{\nu'}$ such that $i_{\nu'}>\frac{d}{2}-1$.} This follows by the same path if one has the analogy of (\ref{qqt2}) for $C_{(i)}$:
\begin{align}\label{g7chen}
\|C_{(i_{\nu'})}\|_{L^{\infty}_{t}L^{\rho}_x}\lesssim \epsilon s^{-\frac{1}{2}(i_{\nu'}+\frac{3}{2}-\frac{d}{2})}.
\end{align}
Notice that  (\ref{g7chen}) follows by (\ref{370}).

The left (\ref{ka99}), (\ref{kka99}) parts follow  by the same argument in {\bf Case 2}.

Therefore, as a summary, we have proved (\ref{pzji1}) with assuming (\ref{X1})-(\ref{X3}). The left (\ref{pzji2}) is the same.
\end{proof}

\begin{Lemma}\label{j31}
Let $v: [0,\infty)\times \Bbb R^d\to \mathcal{M}$ be heat flow with initial data $v_0\in \mathcal{Q}(\Bbb R^d,\mathcal{M})$. Assume also that (\ref{77}) holds and (\ref{X1})-(\ref{X3}) hold in  $s_*\le s<\infty$.
Let $\{{e}_{l}\}^{m}_{l=1}$ be the corresponding caloric gauge. Then for $s_*\le s<\infty$, $0\le j\le L$, the differential fields $\{\psi_i\}^{d}_{i=1}$ satisfy
\begin{align}
\|\partial^{j}_x\psi_x\|_{\dot{H}^{\frac{d}{2}-1}_x}&\lesssim_{j} \epsilon_1 s^{-\frac{j}{2}}\label{p1990}\\
\|\partial^{j}_x\psi_x\|_{L^{\infty}_x}&\lesssim_{L} \epsilon_1  s^{-\frac{j+1}{2}}.\label{p1991}
\end{align}
And the heat tension field $\psi_s$ satisfies
\begin{align}
\|\partial^{l}_x\partial_s v\|_{\dot{H}^{\frac{d}{2}-1}_x}&\lesssim_{j} \epsilon_1 s^{-\frac{l+1}{2}}\label{p1994}\\
\|\partial^{l}_x\partial_s v\|_{L^{\infty}_x}&\lesssim_{j} \epsilon_1 s^{-\frac{l+2}{2}}\label{p1995}\\
\|\partial^{l}_x\psi_s\|_{\dot{H}^{\frac{d}{2}-1}_x}&\lesssim_{l} \epsilon_1 s^{-\frac{l+2}{2}}\label{p1992}\\
\|\partial^{l}_x\psi_s\|_{L^{\infty}_x}&\lesssim_{l} \epsilon_1 s^{-\frac{l+2}{2}}.\label{p1993}
\end{align}
for any $s_*\le s<\infty$, $0\le l\le L$.
Moreover, we have
\begin{align}\label{1999}
\|\psi_s\|_{L^2_s\dot{H}^{\frac{d}{2}-1}_x}&\lesssim  \epsilon_1.
\end{align}
\end{Lemma}
\begin{proof}
Recall that the fixed integer $L'$ in Section 3 satisfies $L'=L+5$.
By the definition of differential field $\{\psi_x\}$ and the isometry of $d\mathcal{P}$, we obtain
\begin{align}
\psi^{l}_i&=\partial_iv\cdot d \mathcal{P}(e_{l})\\
\psi^{l}_s&=\partial_sv\cdot d \mathcal{P}(e_{l}),\label{8c97}
\end{align}
where we write $\partial_\alpha v$ instead of $\partial_\alpha(\mathcal{P}\circ v)$ for simplicity and $i=1,...,d$.
Generally one has
\begin{align}
\partial^{j}_x\psi_x=\sum_{k_1+k_2=j,k_1,k_2\in\Bbb N}\partial^{k_1+1}_xv\cdot\partial^{k_2}_x d\mathcal{P}(e)\label{jni80}\\
\partial^{j}_x\psi_s=\sum_{k_1+k_2=j,k_1,k_2\in\Bbb N}\partial^{k_1}_x\partial_sv\cdot\partial^{k_2}_x \label{jni81} d\mathcal{P}(e).
\end{align}
(\ref{XX2}) and Lemma \ref{j2} yield the bounds for $\mathcal{P}e$:
\begin{align}
\|\partial^{j}_x(\mathcal{P}(e)-\chi^{\infty})\|_{\dot{H}^{\frac{d}{2}}_x}&\lesssim \epsilon s^{-\frac{j}{2}}\label{qqqvcx}\\
\|\partial^{j}_x(\mathcal{P}(e)-\chi^{\infty})\|_{L^{\infty}_x}&\lesssim  \epsilon s^{-\frac{j}{2}}.\label{2qqqvcx}
\end{align}
Then (\ref{qqqvcx})-(\ref{2qqqvcx}), Lemma \ref{j1}, fractional Leibnitz rule and Sobolev inequalities imply (\ref{p1990}).
Meanwhile, (\ref{XX2}) and  (\ref{1.8a}) yield (\ref{p1991})
by (\ref{jni80}).

(\ref{p1995}) directly follows from the heat flow equation and (\ref{zzFD}). (\ref{jni81}) together with (\ref{p1995}) gives  (\ref{p1993}).
And by the heat flow equation, (\ref{p1994}) reduces to prove
\begin{align*}
\|\partial^{j}_x[S(v)(\partial_xv,\partial_x v)]\|_{\dot{H}^{\frac{d}{2}-1}}\lesssim \epsilon s^{-\frac{j+1}{2}},
\end{align*}
which follows by fractional Leibnitz rule, Sobolev inequalities and Lemma \ref{j1}.
Similarly, (\ref{jni81}) together with (\ref{p1994}) gives  (\ref{p1992}).
Lastly, (\ref{1999}) follows by (\ref{1mkn}) and (\ref{jni81})-(\ref{2qqqvcx}).

\end{proof}

\begin{Lemma}\label{j3h}
Let $v: [0,\infty)\times \Bbb R^d\to \mathcal{M}$ be heat flow with initial data $v_0\in \mathcal{Q}(\Bbb R^d,\mathcal{M})$. Assume also that (\ref{77}) holds and (\ref{X1})-(\ref{X3}) hold in $s\in [s_*,\infty)$.
Let $\{{e}_{l}\}^{m}_{l=1}$ be the corresponding caloric gauge. Then the connection coefficients satisfy
\begin{align}
\|\partial^{j}_xA_{x}\|_{\dot{H}^{\frac{d}{2}-1}_x}&\lesssim_{j} \epsilon^2s^{-\frac{j}{2}}\label{Nbv1}\\
\|\partial^{j}_xA_{x}\|_{L^{\infty}_x}&\lesssim_{j} \epsilon^2s^{-\frac{j+1}{2}}.\label{Nbv2}
\end{align}
for any $s\in [s_*,\infty)$, $0\le j\le L$.
\end{Lemma}
\begin{proof}
By (\ref{edf}), we have
\begin{align*}
[{A_i}]^{p}_{q}(s)= \int_s^\infty\langle {\mathbf{R}(v(s'))\left( {{\partial _s}v(s'),{\partial _i}v(s')} \right)} e_{p},e_{q}\rangle ds',
\end{align*}
which can be schematically written as
\begin{align}\label{journal}
[{A_i}]^{p}_{q}(s)= \int_s^\infty (\psi_s \psi_x)\langle{\mathbf{R}}\left( e_{l_0},e_{l_1}\right)e_{l_2},e_{l_3}\rangle ds'.
\end{align}
Following arguments of Lemma \ref{kx1} gives
\begin{align*}
\|\partial^{j}_x\langle{\mathbf{R}}\left( e_{l_0},e_{l_1}\right)e_{l_2},e_{l_3}\rangle\|_{\dot{H}^{\frac{d}{2}}_x}&\lesssim \epsilon s^{-\frac{j}{2}}\\
\|\partial^{j}_x\langle{\mathbf{R}}\left( e_{l_0},e_{l_1}\right)e_{l_2},e_{l_3}\rangle\|_{L^{\infty}_x}&\lesssim  c(j)  s^{-\frac{j}{2}},
\end{align*}
where $c(0)=1$ and $c(j)=\epsilon$ if $1\le j\le L$ in the second line. Then Lemma \ref{j31} and fractional Leibnitz rules show
\begin{align}
\left\|\partial^{j}_x\left(\psi_x \psi_s\langle{\mathbf{R}}( e_{l_0},e_{l_1})e_{l_2},e_{l_3}\rangle\right)\right\|_{\dot{H}^{\frac{d}{2}-1}_x}&\lesssim \epsilon^2 s^{-\frac{j+2}{2}}\label{Nvz1}\\
\left\|\partial^{j}_x\left(\psi_x \psi_s\langle{\mathbf{R}}( e_{l_0},e_{l_1})e_{l_2},e_{l_3}\rangle\right)\right\|_{L^{\infty}_x}&\lesssim \epsilon^2 s^{-\frac{j+3}{2}}.\label{Nvz2}
\end{align}
Thus integrating  (\ref{Nvz2}) in $s'\in[s,\infty)$ gives (\ref{Nbv2}) for all $0\le j\le L$. And integrating (\ref{Nvz1}) in $s'\in[s,\infty)$ gives (\ref{Nbv1}) for any $1\le j\le L$. Hence, it suffices to prove (\ref{Nbv1}) with $j=0$.

As before, assume that $d=2d_0+1$ is odd. Then one has by fractional Leibnitz rule and Sobolev embedding
that
\begin{align*}
&\|\psi_x \psi_s\langle{\mathbf{R}}( e_{l_0},e_{l_1})e_{l_2},e_{l_3}\rangle\|_{L^1_s{\dot H}^{\frac{d}{2}-1}}
\\
&\lesssim \sum_{\sum^{3}_{i=1}j_i=d_0-1}\|\partial^{j_1}_x\psi_x\partial^{j_2}_x \psi_s\partial^{j_3}_x\langle{\mathbf{R}}( e_{l_0},e_{l_1})e_{l_2},e_{l_3}\rangle\|_{L^1_s{\dot H}^{\frac{1}{2}}_x}\\
&\lesssim \sum_{\sum^{3}_{i=1}\beta_i=d_0-\frac{1}{2}, \beta_i\in  \Bbb N \bigcup (\Bbb N+\frac{1}{2}) }\||\nabla|^{\beta_1}\psi_x\|_{L^2_sL^{p(\beta_1)}_x}\||\nabla|^{\beta_2}_x \psi_s\|_{L^2_sL^{q(\beta_2)}_x}\\
&\mbox{  }\mbox{  }\mbox{  }\mbox{  }\mbox{  }\mbox{  }\mbox{  }\mbox{  }
\mbox{  }\mbox{  }\mbox{  }\mbox{  }\mbox{  }\mbox{  }\mbox{  }\mbox{  }
\mbox{  }\mbox{  }\mbox{  }\mbox{  }\mbox{  }\mbox{  }\mbox{  }\mbox{  }
\mbox{  }\mbox{  }\mbox{  }\mbox{  }\mbox{  }\mbox{  }\mbox{  }\mbox{  }
\times\||\nabla|^{\beta_3}\langle{\mathbf{R}}( e_{l_0},e_{l_1})e_{l_2},e_{l_3}\rangle\|_{L^{\infty}_sL^{r(\beta_3)}_x}\\
&\lesssim \|\psi_x\|_{L^2_s\dot{H}^{\frac{d}{2}}}\|\psi_s\|_{L^2_s\dot{H}^{\frac{d}{2}-1}}\\
&\lesssim \epsilon^2,
\end{align*}
where in the $p(\beta_1), q(\beta_2),r(\beta_3)$ are defined by
\begin{align*}
\frac{\beta_1}{d}=\frac{1}{p(\beta_1)}\\
\frac{1}{d}+\frac{\beta_2}{d}=\frac{1}{q(\beta_2)}\\
\frac{\beta_3}{d}=\frac{1}{r(\beta_3)}.
\end{align*}
Thus (\ref{Nbv1}) for all $0\le j\le L$ is done as well.
\end{proof}

\subsection{Proof of Proposition 1.1.}

We want to prove Proposition  1.1 by bootstrap. First, there exists a sufficiently large $s_*$ such that (\ref{X1})-(\ref{X3}) hold  for $s\in[s_*,\infty)$ because Lemma \ref{QQ} says the left hand sides of (\ref{X1}) and (\ref{X3}) decay  to 0 as $s\to\infty$. Second, by Lemma \ref{j3h}, one can push $s_*$ to be $0$.  Therefore, all the bounds stated by Lemma \ref{j1}-Lemma \ref{j3h} hold for all $s\in [0,\infty)$.

\section{ Decay estimates in block spaces of $F_k$ }

According to the definition of $F_{k}$ space, it is natural to track the following four block spaces of $F_{k}$:
\begin{align}\label{ijnzbv}
L^{\infty}_tL^{2}_x, \mbox{  }L^{p_d}_{t,x},\mbox{  } L^{p_{d}}_xL^{\infty}_t,\mbox{  }  L^{2,\infty}_{\vec e}
\end{align}
along the heat flow. We will see in Section 5 there is no need to track the $L^{2,\infty}_{\vec e}$ blocks of $F_{k}$ in the heat direction, which makes large convenience for us.

\subsection{Tracking the $L^{p_{d}}_{t,x}\bigcap L^{\infty}_tL^2_x$ block}

Let $u$ be a solution to SL. Proposition 1.1
with additional efforts yields
\begin{Corollary}\label{cXS}
Let  $u\in C([-T,T];\mathcal{Q}(\Bbb R^d,\mathcal{N}))$ be a solution to SL.
Let $v$ be the solution of heat flow with initial data $u$.  Assume also that $v$ is global
in the heat direction and $\{e_{l}\}^{2n}_{l=1}$ denotes the caloric gauge. Denote $\{\psi_i\}^d_{i=1}$ the associated differential fields. Define the frequency envelopes $\{h_{k}(\sigma)\}$ by
\begin{align*}
h_{k}(\sigma)&=\sup_{s\ge 0,k'\in\Bbb Z}(1+s2^{2k'})^{4}2^{-\delta|k-k'|}2^{(\frac{d}{2}-1)k'}2^{\sigma k'}\|P_{k'}\psi_x\|_{L^{\infty}_tL^2_x}\\
h_{k}&:=h_{k}(0).
\end{align*}
Suppose that $\{h_{k}(\sigma)\}$ satisfy
\begin{align}\label{itical}
\sum_{k\in\Bbb Z}h^2_{k}\le \epsilon_1,\mbox{  }\sum_{k\in\Bbb Z}h^2_{k}(\sigma)<\infty, \mbox{ }\forall\sigma \in [0,\vartheta].
\end{align}
Then $v$ satisfies for $s\in\Bbb R^+$, $0\le l\le L$ that
\begin{align}
\|\partial^{l}_xv\|_{L^{\infty}_t\dot{H}^{\frac{d}{2}}_{x}}&\lesssim s^{-\frac{l}{2}}\epsilon_1\label{FD}\\
\|\partial^{l}_{x}(d\mathcal{P}(e_l)-\chi^{\infty}_l)\|_{L^{\infty}_t\dot{H}^{\frac{d}{2}}_x}&\lesssim  s^{-\frac{l}{2}}\epsilon_1\label{xindv}
\end{align}
And the corresponding differential fields and connection coefficients satisfy
\begin{align}
\|\partial^{j}_x\phi_x\|_{L^{\infty}_t{\dot H}^{\frac{d}{2}-1}_x}&\lesssim_j \epsilon_1 s^{-\frac{j}{2}}\\
\|\partial^{j}_xA_{x}\|_{L^{\infty}_{t}\dot{H}^{\frac{d}{2}-1}_x}&\lesssim_{j} \epsilon_1 s^{-\frac{j}{2}}\label{a993}\\
\|\partial^{j}_x\phi_x\|_{L^{\infty}_{t}L^{\infty}_x}&\lesssim_j \epsilon_1 s^{\frac{j+1}{2}}\\
\|\partial^{j}_xA_{x}\|_{L^{\infty}_{t}L^{\infty}_x}&\lesssim_{j} \epsilon_1 s^{-\frac{j+1}{2}}
\end{align}
and for all $0\le j\le L$, $s>0$.
\end{Corollary}

\begin{proof}
Compared with Proposition 1.1, we have assumption (\ref{itical}) here rather than
 \begin{align}\label{BOOT}
\|v\|_{L^{\infty}_t\dot{H}^{\frac{d}{2}}_x}&\le \epsilon_1.
\end{align}
 In order to transpose (\ref{itical})  to  (\ref{BOOT}), we need a tricky bootstrap argument.

First, we reconstruct the subcritical theory (energy dependent) estimates presented in Lemma \ref{Q} with assumption (\ref{itical}).
Since the energy conserves along the Schrodinger map flow we get
\begin{align}
\|du\|_{L^{\infty}_tL^2_x}\le \|du_0\|_{L^2_x}.
\end{align}
Checking the proof of (\ref{xvb67}) of Lemma \ref{QQ}, one finds only the following estimates are a-prior used:
\begin{align}
\|dv\|_{L^{\infty}_tL^{\infty}_x}&\lesssim \epsilon s^{-\frac{1}{2}}\label{y8}\\
\|\nabla^{j}dv\|_{L^{\infty}_tL^{\infty}_x}&\lesssim \epsilon s^{-\frac{j}{2}}.\label{y98}\\
\|\nabla^{j}dv\|_{L^{\infty}_tL^{2}_x}&\lesssim \|dv_0\|_{L^2_x}s^{-\frac{j}{2}}\label{y99}\\
\|\nabla^{j}dv\|_{L^{\infty}_tL^{\infty}_x}&\lesssim \|dv_0\|_{L^2_x}s^{-\frac{2j+d}{4}}.\label{y100}
\end{align}
Thus if (\ref{y8})-(\ref{y100}) are obtained with assumption (\ref{itical}), then the estimate  (\ref{xvb67}) in Lemma \ref{QQ} holds here as well.
Notice that (\ref{y8})  follows by (\ref{itical}), Gagliardo-Nirenberg inequality
\begin{align*}
\|dv\|_{L^{\infty}_tL^{\infty}_x}&\lesssim \|\phi_x\|_{L^{\infty}_tL^{\infty}_x}
\lesssim \|\phi_x\|^{\frac{1}{2}}_{L^{\infty}_t{\dot H}^{\frac{d}{2}-1}_x}\|\phi_x\|^{\frac{1}{2}}_{L^{\infty}_t{\dot H}^{\frac{d}{2}+1}_x}
\lesssim \epsilon s^{-\frac{1}{2}}.
\end{align*}
Now, we turn to prove (\ref{y98}). Denote
\begin{align*}
Z_{\infty,k}(s)=\sup_{\tilde{s}\in[0,s]}\tilde{s}^{\frac{k}{2}}\|\partial^{k}_xv(\tilde{s})\|_{L^{\infty}_{t,x}}.
\end{align*}
By the heat flow equation, one has
\begin{align*}
\|\partial^{k+1}_xv(s)\|_{L^{\infty}_{t,x}}&\lesssim
s^{-\frac{1}{2}}\|\partial^{k}_x v(s/2)\|_{L^{\infty}_{t,x}}
+\int^{s}_{\frac{s}{2}}(s-s')^{\frac{1}{2}}\|\partial^{k}_x(S(v)(\partial_x,\partial_x))\|_{L^{\infty}_{t,x}}ds'\\
&\lesssim  s^{-\frac{k+1}{2}}Z_{\infty,k}(s)
+\int^{s}_{\frac{s}{2}}(s-s')^{\frac{1}{2}}\|\partial^{k+1}_xv\|_{L^{\infty}_{t,x}}\|dv\|_{L^{\infty}_{t,x}}ds'\\
&+\sum^{k+2}_{l=2}\sum_{\sum^{l}_{i=1}j_{i}=k+2,1\le j_i\le k}\int^{s}_{\frac{s}{2}}(s-s')^{\frac{1}{2}}\|\partial^{j_1}_xv\|_{L^{\infty}_{t,x}}....\|\partial^{j_l}v\|_{L^{\infty}_{t,x}}ds'.
\end{align*}
Then we have by (\ref{y8}) that
\begin{align*}
Z_{\infty,1}(s)&\lesssim \epsilon\\
Z_{\infty,k'}(s)&\lesssim \epsilon, \mbox{ }\forall k'\le k   \Rightarrow Z_{\infty,k+1}(s)\lesssim \epsilon Z_{\infty,k+1}(s)+ \epsilon.
\end{align*}
Therefore, using this induction relation we get
\begin{align}\label{Kujer}
s^{\frac{k}{2}}\|\partial^{k}_xv(\tilde{s})\|_{L^{\infty}_{t,x}}\lesssim \epsilon.
\end{align}
And transposing this extrinsic bound to the intrinsic quantities $|\nabla^{j}\partial_x v|$  gives
 (\ref{y98}).
Denote
\begin{align*}
Z_{2,k}(s)=\sup_{\tilde{s}\in[0,s]}\tilde{s}^{\frac{k-1}{2}}\|\partial^{k}_xv(\tilde{s})\|_{L^{\infty}_{t}L^2_{x}}.
\end{align*}
Then similarly one has
by the heat flow equation that
\begin{align*}
\|\partial^{k+1}_xv(s)\|_{L^{\infty}_{t}L^2_{x}}&\lesssim
s^{-\frac{1}{2}}\|\partial^{k}_x v(s/2)\|_{L^{\infty}_{t}L^2_{x}}
+\int^{s}_{\frac{s}{2}}(s-s')^{\frac{1}{2}}
\|\partial^{k}_x(S(v)(\partial_x,\partial_x))\|_{L^{\infty}_{t}L^2_{x}}ds'\\
&\lesssim  s^{-\frac{k+1}{2}}Z_{2,k}(s)
+\int^{s}_{\frac{s}{2}}(s-s')^{\frac{1}{2}}\|\partial^{k+1}_xv\|_{L^{\infty}_{t}L^{2}_{x}}\|dv\|_{L^{\infty}_{t,x}}ds'\\
&+\sum^{k+2}_{l=2}\sum_{\sum^{l}_{i=1}j_{i}=k+2,1\le j_i\le k}\int^{s}_{\frac{s}{2}}(s-s')^{\frac{1}{2}}\|\partial^{j_1}_xv\|_{L^{\infty}_{t}L^2_{x}}\|...\|_{L^{\infty}_{t,x}}....\|\partial^{j_l}v\|_{L^{\infty}_{t,x}}ds'.
\end{align*}
Thus we obtain by (\ref{y8}), (\ref{Kujer}) that
\begin{align*}
Z_{2,1}(s)&\lesssim  \|dv_0\|_{L^{\infty}_tL^2_x}\\
Z_{2,k'}(s)&\lesssim \|dv_0\|_{L^{\infty}_tL^2_x}, \mbox{ }\mbox{  } \forall k'\le k\Rightarrow Z_{2,k+1}(s)\lesssim \epsilon Z_{2,k+1}(s)+\|dv_0\|_{L^{\infty}_tL^2_x},
\end{align*}
which gives us (\ref{y99}) by transposing to intrinsic quantities. And we end the first step by pointing that (\ref{y100}) follows by (\ref{y8})-(\ref{y99}) and applying $\|e^{s\Delta}f\|_{L^{\infty}_x}\lesssim s^{-\frac{d}{4}}\|f\|_{L^2_x} $ to $v$ written by the Duhamel principle (see Lemma \ref{Q} for the route).

Second, we set up our bootstrap to obtain ${L^{\infty}_t{\dot H}^{\frac{d}{2}}_x}$ bounds for $v$.
Since (\ref{xvb67})-(\ref{xvb68}) have been verified in the first step, there exits $\hat{s}$ which is sufficiently large such that for $s>\hat{s}$
\begin{align}\label{dcv0}
\|d\mathcal{P}(e_l)-\chi^{\infty}_l\|_{L^{\infty}_t\dot{H}^{\frac{d}{2}}_x}\le 1.
\end{align}
Then by $\partial_iv=d\mathcal{P}(e_l)\psi^{l}_i$ we get for $s>\hat{s}$
\begin{align}\label{dcv2}
\|\partial_x v\|_{L^{\infty}_t\dot{H}^{\frac{d}{2}-1}_x}\le \epsilon.
\end{align}
Then applying Proposition  1.1 with initial time $\hat{s}$ we get for all $s\ge \hat{s}$
\begin{align}\label{dcv1}
\|d\mathcal{P}(e_l)-\chi^{\infty}_l\|_{L^{\infty}_t\dot{H}^{\frac{d}{2}}_x}\lesssim \epsilon.
\end{align}
Comparing (\ref{dcv0}) with (\ref{dcv1}), we see by bootstrap that  (\ref{dcv0}) indeed holds for all $s\ge 0$.
And thus (\ref{dcv2}) holds for all $s\ge 0$ as well.
 Then our corollary follows directly by Proposition  1.1.
\end{proof}

Our main result for this section is
\begin{Proposition}\label{mill}
Let $\sigma\in[0,\vartheta]$. Assume that $u$ is a solution to SL given in Corollary \ref{cXS}. With abuse of notations, denote $\{h_{k}(\sigma)\}$ the frequency envelopes
\begin{align}
h_{k}(\sigma)&=\sup_{s\ge 0,k'\in\Bbb Z}(1+s2^{2k'})^{4}2^{-\delta|k-k'|}2^{(\frac{d}{2}-1)k'}2^{\sigma k'}\|P_{k'}\psi_x\|_{L^{\infty}_tL^2_x\bigcap L^{p_{d}}_{t,x}}\\
h_{k}&:=h_{k}(0).
\end{align}
Suppose that
\begin{align}\label{y7788}
\sum_{\Bbb Z}h^2_{k} \le  \epsilon_1.
\end{align}
Then for all $0\le \L\le \frac{1}{2}L-1$,
\begin{align}\label{ba7788}
&\|P_{k}(d\mathcal{P}(e)-\chi^{\infty})\|_{L^{p_d}_{t,x}\bigcap L^{\infty}_tL^2_x}\lesssim (1+s2^{2k})^{-\L+1}2^{-\frac{d}{2}k}2^{-\sigma k}h_{k}(\sigma).
\end{align}
Moreover, for all $0\le \L\le L-2$, $\mathcal{G}',\mathcal{G}''$ satisfy
\begin{align}
\|P_{k}(\mathcal{G}')\|_{L^{p_d}_{t,x}\bigcap L^{\infty}_tL^2_x}&\lesssim (1+s2^{2k})^{-\L+1}2^{-\frac{d}{2}k}2^{-\sigma k}h_{k}(\sigma)\label{azji1}\\
\|P_{k}(\mathcal{G}'')\|_{L^{p_d}_{t,x}\bigcap L^{\infty}_tL^2_x}&\lesssim (1+s2^{2k})^{-\L+1}2^{-\frac{d}{2}k}2^{-\sigma k}h_{k}(\sigma).\label{azji2}
\end{align}
The connection coefficients $A_x$ satisfy
\begin{align}
\|P_{k}A_x\|_{L^{p_{d}}_{t,x}\bigcap L^{\infty}_tL^2_x}\lesssim 2^{-\sigma k+k}(1+s2^{2k})^{-\L+1}h_{k,s}(\sigma)\label{7kkk8}
\end{align}
where we denote
\begin{align}\label{8h8}
h_{k,s}(\sigma)=\left\{
                  \begin{array}{ll}
                    \sum_{-j\le l\le k}h_{l}(\sigma)h_{l}, & \hbox{ }s\in[2^{2j-1},2^{2j+1}], k+j\le 0 \\
                    2^{k+j}h_{-j}h_{k}(\sigma), & \hbox{ }s\in[2^{2j-1},2^{2j+1}], k+j\ge 0
                  \end{array}
                \right.
\end{align}
\end{Proposition}
\begin{proof}
Let $0\le \L\le \frac{1}{2}L-1$.
{\bf Step 1. Proof of (\ref{ba7788}).} We first verify the bounds for $\partial_x v$:
\begin{align}
2^{\frac{d}{2}k-k}\|P_{k}\partial_xv(\upharpoonright_{s=0})\|_{L^{\infty}_{t}L^2_x\bigcap L^{p_{d}}_{t,x}}&\lesssim 2^{-\sigma k}h_{k}(\sigma)\label{ball1}\\
2^{\frac{d}{2}k-k}\|P_{k}\partial_xv (s)\|_{L^{\infty}_{t}L^2_x\bigcap L^{p_{d}}_{t,x}}&\lesssim 2^{-\sigma k}h_{k}(\sigma)(1+s2^{2k})^{-\L}\label{ball2}.
\end{align}
Now, we turn to prove (\ref{ball1}). We will frequently use the following bilinear estimates:
\begin{align}
&\|P_{k}(fg)\|_{L^{p_{d}}_{t}L^{p_{d}}_x}\lesssim \|P_{[k-4,k+4]}f\|_{L^{p_{d}}_{t}L^{p_{d}}_x}\|P_{\le k-4}g\|_{L^{\infty}_{t}L^{\infty}_x}\nonumber\\
&+\sum_{k_1\ge k-4}2^{\frac{d}{2}k}\|P_{k_1}f\|_{L^{p_{d}}_{t}L^{p_{d}}_x}\|P_{k_1}g\|_{L^{\infty}_{t}L^{2}_x}
+ \sum_{k_2\le k-4}\|P_{k_2}f\|_{L^{p_{d}}_{t}L^{\infty}_x}\|P_{[k-4,k+4]}g\|_{L^{\infty}_{t}L^{p_{d}}_x}\label{UUU}\\
&\|P_{k}(fg)\|_{L^{\infty}_{t}L^{2}_x}\lesssim \|P_{[k-4,k+4]}f\|_{L^{\infty}_{t}L^{2}_x}\|P_{\le k-4}g\|_{L^{\infty}_{t}L^{\infty}_x}\nonumber\\
&+\sum_{k_1\ge k-4}2^{\frac{d}{2}k}\|P_{k_1}f\|_{L^{\infty}_{t}L^{2}_x}\|P_{k_1}g\|_{L^{\infty}_{t}L^{2}_x}
+ \sum_{k_2\le k-4}\|P_{k_2}f\|_{L^{\infty}_{t}L^{\infty}_x}\|P_{[k-4,k+4]}g\|_{L^{\infty}_{t}L^{2}_x}.
\label{UUU1}
\end{align}
By definition and Corollary 5.1,
\begin{align}
2^{\frac{d}{2}k-k}\|P_{k}\psi_x\upharpoonright_{s=0}\|_{L^{p_{d}}_{t,x}\bigcap L^{\infty}_{t}L^2_x}&\le 2^{-\sigma k}h_{k}(\sigma)\nonumber\\
2^{\frac{d}{2}k}(1+s2^{2k})^{\L}\|P_{k}(d\mathcal{P}(e)-\chi^{\infty})\|_{L^{\infty}_{t}L^2_{x}}&\le \epsilon.\label{voke1}
\end{align}
Then we obtain
by the identity $\partial_iv=\sum^{2n}_{l=1}\psi^{l}_id\mathcal{P}(e_{l})$ and bilinear estimates (\ref{UUU})-(\ref{UUU1})  that
\begin{align*}
\|P_{k}\partial_xv(\upharpoonright_{s=0})\|_{L^{p_{d}}_{t,x}\bigcap L^{\infty}_tL^2_x}&\lesssim
2^{-\frac{d}{2}k+k}2^{-\sigma k}h_k(\sigma)+2^{\frac{d}{2}k}\sum_{k_1\ge k-4}2^{-dk_1+k_1}2^{-\sigma k_1}h_{k_1}(\sigma)\\
&+ 2^{-\frac{d}{2}k}2^{d(\frac{1}{2}-\frac{1}{p_{d}})k}\sum_{k_2\le k-4}2^{k_2-\sigma k_2}2^{\frac{d}{p_{d}}k_2}h_{k_2}(\sigma).
\end{align*}
Since $d\ge 3$, by the slow variation of frequency envelopes we get for $\sigma\in[0,\vartheta]$
\begin{align}\label{voke2}
2^{\frac{d}{2}k-k}\|P_{k}\partial_xv\upharpoonright_{s=0}\|_{L^{p_{d}}_{t}L^{p_{d}}_x\bigcap L^{\infty}_tL^2_x}\lesssim
2^{-\sigma k}h_k(\sigma).
\end{align}
Thus (\ref{ball1}) is done.

Now, let us consider (\ref{ball2}). This follows by (\ref{ball1}) and the route of our previous work  [Step 1, Prop. 3.2,\cite{LIZE}].
Moreover, the argument of  [Step 1, Prop. 3.2, \cite{LIZE}] and Lemma  \ref{EH} in fact yield the following refined bounds
\begin{align}
2^{\frac{d}{2}k}\|P_{k}v\|_{L^{p_{d}}_{t,x}\bigcap L^{\infty}_tL^2_x}&\lesssim_{M}  2^{-\sigma k}\widetilde{h}_{k}(\sigma)(1+2^{2k}s)^{-M}\label{nn7}\\
\|P_{k}\partial_sv\|_{L^{p_{d}}_{t,x}\bigcap L^{\infty}_tL^2_x}&\lesssim_{M}2^{-\sigma k}(1+2^{2k}s)^{-M}2^{2k-\frac{d}{2}k}\widetilde{h}_{k}(\sigma),\label{nn8}
\end{align}
where $M\in[0,\frac{1}{2}L'-1]$, and  $\{\widetilde{h}_{k}(\sigma)\}$ are defined by
\begin{align*}
\widetilde{h}_{k}(\sigma)&=\sup_{k\in\Bbb Z}2^{-\delta|k-k'|}2^{\frac{d}{2}k'}\|P_{k'}v_0\|_{L^{p_{d}}_{t,x}\bigcap L^{\infty}_tL^2_x}.
\end{align*}
And (\ref{voke2}) shows for $\sigma \in[0,\vartheta]$,
\begin{align}\label{nn9}
 \widetilde{h}_{k}(\sigma)\lesssim h_{k}(\sigma).
\end{align}

{\bf Step 1.2.} Second, we transfer (\ref{nn8}) to bounds for $\psi_s$.

Using bilinear estimates (\ref{UUU})-(\ref{UUU1}) to control $\partial_sv(d\mathcal{P}(e)-\chi^{\infty})$, we conclude by (\ref{nn8}), (\ref{voke1}) that
\begin{align}\label{niuniu}
&(1+s2^{2k})^{\L}\|P_{k}\psi_s\|_{L^{\infty}_{t}L^2_x\bigcap L^{p_{d}}_{t,x}}
\lesssim 2^{-\frac{d}{2}k+2k}2^{-\sigma k}{h}_{k}(\sigma).
\end{align}
for all $\sigma\in[0,\vartheta]$.

{\bf Step 1.3.}
Third, we calculate $\|P_{k}(d\mathcal{P}(e)-\chi^{\infty})\|_{L^{p_{d}}_{t,x}}$. Recall the formula
\begin{align}
d\mathcal{P}(e)-\chi^{\infty}&=\int^{\infty}_s ({\bf D}d\mathcal{P})(e;e) \psi_sds'\label{NbZ}
\end{align}
 where $\chi^{\infty}$ denotes the limit at $s\to\infty$.
And we have the $L^{\infty}_tL^2_x$ bounds for $({\bf D}d\mathcal{P})(e;e)$
\begin{align}\label{Niu90}
\|\partial^{j}_x({\bf D}d\mathcal{P}(e;e)) \|_{L^{\infty}_t{\dot H}^{\frac{d}{2}}_x}\lesssim \epsilon 2^{-\frac{j}{2}s},
\end{align}
for any $0\le j\le L$.
Then applying bilinear estimates (\ref{UUU})-(\ref{UUU1}) to (\ref{NbZ}) with (\ref{niuniu}), (\ref{Niu90}) gives
\begin{align*}
&\|P_k[({\bf D}d\mathcal{P})(e;e) \psi_s]\|_{L^{\infty}_tL^2_x\bigcap L^{p_{d}}_{t,x}}\lesssim 2^{-\frac{d}{2}k}h_{k}(\sigma)2^{-\sigma k}(1+s2^{2k})^{-\L}.
\end{align*}
Therefore, we conclude by integrating the above estimates in $s'\in[s,\infty)$ that
\begin{align*}
&\|P_{k}(d\mathcal{P}(e)-\chi^{\infty})\|_{L^{p_d}_{t,x}\bigcap L^{\infty}_tL^2_x}\lesssim (1+s2^{2k})^{-\L+1}2^{-\frac{d}{2}k}2^{-\sigma k}h_{k}(\sigma)
\end{align*}
 for any $k\in\Bbb Z$, $0\le \L\le \frac{1}{2}L-1$, which particularly yields (\ref{ba7788}).

{\bf Step 2. Proof of (\ref{azji1})-(\ref{azji2}).}
The same arguments of Step 1.3 give (\ref{azji1}) and (\ref{azji2}) by using
\begin{align*}
&\mathcal{G}'=\int^{\infty}_s\psi_s\Xi^{\infty}ds'+\int^{\infty}_s\psi_s\mathcal{G}''ds'\\
&\mathcal{G}''=\int^{\infty}_s\psi_s\Theta^{\infty}ds'+\int^{\infty}_s\psi_s\mathcal{G}'''ds'\\
&2^{\frac{d}{2}k}\|P_{k}g\|_{L^{\infty}_tL^{2}_x}\lesssim \epsilon(1+s2^{2k})^{-\L}, \mbox{  }\forall g=\mathcal{G}'', \mathcal{G}'''.
\end{align*}

{\bf Step 3. Proof of (\ref{7kkk8}).}
Lastly, we prove (\ref{7kkk8}). Applying bilinear Littlewood-Paley decomposition to $\partial_iv\cdot d\mathcal{P}(e_l)=\psi^l_i$, we obtain by (\ref{nn7}) and   (\ref{Niu90}) that
\begin{align}\label{pnnnb1}
2^{\frac{d}{2}k-k}\|P_{k}\psi_x\|_{L^{p_{d}}_{t,x}\bigcap L^{\infty}_{t}L^{2}_x}\lesssim 2^{-\sigma k}h_{k}(\sigma) (1+2^{2k}s)^{-\L+1}.
\end{align}
Recall the schematic formula (\ref{journal}) of $A_x$ and the first order decomposition of $\mathcal{G}$:
\begin{align}
[{A_i}]^{p}_{q}(s)&= \sum\int_s^\infty (\psi_s \psi_x){\mathbf{R}}\left( e_{l_0},e_{l_1},e_{l_2},e_{l_3}\right)ds'\label{journall}\\
\mathcal{G}&:=\sum{\mathbf{R}}\left( e_{l_0},e_{l_1},e_{l_2},e_{l_3}\right)=\Gamma^{\infty}+\widetilde{\mathcal{G}}\Gamma^{\infty}+\Xi^{\infty}\int^{\infty}_s \psi_sds'+\int^{\infty}_s \psi_s\mathcal{G}'ds'\nonumber
\end{align}
By (\ref{azji1}) and (\ref{niuniu}), the curvature term $\widetilde{\mathcal{G}}$ satisfies
\begin{align*}
2^{\frac{d}{2}k}\|P_{k}\widetilde{\mathcal{G}}\|_{L^{p_{d}}_{t,x}\bigcap L^{\infty}_{t}L^{2}_x}\lesssim 2^{-\sigma k}h_{k}(\sigma)(1+2^{2k}s)^{-\L+1}.
\end{align*}
Thus by (\ref{niuniu}), (\ref{pnnnb1}) and (\ref{journall}), we get (\ref{7kkk8}).

\end{proof}

\subsection{Differential fields with respect to $t$ variable}

Proposition \ref{mill} has tracked the bounds of $\mathcal{G}',\mathcal{G}''$ in  $L^{p_{d}}_xL^{\infty}_t\bigcap L^{\infty}_tL^2_x$ along the heat flow direction. The left unknown block spaces of $F_k$ is $L^{p_{d}}_xL^{\infty}_{t}$.

First, we reduce the estimate of $L^{p_{d}}_xL^{\infty}_t$ norms to the space $L^{p_{d}}_{t,x}$ which is more flexible when handling with geometric quantities.
\begin{Lemma}\label{qq2}
If we have verified that for all $0\le \L\le \frac{1}{2}L-1$, $f=\mathcal{G}',\mathcal{G}''$ there holds
\begin{align}
2^{\frac{d}{2}k}\|P_{k}f\|_{L^{p_{d}}_{t,x}}&\lesssim_{L} 2^{-\sigma k}h_{k}(\sigma) (1+2^{2k}s)^{-\L+1}\label{zji1}\\
2^{\frac{d}{2}k-2k}2^{\frac{d}{2}k}\| \partial_tP_{k}f\|_{L^{p_{d}}_{t,x}}&\lesssim_{L}2^{-\sigma k}h_{k}(\sigma) .\label{zji2}
\end{align}
Then, consequently we have
\begin{align}
2^{-\frac{d}{d+2}k}\|\partial^{L}_x P_{k}f\|_{L^{p_{d}}_{x}L^{\infty}_t}\lesssim_{L} 2^{-\sigma k}h_{k}(\sigma)  (1+2^{2k}s)^{\frac{1}{p'_d}(1-\L)}
\end{align}
for all $0\le \L\le \frac{1}{2}L-1$, and $f=\mathcal{G}',\mathcal{G}''$.
\end{Lemma}
\begin{proof}
By Gagliardo-Nirenberg inequality and H\"older inequality, one has
\begin{align}\label{c56}
\|f\|_{L^{p_{d}}_xL^{\infty}_t}\lesssim \|f\|^{\frac{1}{p'_{d}}}_{L^{p_{d}}_{t,x}}\|\partial_t f\|^{\frac{1}{p_d}}_{L^{p_{d}}_{t,x}},
\end{align}
where $\frac{1}{p'_d}+\frac{1}{p_d}=1$. Then the Lemma follows directly.
\end{proof}

Let $\sigma\in[0,\vartheta]$, (\ref{zji1}) has been verified in Proposition \ref{mill}.
Thus Lemma \ref{qq2} reduces  the problem to prove (\ref{zji2}). We recall the following result.
\begin{Lemma}\label{xiaxia}
Let $d\ge 3$. Assume that $u$ is a solution to SL given in Proposition \ref{mill}. Let (\ref{y7788}) hold.
Assume in addition that
\begin{align}
\|P_{k}(A_x)\|_{F_{k}\bigcap S^{\frac{1}{2}}_k}&\lesssim  2^{-\frac{d}{2}k+k} 2^{-\sigma k}h_{k,s}(\sigma) (1+2^{2k}s)^{-4}\label{key2}\\
\|P_{k}(\widetilde{\mathcal{G}})\|_{F_{k}}&\lesssim  2^{-\frac{d}{2}k} 2^{-\sigma k}h_{k}(\sigma) (1+2^{2k}s)^{-4}\label{key3}
\end{align}
Then the corresponding differential field $\phi_t$ and connection coefficient $A_t$ satisfy
\begin{align}
\|P_{k}\phi_t\|_{L^{p_{d}}_{t,x}}&\lesssim 2^{-\frac{d}{2}k+2k} 2^{-\sigma k}h_{k}(\sigma)(1+2^{2k}s)^{-2}.\label{qab17788}\\
\|P_{k}(A_t)\|_{ L^{p_{d}}_{t,x}}&\lesssim 2^{-\frac{d}{2}k+2k} 2^{-\sigma k} h_{k}(\sigma)(1+s2^{2k})^{-1}.\label{pt9}
\end{align}
\end{Lemma}
\begin{proof}
The proof of [Lemma 5.5,\cite{huangBIKThuang}] and [Section 4, \cite{LIZE}] reveals  that (\ref{key2})-(\ref{key3}) imply
\begin{align}
\|P_{k}A_x(\upharpoonright_{s=0})\|_{L^{p_{d}}_{t,x}}&\lesssim 2^{-\frac{d}{2}k+k} 2^{-\sigma k}h_{k}(\sigma).\label{key11}
\end{align}
And the proof of [Lemma 5.6,\cite{huangBIKThuang}] and [Section 4, \cite{LIZE}] reveals that (\ref{qab17788}) is a corollary of
(\ref{key2})-(\ref{key3}) if one has obtained
\begin{align}
\|P_{k}\phi_t(\upharpoonright_{s=0})\|_{L^{p_{d}}_{t,x}}&\lesssim 2^{-\frac{d}{2}k+2k} 2^{-\sigma k}h_{k}(\sigma).\label{key1}
\end{align}
Let us verify (\ref{key1}). When $s=0$, there holds $\phi_t=\sqrt{-1}D_i\phi_i$. Using (\ref{key11}) to bound $\|P_{k}(A_x\phi_x)(\upharpoonright_{s=0})\|_{L^{p_{d}}_{t,x}}$, we obtain (\ref{key1}) by bilinear Littlewood-Paley decomposition and $d\ge 3$. Thus (\ref{qab17788}) is done.

Now let us prove (\ref{pt9}). Recall the formula
\begin{align} \label{bvvvv670}
A_t(s)&=\Gamma^{\infty}\int^{\infty}_s\phi_s\diamond \phi_tds'+\int^{\infty}_s(\phi_s\diamond \phi_t)\widetilde{\mathcal{G}}ds'.
\end{align}
We have seen for $d\ge 3$ in  (\ref{niuniu})
\begin{align*}
2^{\frac{d}{2}k-2k}\|P_{k}\phi_s\|_{L^{\infty}_tL^2_x\bigcap L^{p_{d}}_{t,x}}\lesssim 2^{-\sigma k}h_{k}(\sigma)(1+s2^{2k})^{-\L}.
\end{align*}
Since $d\ge 3$, bilinear Littlewood-Paley decomposition gives for $s\in [2^{2j-1}, 2^{2j+1}]$, $k,j\in\Bbb Z$,
\begin{align}
&\|P_{k}(\phi_s\diamond\phi_t)\|_{L^{p_{d}}_{t,x}}\nonumber\\
&\lesssim\sum_{k_1\le k-4,|k_2-k|\le 4} 2^{\frac{d}{2}k_1} \| P_{k_1}\phi_s\|_{L^{\infty}_tL^2_x}\|P_{k_2}\phi_t\|_{L^{p_{d}}_{t,x}}\nonumber\\
&+\sum_{k_2\le k-4,|k_1-k|\le 4} 2^{dk(\frac{1}{2}-\frac{1}{p_{d}})}2^{\frac{d}{2}k_1} \| P_{k_1}\phi_s\|_{L^{\infty}_tL^2_x}2^{\frac{d}{p_{d}}k_2}\|P_{k_2}\phi_t\|_{L^{p_{d}}_{t,x}}\nonumber\\
&+\sum_{k_1,k_2\le k-4,|k_1-k_2|\le 8} 2^{\frac{d}{2}k}  \|P_{k_1}\phi_s\|_{L^{\infty}_tL^2_x}\|P_{k_2}\phi_t\|_{L^{p_{d}}_{t,x}}\nonumber\\
&\lesssim 1_{k+j\ge 0}2^{-\frac{d}{2}k+2k}(1+s2^{2k})^{-2}2^{-\sigma k}h_{k}(\sigma)\nonumber\\
&\mbox{  }+1_{k+j\le 0}2^{dj-4j}2^{\frac{d}{2}k}2^{-\sigma k}h_{-j}(\sigma)h_{-j}.\label{9jz2}
\end{align}
By $d\ge 3$, integrating (\ref{9jz2}) in $s'\ge [s,\infty)$ yields
\begin{align*}
\int^{\infty}_s\|P_{k}(\phi_s\diamond\phi_t)\|_{L^{\infty}_tL^2_x\bigcap L^{p_{d}}_{t,x}}ds'\lesssim  2^{-\frac{d}{2}k+2k}(1+s2^{2k})^{-1}2^{-\sigma k}h_{k}(\sigma),
\end{align*}
by which the first term in the RHS of  (\ref{bvvvv670}) is done.
Moreover, (\ref{9jz2}), (\ref{key3}) and bilinear Littlewood-Paley decomposition lead to
\begin{align}
\|P_{k}(\phi_s\diamond\phi_t)\widetilde{\mathcal{G}})\|_{L^{p_{d}}_{t,x}}
&\lesssim 1_{k+j\ge 0}h_{k}(\sigma)2^{-\frac{d}{2}k+2k}(1+s2^{2k})^{-2}\nonumber\\
&+1_{k+j\le 0}2^{dj-4j}2^{\delta|k+j|}2^{\frac{d}{2}k}2^{-\sigma k}h_{-j}(\sigma)h_{-j}\label{9jz3}
\end{align}
for any $s\in [2^{2j-1}, 2^{2j+1}]$, $k,j\in\Bbb Z$. Since $d\ge 3$, integrating (\ref{9jz3}) in $s'\ge [s,\infty)$ gives acceptable bounds for the second term in the RHS of  (\ref{bvvvv670}). Thus (\ref{pt9}) is done.
\end{proof}

The following is the main result for this section.
\begin{Proposition}\label{xia2}
Let $d\ge 3$, $\sigma\in[0,\vartheta]$. Assume that $u$ is a solution to SL given in Proposition \ref{mill}.
Then we have for all $s>\infty$, $\L\in [0,\frac{1}{2}L-1]$,
\begin{align}
\|P_{k}(\mathcal{G}')\|_{L^{\infty}_{t}L^2_x\bigcap L^{p_{d}}_{t,x}}+\|P_{k}(\mathcal{G}'')\|_{L^{\infty}_{t}L^2_x\bigcap L^{p_{d}}_{t,x}}&\lesssim 2^{-\frac{d}{2}k} 2^{-\sigma k} h_{k}(\sigma)(1+s2^{2k})^{-\L+1}\label{py1}\\
2^{-\frac{d}{d+2}k}\|P_{k}(\mathcal{G}')\|_{L^{p_{d}}_xL^{\infty}_t}&\lesssim 2^{-\frac{d}{2}k} 2^{-\sigma k} h_{k}(\sigma)(1+s2^{2k})^{(1-\L)/p'_d}.\label{py2}
\end{align}
In fact, (\ref{py1})-(\ref{py2}) hold for all $\{\mathcal{G}^{(j)}\}^{\infty}_{j=0}$, where we denote
\begin{align}\label{UJ}
 \mathcal{G}^{(j)}=(\nabla^{j}{\bf R})(\underbrace{e,...,e}_{j};\underbrace{e,...,e}_{4})-{\rm limit}.
\end{align}

\end{Proposition}
\begin{proof}
(\ref{py1}) has been proved in Proposition \ref{mill}. By Lemma \ref{qq2}, it suffices to prove (\ref{zji2}).
We have
\begin{align*}
\partial_t\mathcal{G}'={{\nabla^2\bf R}}(\partial_t v,e_x; e_x,...,e_x)+
\sum_{\sum^{4}_{z=0}j_z=1}{{\nabla\bf R}}(\nabla^{j_0}_te_x; \nabla^{j_1}_te_x,...,\nabla^{j_4}_te_x).
\end{align*}
Schematically we write
\begin{align}\label{g9}
\partial_t\mathcal{G}'=\phi_t\mathcal{G}''+A_t\mathcal{G}'.
\end{align}
Since $d\ge 3$, applying bilinear Littlewood-Paley decomposition to (\ref{g9}) we obtain from (\ref{py1}) and Lemma
\ref{xiaxia} that
\begin{align}
\|P_{k}(\phi_t\mathcal{G}')\|_{L^{p_{d}}_{t,x}}+\|P_{k}(A_t\mathcal{G}')\|_{L^{p_{d}}_{t,x}}\lesssim
2^{-\frac{d}{2}k+2k}2^{-\sigma k}h_{k}(\sigma)(1+2^{2k}s)^{-1},
\end{align}
by which (\ref{zji2}) follows. Thus the proof is completed. Lastly, we observe that (\ref{UJ})  holds  for all $j$ by repeating the previous arguments.
\end{proof}

\section{Proof of Theorem 1.1 for $d\ge 3$.}

\subsection{Before iteration}

As mentioned in Section 1, the key estimates are the $F_{k}\bigcap S^{\frac{1}{2}}$ norm of $A_x$ along the heat direction (see (\ref{no3}))  and the $F_{k}$ norm of $\widetilde{\mathcal{G}}$ along the heat direction.

Recall the expression for $\widetilde{\mathcal{{G}}}$:
\begin{align}
\widetilde{{\mathcal G}}&=\mathcal{G}-\Gamma^{\infty}=\mathcal{U}_0+\mathcal{U}_1\nonumber\\
\mathcal{U}_0&:=\Xi^{\infty}
\int^{\infty}_s(\partial_i\phi_i)ds'\label{notion1}\\
\mathcal{U}_1&:=\Xi^{\infty}\int^{\infty}_s(A_i\phi_i)ds'
+\int^{\infty}_s(\partial_i\phi_i)\mathcal{G}'ds'+\int^{\infty}_s(A_i\phi_i)\mathcal{G}'ds'\label{notion2}
\end{align}

\begin{Lemma}\label{nohj}
Let $u$ be solution to $SL$ in $C([-T,T];\mathcal{Q}(\Bbb R^d,\mathcal{N}))$. And let $\{h_{k}(\sigma)\}$ be frequency envelope such that
\begin{align}\label{m1}
2^{(\frac{d}{2}-1)k}\|\phi_x(s)\|_{F_{k}}\lesssim 2^{-\sigma k}h_{k}(\sigma)(1+s2^{2k})^{-4}.
\end{align}
Suppose that
\begin{align}\label{m0}
\|h_{k}(0)\|_{\ell^2}\le \epsilon\ll1.
\end{align}
Moreover, we assume that
\begin{align}\label{buy88}
2^{\frac{d}{2}k}\|P_{k} {\mathcal{U}}_{1}\|_{F_{k}}\lesssim   (1+2^{2k}s)^{-4}h_{k}.
\end{align}
Then we have
\begin{align}\label{noty}
2^{\frac{d}{2}k-k}\|P_{k}A_x\|_{F_{k}\bigcap S^{\frac{1}{2}}_k}\lesssim  (1+2^{2k}s)^{-4} 2^{-\sigma k}h_{k,s}(\sigma),
\end{align}
where $h_{k,s}(\sigma)$ is defined by (\ref{8h8}).
\end{Lemma}
\begin{proof}
Since (\ref{m1}) dominates (\ref{notion1}), and (\ref{buy88}) bounds (\ref{notion2}), we get
\begin{align}\label{xz0889}
2^{\frac{d}{2}k}\|P_{k}\widetilde{\mathcal{G}}\|_{F_{k}}\le  (1+s2^{2k})^{-4} 2^{-\sigma k}h_{k}(\sigma).
\end{align}
Let $B\ge 1$ denote the smallest constant such that
\begin{align}\label{Nin}
(1+s2^{2k})^{4}2^{\frac{d}{2}k-k}\|P_{k}A_x\|_{F_{k}\bigcap S^{\frac{1}{2}}_k}\le B 2^{-\sigma k}h_{k,s}(\sigma),
\end{align}
for all $\sigma\in[0,\vartheta],k\in\Bbb Z,s>\infty$.
Then one can check
\begin{align}\label{buy77}
2^{\frac{d-2}{2}k+\sigma k}\int^{\infty}_s\|P_{k}[(\phi_s\diamond\phi_x) \widetilde{\mathcal{G}}]\|_{F_{k}\bigcap S^{\frac{1}{2}}_k}ds'\lesssim (B\epsilon+1)h_{k,s}(\sigma).
\end{align}
It is now standard to derive (\ref{buy77}) from (\ref{xz0889}) and (\ref{Nin}), see our previous paper (\cite{LIZE}, Lemma 4.1, Step 1). Therefore, we obtain
\begin{align*}
B\lesssim 1+\epsilon B,
\end{align*}
by which (\ref{noty}) follows.
\end{proof}

We now prove a stronger estimate of (\ref{buy88}).
\begin{Lemma}\label{pkij90}
Let $u$ be solution of SL satisfying Lemma \ref{nohj}.
Then for all $k\in\Bbb Z$ we have
\begin{align}\label{buyo78}
\|P_{k} {\mathcal{U}}_1\|_{F_{k}}\lesssim \epsilon 2^{-\frac{d}{2}k} 2^{-\sigma k}h_{k}(\sigma)(1+s2^{2k})^{-4}.
\end{align}
\end{Lemma}
\begin{proof} Recall that  $F_{k}\hookrightarrow L^{p_{d}}_{t,x}\bigcap L^{\infty}_tL^2_x$ in $L^2_{k}$.
Let $d\ge 3$, then Lemma \ref{nohj} and  (\ref{xz0889}) show that for $u$ in Lemma \ref{pkij90}, the assumptions (\ref{key2})-(\ref{key3}) of Lemma \ref{xiaxia} hold. Thus by Proposition \ref{xia2} we obtain (\ref{py1})-(\ref{py2}).

\emph{Bilinear estimates.}
we will use Lemma \ref{1bi1} to do bilinear estimates.
Assume that $s\in[2^{2j_0}-1,2^{2j_0+1}]$. We prove the lemma according to $k+j_0\ge 0$ or  $k+j_0\le 0$.
For $s'\in[2^{2j}-1,2^{2j+1}]$, assumption (\ref{m1}) and Lemma \ref{nohj} give
\begin{align}
\|A_i\phi_i\|_{F_{k}}&\lesssim 2^{2k-\frac{d}{2}k}2^{-\sigma k}2^{-\frac{k+j}{2}}h^{2}_{-j}h_{-j}(\sigma)\mbox{ }\mbox{ }\mbox{ }\mbox{ }{\rm{if}}\mbox{ }k+j\le 0\label{j9}\\
\|A_i\phi_i\|_{F_{k}}&\lesssim 2^{2k-\frac{d}{2}k}2^{-\sigma k}(1+2^{2k+2j})^{-{4}}h^{2}_{-j}h_{k}(\sigma)\mbox{ }\mbox{ }{\rm{if}}\mbox{ }k+j\ge 0.\label{i9}
\end{align}
Thus (\ref{notion1}) is done by integrating (\ref{j9})-(\ref{i9}) w.r.t. $j\ge j_0$.

For (\ref{notion2}), we apply Lemma \ref{1bi1}. We take the more complex term $\int^{\infty}_s(A_i\phi_i)\mathcal{G}'ds'$  of (\ref{notion2}) as the candidate, the $\partial_i\phi_i\mathcal{G}'$ term is easier. Using Lemma \ref{1bi1}, the $\rm{High\times Low}$ interaction of $(A_i\phi_i)\mathcal{G}'$ is dominated by
\begin{align}
\sum_{|k_1-k|\le4} \|P_{k}(P_{k_1}(A_i\phi_i)P_{\le k-4}\mathcal{G}')\|_{F_{k}}\lesssim \|P_{k}(A_i\phi_i)\|_{F_{k}}.
\end{align}
Thus the $\rm{High\times Low}$ part is done by (\ref{j9}), (\ref{i9}).

Now let us consider the $\rm{High\times High}$ part of $(A_i\phi_i)\mathcal{G}'$. By Lemma \ref{1bi1} and (\ref{py1}),
\begin{align}
&\sum_{|k_1-k_2|\le 8,k_1,k_2\ge k-4} \|P_{k}(P_{k_1}(A_i\phi_i)P_{k_2}\mathcal{G}')\|_{F_{k}}\nonumber\\
&\lesssim \sum_{|k_1-k_2|\le 8,k_1,k_2\ge k-4}(2^{\frac{d}{d+2}(k_1-k)}+2^{\frac{d-1}{2}(k_1-k)})
\|P_{k_1}(A_i\phi_i)\|_{F_{k}}\|P_{k_2}\mathcal{G}'\|_{L^{\infty}_x}\nonumber\\
&\lesssim \sum_{k_1,k_2\ge k-4}(2^{\frac{d}{d+2}(k_1-k)}+2^{\frac{d-1}{2}(k_1-k)})
\|P_{k_1}(A_i\phi_i)\|_{F_{k_1}}h_{k_1}(1+2^{2k_1+2j})^{-20}.\label{j00}
\end{align}
Applying (\ref{i9}), for $k+j_0\ge 0$, we get (\ref{j00}) is bounded by
\begin{align*}
&\sum_{k_1\ge k-4}(2^{\frac{d}{d+2}(k_1-k)}+2^{\frac{d-1}{2}(k_1-k)}) 2^{2k_1-\frac{d}{2}k_1}
2^{-\sigma k_1}(1+2^{2k_1+2j})^{-24}h^{2}_{-j}h_{k_1}h_{k_1}(\sigma)\\
&\lesssim  2^{2k-\frac{d}{2} k}2^{-\sigma k}(1+2^{2k+2j})^{-20}h^{3}_{k}h_{k}(\sigma).
\end{align*}
Thus the $\rm{High\times High}$ part for $k+j_0\ge 0$ is done. If $k+j\le 0$, (\ref{j00}) is dominated by
\begin{align*}
&\sum_{k_1\ge -j}\left[2^{\frac{d}{d+2}(k_1-k)}+2^{\frac{d-1}{2}(k_1-k)}\right] 2^{2k_1-\frac{d}{2}k_1}
2^{-\sigma k_1}(1+2^{2k_1+2j})^{-24}h^{2}_{-j}h_{k_1}h_{k_1}(\sigma)\\
&+\sum_{k\le k_1\le -j}(2^{\frac{d}{d+2}(k_1-k)}+2^{\frac{d-1}{2}(k_1-k)})2^{-\frac{j+k_1}{2}} 2^{2k_1-\frac{d}{2}k_1}
h^{2}_{-j}h_{-j}2^{-\sigma k_1}h_{k_1}(\sigma)\\
&\lesssim  2^{2k-\frac{d}{2} k}2^{-\sigma k}h^{2}_{k}h_{k}(\sigma)+2^{-\sigma k}h^{2}_{-j}h_{-j}(\sigma)F_{j,k}(d)
\end{align*}
where $F_{j,k}(d)$ is defined by
\begin{align*}
{{\text{F}}_{j,k}}(d){\text{ = }}\left\{ \begin{gathered}
  {2^{\frac{d}
{{d + 2}}( - j - k)}}{2^{ - 2j + \frac{d}
{2}j}} + {2^{ - \frac{{d - 1}}
{2}k}}{2^{ - \frac{3}
{2}j}},d = 3,4 \hfill \\
  {2^{ - \frac{{k + j}}
{2}}}{2^{2k - \frac{d}
{2}k}}{2^{\delta |k + j|}} + {2^{ - \frac{{d - 1}}
{2}k}}{2^{ - \frac{3}
{2}j}},d \geqslant 5 \hfill \\
\end{gathered}  \right.
\end{align*}
Then if $k+j_0\le 0$, we have
\begin{align*}
&\sum_{|k_1-k_2|\le 8, k_1,k_2\ge k-4}\int^{\infty}_s\|P_{k}[P_{k_1}(A_i\phi_i)P_{k_2}\mathcal{G}']\|_{F_{k}}ds'\\
&\lesssim \sum^{\infty}_{j\ge -k}2^{2j} 2^{2k-\frac{d}{2} k}h^{2}_{k}h_{k}(\sigma)(1+2^{2j+2k})^{-20}
+ \sum_{j_0\le j\le -k}2^{2j}2^{-\sigma k}h^{2}_{-j}h_{-j}(\sigma)2^{-\sigma k}F_{j,k}(d)\\
&\lesssim 2^{-\sigma k}2^{-\frac{d}{2} k}h^{2}_{k}h_{k}(\sigma).
\end{align*}
Hence the $\rm{High\times High}$ part for all $s>0$ is done.

Now let us consider the $\rm{Low\times High}$ part of $(A_i\phi_i)\mathcal{G}'$.
By Lemma \ref{1bi1}, (\ref{py1})-(\ref{py2})
\begin{align}
&\sum_{|k_2-k|\le 4,k_1\le k-4} \|P_{k}(P_{k_1}(A_i\phi_i)P_{k_2}\mathcal{G}')\|_{F_{k}}\nonumber\\
&\lesssim \sum_{|k-k_2|\le 4,k_1\le k-4}2^{\frac{d-1}{2}(k_1-k)}\|P_{k_1}(A_i\phi_i)\|_{F_{k_1}}\|P_{k}\mathcal{G}'\|_{L^{\infty}}
+\|P_{k_1}(A_i\phi_i)\|_{L^{\infty}}2^{\frac{d}{d+2}k}\|P_{k}\mathcal{G}'\|_{L^{p_{d}}_xL^{\infty}_t}\nonumber\\
&\lesssim \sum_{|k-k_2|\le 4,k_1\le k-4}2^{-\sigma k}h_{k}(\sigma)
\left(2^{\frac{d-1}{2}(k_1-k)}\|P_{k_1}(A_i\phi_i)\|_{F_{k_1}}
+2^{-\frac{d}{2}k}2^{\frac{d}{2}k_1}\|P_{k_1}(A_i\phi_i)\|_{F_{k_1}}\right).\label{j11}
\end{align}
Then by (\ref{j9}), for $k+j_0\ge 0$ one has (\ref{j11}) is dominated by
\begin{align*}
&(1+2^{2k+2j})^{-10}2^{-\sigma k}h_{k}(\sigma)
\sum_{-j\le k_1\le k-4}\left(2^{\frac{d-1}{2}(k_1-k)}2^{-\frac{d}{2}k_1+2k_1}+ 2^{-\frac{d}{2}k}2^{2k_1}\right)h^2_{-j}h_{k_1}
\\
&+ (1+2^{2k+2j})^{-10}2^{-\sigma k}h_{k}(\sigma)\sum_{k_1\le -j}\left(2^{\frac{d-1}{2}(k_1-k)}2^{-\frac{d}{2}k_1+2k_1}+2^{-\frac{d}{2}k}2^{2k_1}\right)
2^{-\frac{k_1+j}{2}}h^2_{-j}h_{k_1}\\
&\lesssim (1+2^{2k+2j})^{-10}2^{-\sigma k}h_{k}(\sigma)\left( 2^{-\frac{d}{2}k+2k}h^2_{-j}h_{k}
+h^3_{-j}2^{-\frac{3}{2}j}2^{-\frac{d-1}{2}k}+h^3_{-j}2^{-\frac{d}{2}k}2^{-2j}\right).
\end{align*}
Thus the $\rm{Low\times High}$ part of $(A_i\phi_i)\mathcal{G}'$ for $k+j_0\ge 0$ is done by integrating
the above in $s\ge 2^{2j_0-1}$.

By (\ref{i9}), for $k+j\le 0$ we see (\ref{j11}) is dominated by
\begin{align*}
&2^{-\sigma k}h_{k}(\sigma)
\sum_{k_1\le k}\left(2^{\frac{d-1}{2}(k_1-k)}2^{-\frac{d}{2}k_1+2k_1}+2^{-\frac{d}{2}k}2^{2k_1}\right)
2^{-\frac{k_1+j}{2}}h^2_{-j}h_{k_1}\\
&\lesssim 2^{-\sigma k}h_{k}(\sigma)
h^2_{-j}h_{k}2^{-\frac{k+j}{2}}2^{-\frac{d}{2}k+2k}.
\end{align*}
Hence, the $\rm{Low\times High}$ part of $(A_i\phi_i)\mathcal{G}'$ for $k+j_0\le 0$ is bounded by
\begin{align*}
&\int^{\infty}_s\|P^{lh}_{k}[(A_i\phi_i)\mathcal{G}']\|_{F_{k}}ds'\\
&\lesssim \sum_{j_0\le j\le -k} 2^{2j}2^{-\sigma k}h_{k}(\sigma)
h^2_{-j}h_{k}2^{-\frac{k+j}{2}}2^{-\frac{d}{2}k+2k}\\
&+\sum_{-k\le j<\infty}2^{2j}(1+2^{2k+2j})^{-10}2^{-\sigma k}h_{k}(\sigma)\left( 2^{-\frac{d}{2}k+2k}h^2_{-j}h_{k}
+h^3_{-j}2^{-\frac{3}{2}j}2^{-\frac{d-1}{2}k}+h^3_{-j}2^{-\frac{d}{2}k}2^{-2j}\right)\\
&\lesssim 2^{-\sigma k}h_{k}(\sigma)h^3_{k}2^{-\frac{d}{2}k}.
\end{align*}

So all the three interaction parts are done and our lemma follows.
\end{proof}

\begin{Corollary}\label{IJ}
Let $u\in C([-T,T];\mathcal{Q}(\Bbb R^d,\mathcal{N}))$ be solution of SL satisfying (\ref{m0}), (\ref{m1}).
Then for all $k\in\Bbb Z$ we have
\begin{align}\label{0xx9ijn}
\|P_{k} {\widetilde{\mathcal{G}}}\|_{F_{k}}\lesssim  2^{-\frac{d}{2}k} 2^{-\sigma k}h_{k}(\sigma)(1+s2^{2k})^{-4}.
\end{align}
Moreover, the connection coefficients satisfy
\begin{align}\label{0xx}
\|P_{k} A_x(s)\|_{F_{k}\bigcap S^{\frac{1}{2}}_k}\lesssim  2^{-\frac{d}{2}k+k} 2^{-\sigma k}h_{k,s}(\sigma)(1+s2^{2k})^{-4}.
\end{align}
\begin{proof}
 Define the function $\Phi:[0,T_*)\to \Bbb R^+$ by
\begin{align}
\Phi(T):=\sup_{T'\in[0,T]}\sup_{s>0,k\in\Bbb Z} 2^{\frac{d}{2}k} 2^{\sigma k}h^{-1}_{k}(\sigma)(1+s2^{2k})^{4}\|P_{k} {\mathcal{U}}_1\|_{F_{k}(T')}.
\end{align}
$\Phi(T)$ is a continuous function in $T\in [0,T_*)$ by Sobolev embeddings, Lemma \ref{aaaHeat} and the fact $\{h_{k}(\sigma)\}$ is a frequency envelope.
Let ${\bf u}(s,x)$ denote the solution to the heat flow equation with initial data $u_0$.
By the relation $F_{k}\hookrightarrow L^2_xL^{\infty}_t\bigcap L^{p_{d}}$, we see
\begin{align}
\lim_{T\downarrow 0}\Phi(T)=\sup_{s>0,k\in\Bbb Z} 2^{\frac{d}{2}k} 2^{\sigma k}h^{-1}_{k}(\sigma)(1+s2^{2k})^{4}\|P_{k} {\bf{U}}_1\|_{L^2_x},\label{iju809}
\end{align}
where ${\bf U}_1$ is defined by
\begin{align}
 {\bf{U}}_1=\int^{\infty}_s {{\phi}_s\mathcal{{G}}'}ds',
\end{align}
with all values taken at the point ${\bf u}(s',x)$ in the above integral. We have seen in the proof of Proposition \ref{mill} that the RHS of  (\ref{iju809}) is controlled by $\epsilon_1$. Thus we get
\begin{align}
\lim_{T\downarrow 0}\Phi(T)\lesssim \epsilon_1.
\end{align}
And we have seen $(\ref{buy88})\Rightarrow (\ref{buyo78})$ in Lemma \ref{pkij90}. Hence there holds
\begin{align}
\Phi(T)\lesssim 1\Rightarrow \Phi(T)\lesssim \epsilon_1.
\end{align}
By bootstrap, for all $T\in[0,T_*)$ we conclude
\begin{align*}
\Phi(T)\lesssim \epsilon_1.
\end{align*}
Then the bound (\ref{0xx9ijn}) for $\widetilde{\mathcal{G}}$ follows by adding the $\|\mathcal{U}_0\|_{F_{k}}$ part.
\end{proof}

\end{Corollary}

The connection bound (\ref{0xx}) suffices to bound the evolution of $\phi_{x,t}$ along the heat direction. And for $s=0$,  (\ref{0xx}) suffices to control the evolution of $\phi_{x}$ along the Schr\"odinger direction.
We omit the details for this part in high dimensions since it is relatively easy to supplement them following our previous work [Section 5, \cite{LIZE}].
In fact, it suffices to bound the cubic terms of the form
\begin{align}
\phi_\mu\phi_\nu\widetilde{\mathcal{G} }
\end{align}
in the $F_{k},L^{p_{d}}$ and $N_k$ spaces.

Then one can get
\begin{Proposition}\label{01}
Let $\sigma\in [0,\vartheta]$
and $\epsilon_0$ be a sufficiently small constant.  Assume that  $u\in C([-T,T];\mathcal{Q}(\Bbb R^d,\mathcal{N}))$ is the solution to SL with initial data $u_0$. Let $\{c_k\}$ be an $\epsilon_0$-frequency envelope of order $\frac{1}{8}\delta$. And let $\{c_k(\sigma)\}$ be another frequency envelope of order $\frac{1}{8}\delta$ for which
\begin{align}
2^{\frac{d}{2}k}\|P_{k}  u_0\|_{L^2_x}&\le c_k\label{zzzzw1}\\
2^{\frac{d}{2}k}\|P_{k}  u_0\|_{L^2_x}&\le c_k(\sigma)2^{-\sigma k}\label{zzzzw}
\end{align}
Denote $\{\phi_i\}^{d}_{i=1}$ the corresponding differential fields of the heat flow initiated from $u$. Suppose also that at the heat initial time $s=0$,
\begin{align}\label{PI}
\sum^{d}_{i=1}2^{\frac{d}{2}k-k}\|P_{k}\phi_i\|_{G_k(T)}&\le \epsilon^{-\frac{1}{2}}_0c_k.
\end{align}
Then when $s=0$, we have for all $i=1,...,d$, $k\in\Bbb Z$,
\begin{align}
2^{\frac{d}{2}k-k}\|P_{k}\phi_i\|_{G_k(T)}&\lesssim c_k\label{mmmmmmm1}\\
2^{\frac{d}{2}k-k}\|P_{k}\phi_i\|_{G_k(T)}&\lesssim c_k(\sigma)2^{-\sigma k}.\label{mmmmmmm}
\end{align}
\end{Proposition}
\begin{proof}
By bootstrap assumption (\ref{PI}) and the fact $\{c_{k}\}$ is  an  $\epsilon_0$-frequency envelope, we see
the frequency envelope
\begin{align}\label{H2}
b_{k}(\sigma)=\sup_{k'\in\Bbb Z}2^{-\delta|k-k'|}2^{\frac{d}{2}k'-k'}2^{\sigma k'}\|P_{k'}\phi_x(\upharpoonright_{s=0})\|_{G_{k'}}
\end{align}
satisfies
\begin{align}
\sum_{k\in \Bbb Z}b^2_{k}\le \epsilon_0.
\end{align}
Thus by $G_{k}\hookrightarrow F_k$, we see that the assumptions (\ref{m0}), (\ref{m1}) hold. Then applying Corollary \ref{IJ} gives at $s=0$
\begin{align}
\|P_{k}A_x(\upharpoonright_{s=0})\|_{L^{p_{d}}_{t,x}}\lesssim 2^{-\frac{d}{2}k+k}2^{-\sigma k}b_{k}(\sigma)\\
\|P_{k}A_t(\upharpoonright_{s=0})\|_{L^2_{t,x}}\lesssim 2^{-\frac{d}{2}k+2k}2^{-\sigma k}b_{k}(\sigma)\\
\|P_{k}\phi_t(\upharpoonright_{s=0})\|_{L^2_{t,x}}\lesssim 2^{-\frac{d}{2}k+2k}2^{-\sigma k}b_{k}(\sigma).
\end{align}
Using the evolution equation of $\phi_x$ along the Schr\"odigner direction and arguments of \cite{huangBIKThuang} and our previous work [Section 5,\cite{LIZE}],  one can prove
\begin{align}
b_{k}(\sigma)\lesssim c_{k}(\sigma)+\epsilon b_{k}(\sigma).
\end{align}
Thus (\ref{mmmmmmm1}) and (\ref{mmmmmmm}) are proved.
\end{proof}

\subsection{Iteration for once}

\begin{Proposition}\label{02}
Given $\sigma\in [\vartheta,2\vartheta]$.
Let $\epsilon_0$ be a sufficiently small constant.  Let $u\in C([-T,T];\mathcal{Q}(\Bbb R^d,\mathcal{N}))$ be the solution to SL with initial data $u_0$. Let $\{c_k\}$ be an $\epsilon_0$-frequency envelope of order $\frac{1}{16}\delta$. And let $\{c_k(\sigma)\}$ be another frequency envelope of order $\frac{1}{16}\delta$ which satisfies
\begin{align*}
2^{\frac{d}{2}k}\|P_{k} u_0\|_{L^2_x}&\le c_k  \\
2^{\frac{d}{2}k}\|P_{k} u_0\|_{L^2_x}&\le 2^{-\sigma k}c_k(\sigma)
\end{align*}
Denote $\{\phi_i\}^{d}_{i=1}$ the corresponding differential fields of the heat flow initiated from $u$. Suppose also that at the heat initial time $s=0$,
\begin{align}\label{H1}
2^{\frac{d}{2}k-k}\|P_{k}\phi_i(s=0)\|_{G_k(T)}&\le \epsilon^{-\frac{1}{2}}_0c_k.
\end{align}
Then when $s=0$, we have for all $i=1,...d$, $k\in\Bbb Z$,
\begin{align}
2^{\frac{d}{2}k-k}\|P_{k}\phi_i\|_{G_k(T)}&\lesssim c_k\label{4.32}\\
2^{\frac{d}{2}k-k}\|P_{k}\phi_i\|_{G_k(T)}&\lesssim 2^{-\sigma k}[c_k(\sigma)+c_k(\sigma-\vartheta)c_{k}(\vartheta)].\label{4.33}
\end{align}
\end{Proposition}
\begin{proof}
The key point and the engine for our iteration is the estimates of $\partial_s v$.
Thus we begin with improving $\partial_s v$ in Proposition 5.1 (see (5.34)).

Let $\sigma \in [\vartheta,2\vartheta]$. Applying Proposition \ref{01} with $\sigma_0\in[0,\vartheta]$, we have seen
\begin{align}
2^{\frac{d}{2}k-k}\|P_{k}\phi_x(s=0)\|_{F_{k}}&\lesssim 2^{-\sigma_0k}c_{k}(\sigma_0)\\
2^{\frac{d}{2}k-k}\|P_{k}\phi_x(s)\|_{F_{k}}&\lesssim 2^{-\sigma_0 k}c_{k}(\sigma_0)(1+2^{2k}s)^{-4},
\end{align}
which combined with Proposition 5.1 gives
\begin{align}
2^{\frac{d}{2}k}\|P_{k}(d\mathcal{P}(e)-\chi^{\infty})\|_{L^{p_{d}}_{t,x}\bigcap L^{\infty}_tL^2_x}&\lesssim 2^{-\sigma_0 k}c_{k}(\sigma_0)(1+s2^{2k})^{-l+1}\label{P1}\\
2^{\frac{d}{2}k}\|P_{k}v\|_{L^{p_{d}}_{t,x}\bigcap L^{\infty}_tL^2_x}&\lesssim 2^{-\sigma_0 k}c_{k}(\sigma_0)(1+2^{2k}s)^{-l}\label{P2}\\
\|P_{k}(S^{(1)}(v)-S^{(1)}_{\infty})\|_{L^{p_{d}}_{t,x}\bigcap L^{\infty}_tL^2_x}&\lesssim 2^{-\sigma_0 k}c_{k}(\sigma_0)(1+2^{2k}s)^{-l}\label{P3}\\
2^{\frac{d}{2}k-2k}\|P_{k}\partial_sv(s)\|_{_{L^{p_{d}}_{t,x}\bigcap L^{\infty}_tL^2_x}}&\lesssim 2^{-\sigma_0 k}c_{k}(\sigma_0)(1+2^{2k}s)^{-l}.\label{P4}
\end{align}
for $l\in[0,\frac{1}{2}L-1]$.
Define the frequency envelope $\{b_{k}(\sigma)\}$ as (\ref{H2}) with $\sigma\in[0,2\vartheta]$. Then by Proposition \ref{01},
\begin{align}\label{viuiui}
b_{k}(\sigma_0)\lesssim c_{k}(\sigma_0) ,\forall \mbox{  }\sigma_0\in[0,\vartheta].
\end{align}
Then by definition of $b_{k}(\sigma)$ and (\ref{P1}) we infer from bilinear Littlewood-Paley decomposition that
\begin{align}\label{H3}
2^{\frac{d}{2}k-k}\|P_{k}(\partial_x v)(\upharpoonright_{s=0})\|_{L^{\infty}_tL^2_x\bigcap  L^{p_{d}}_{t,x}}\lesssim 2^{-\sigma k} b_{k}(\sigma)+2^{-\sigma k}c_{k}(\vartheta)c_{k}(\sigma-\vartheta), \mbox{ }\forall \sigma \in [\vartheta,2\vartheta].
\end{align}
Let $\mathcal{J}_1(s)$ be the positive continuous function defined on $[0,\infty)$ via
\begin{align}
\mathcal{J}_{1,\sigma, l}(s)=\sup_{k\in\Bbb Z,\tilde{s}\in[0,s]}(1+s2^{2k})^{l}2^{\frac{d}{2}k }2^{\sigma k} \|P_{k}v\|_{L^{\infty}_tL^2_x\bigcap  L^{p_{d}}_{t,x}}1/b^{(1)}_{k}(\sigma)
\end{align}
where $b^{(1)}_{k}(\sigma)$ is defined by
\begin{align}
b^{(1)}_{k}(\sigma)=\left\{
                      \begin{array}{ll}
                        b_{k}(\sigma), & \hbox{ } \sigma\in[0,\vartheta]\\
                        b_{k}(\sigma)+2^{-\sigma k}c_{k}(1)c_{k}(\sigma-1)& \hbox{ }\sigma\in(\vartheta,2\vartheta]
                      \end{array}
                    \right.
\end{align}
(\ref{P2}) has shown
\begin{align}\label{tfcvghjkl}
\sup_{s\ge 0}\mathcal{J}_{1,\sigma,M}(s)\lesssim 1, \mbox{ }\forall \sigma\in [0,\vartheta].
\end{align}
Then by (\ref{H3}) and (\ref{viuiui}), we see
\begin{align}\label{idgu}
\lim_{s\to 0}\mathcal{J}_{1,\sigma,l}(s)\lesssim 1.
\end{align}
By Duhamel principle for the heat flow equation, we get
\begin{align}
\|P_{k}v\|_{L^{\infty}_{t}L^2_x\bigcap L^{p_{d}}_{t,x}}&\lesssim e^{-c(d) 2^{2k}}\|P_{k}v_0\|_{L^{\infty}_{t}L^2_x\bigcap L^{p_{d}}_{t,x}}\nonumber\\
&+\int^{s}_0 e^{-c(d)(s-\tau)2^{2k}}\|P_{k}(S(v)(\partial_xv,\partial_xv)\|_{L^{\infty}_{t}L^2_x\bigcap L^{p_{d}}_{t,x}}.\label{H4}
\end{align}
By Lemma \ref{EH2}, (\ref{P3}), (\ref{tfcvghjkl}) and the definition of $b^{(1)}_k(\sigma)$, (\ref{H4}) is dominated by
\begin{align}\label{H7}
&2^{\frac{d}{2}k}\|P_{k}(S(v)(\partial_xv,\partial_xv)\|_{L^{\infty}_{t}L^2_x\bigcap L^{p_{d}}_{t,x}}
\lesssim \mathcal{J}_{1,\sigma,l} 2^{-\sigma k+2k}b^{(1)}_{k}(\sigma).
\end{align}
for all $\sigma\in[0,2\vartheta]$, $d\ge 3$.
We remark that since (\ref{P3}) only reaches $\sigma_0\in[0,\vartheta]$, one needs to gain  $c_{k}(\vartheta)$ from $\partial_xv\partial_x v$ while applying Lemma \ref{EH2}. Precisely, this problem only occurs in the $I_1$ case (see Proof of  Lemma \ref{EH2}), and we can estimate $I_1$ as
\begin{align}
&\sum_{k_1\le k_2,k_3\ge k_2+5}\|P_{k}(P_{k_3}S(v)P_{k_1}\partial_x vP_{k_2}\partial_xv)\|_{L^{p_d}_{t,x}\bigcap L^{\infty}_tL^2_x}\nonumber\\
&\lesssim  2^{-\frac{d}{2}k}\widetilde{\alpha}_k(\sum_{k_1\le k}2^{k_1}\widetilde{\beta}_{k_1})^2\nonumber\\
&\lesssim  (1+s2^{2k})^{-l} 2^{-\frac{d}{2}k}2^{2k}2^{-(\sigma-\vartheta) k}c_{k}(\sigma-\vartheta)(\sum_{k_1\le k}2^{k_1-\vartheta k}c_{k_1}(\vartheta)) 2^{k}c_{k}\nonumber\\
&\lesssim  (1+s2^{2k})^{-l}2^{-\frac{d}{2}k}2^{2k-\sigma k}b^{(1)}_{k}(\sigma).\label{iteration1}
\end{align}
Thus we arrive at
\begin{align*}
(1+2^{2k}s)^{l}\|P_{k}v\|_{L^{\infty}_{t}L^2_x\bigcap L^{p_{d}}_{t,x}}&\lesssim (1+2^{2k}s)^{l}e^{-c(d) 2^{2k}}\|P_{k}v_0\|_{L^{\infty}_{t}L^2_x\bigcap L^{p_{d}}_{t,x}}\nonumber\\
&+(1+\epsilon\mathcal{J}_{1,\sigma,l}(s))(1+2^{2k}s)^{l}\int^{s}_0 e^{-c(d)(s-\tau)2^{2k}}(1+\tau 2^{2k})^{-l} 2^{-\sigma k+2k}b^{(1)}_{k}(\sigma)d\tau,
\end{align*}
which further shows
\begin{align*}
\mathcal{J}_{1,\sigma, l}(s)\lesssim 1+\epsilon \mathcal{J}_{1,\sigma, l}(s).
\end{align*}
Therefore, by (\ref{idgu}), one gets for all $\sigma\in[0,2\vartheta]$
\begin{align}\label{H5}
2^{\frac{d}{2}k } \|P_{k}v\|_{L^{\infty}_tL^2_x\bigcap  L^{p_{d}}_{t,x}} \lesssim (1+2^{2k}s)^{-l}2^{-\sigma k}b^{(1)}_{k}(\sigma).
\end{align}
and using the heat flow equation, (\ref{H5}) and (\ref{H7}) yield
\begin{align}\label{H8}
2^{\frac{d}{2}k-2k}\|P_{k}\partial_sv\|_{L^{\infty}_tL^2_x\bigcap  L^{p_{d}}_{t,x}}\lesssim 2^{-\sigma k} b^{(1)}_k(\sigma)(1+2^{2k}s)^{-l}.
\end{align}

Then by bilinear Littlewood-Paley decomposition, (\ref{H8}) and the frame bound (\ref{P1}), $\phi_s$ is improved to be
dominated by
\begin{align}\label{H9}
2^{\frac{d}{2}k-2k}\|P_{k}\phi_s\|_{L^{p_{d}}_{t,x}\bigcap L^{\infty}_tL^2_x}\lesssim 2^{-\sigma k} b^{(1)}_k(\sigma)(1+2^{2k}s)^{-l+1}
\end{align}
for all $\sigma\in [0,2\vartheta]$.
Then by (\ref{H9}) the frame bound (\ref{P1}) now can be ameliorated in the range of  $\sigma$ as
\begin{align}\label{H10}
2^{\frac{d}{2}k}\|P_{k}(d\mathcal{P}(e)-\chi^{\infty})\|_{L^{\infty}_{t}L^2_x\bigcap L^{p_{d}}_{t,x}}
\lesssim 2^{-\sigma k}b^{(1)}_{k}(\sigma)(1+2^{2k}s)^{-l+2}.
\end{align}
for all $\sigma\in[0,2\vartheta]$. And similarly,
\begin{align}\label{H11}
2^{\frac{d}{2}k}\|P_{k}(\mathcal{G}')\|_{L^{\infty}_{t}L^2_x\bigcap L^{p_{d}}_{t,x}}.
\lesssim 2^{-\sigma k}b^{(1)}_{k}(\sigma)(1+2^{2k}s)^{-l+2}.
\end{align}

Until now, we have improved all the results before Section 5.2 to $\sigma\in[0,2\vartheta]$. Then repeating the arguments of Section 5.2, one obtains for $d\ge 3$, $\sigma\in[0,2\vartheta]$
\begin{align}
2^{\frac{d}{2}k}\|P_{k}(\mathcal{G}')\|_{L^{\infty}_{t}L^2_x\bigcap L^{p_{d}}_{t,x}}
&\lesssim 2^{-\sigma k}b^{(1)}_{k}(\sigma)(1+s2^{2k})^{-l+2} \nonumber\\
2^{\frac{d}{d+2}k}\|P_{k}(\mathcal{G}')\|_{L^{p_{d}}_xL^{\infty}_{t} }
&\lesssim 2^{-\sigma k}b^{(1)}_{k}(\sigma)(1+s2^{2k})^{(2-l)/p'_d}.\label{H12}
\end{align}
With (\ref{H11}) and (\ref{H12}), running the bootstrap programme in Section 6.1 again gives
\begin{align}\label{H12}
b_{k}(\sigma)\lesssim c_{k}(\sigma)+\epsilon b^{(1)}_{k}(\sigma)
\end{align}
for all $\sigma\in [0,2\vartheta]$.
Then Proposition \ref{02} follows.
\end{proof}

\subsection{$j-th$ iteration and Proof of Theorem 1.2}

Repeating the above iteration scheme for $k$ times yields
\begin{Proposition}\label{03}
Let $\vartheta\in[1-10^{-9},1-10^{-10}]$ be a fixed constant. Let $\delta=\frac{1}{d10^{100}}$.
Given $\sigma\in [j\vartheta,(j+1)\vartheta]$, $j\in\Bbb N$.
Let $\epsilon_0$ be a sufficiently small constant depending only on $j,d$.  Let $u\in C([-T,T];\mathcal{Q}(\Bbb R^d,\mathcal{N}))$ is the solution to SL with initial data $u_0$. Let $\{c_k\}$ be an $\epsilon_0$-frequency envelope of order $\frac{1}{2^{j+3}}\delta$. And let $\{c_k(\sigma)\}$ be another frequency envelope of order $\frac{1}{2^{j+3}}\delta$ which satisfies
\begin{align}
2^{\frac{d}{2}k}\|P_{k} u_0\|_{L^2_x}&\le c_k\label{1wei}\\
2^{\frac{d}{2}k}\|P_{k} u_0\|_{L^2_x}&\le 2^{-\sigma k}c_k(\sigma)\label{2wei}
\end{align}
Denote $\{\phi_i\}^{d}_{i=1}$ the corresponding differential fields of the heat flow initiated from $u$. Suppose also that at the heat initial time $s=0$,
\begin{align}\label{H1}
2^{\frac{d}{2}k-k}\|P_{k}\phi_i(s=0)\|_{G_k(T)}&\le \epsilon^{-\frac{1}{2}}_0c_k.
\end{align}
Then when $s=0$, we have for all $i=1,...d$, $k\in\Bbb Z$,
\begin{align}
2^{\frac{d}{2}k-k}\|P_{k}\phi_i\|_{G_k(T)}&\lesssim  2^{-\sigma k} c^{(j)}_{k}(\sigma )\label{YU2}\\
2^{\frac{d}{2}k}\|P_{k}(d\mathcal{P}(e)-\chi^{\infty})\|_{L^{\infty}_tL^2_x}&\lesssim 2^{-\sigma k} c^{(j)}_{k}(\sigma )\label{YU1}
\end{align}
where $c^{(j)}_{k}(\sigma)$ is defined by induction:
\begin{align}
c^{(0)}_{k}(\sigma)&=c_{k}(\sigma),\mbox{ }{\rm{if}} \mbox{ }\sigma\in[0,\vartheta]\\
c^{(j+1)}_{k}(\sigma)&=c^{(j)}_{k}(\sigma), \mbox{ }{\rm{if}} \mbox{ }\sigma\in[0,j\vartheta]\\
c^{(j+1)}_{k}(\sigma)&=c_{k}(\sigma)+c^{(j)}_{k}(\sigma-\vartheta)c_{k}(\vartheta),\mbox{ } {\rm{if}} \mbox{ }\sigma\in(j\vartheta,(j+1)\vartheta].
\end{align}
\end{Proposition}
\begin{proof}
We will use the following dynamical separations:
\begin{align}
({\bf D}^{l}d\mathcal{P})(e,...,e;e)-{\rm limits}&=\int^{\infty}_s ({\bf D}^{l+1}d\mathcal{P})(e,...,e;e) \psi_sds'\\
 \mathcal{G}^{(l)}-{\rm limits}&=\int^{\infty}_s {\mathcal{G}}^{(l+1)} \psi_sds'\\
 D^{l} S(v) -{\rm limits}&=\int^{\infty}_s {D}^{l+1}S(v)\partial_s vds',
\end{align}
and we denote $({\bf D}^cd\mathcal{P})(\underbrace{e,...,e}_{c};e)$ by $[d\mathcal{P}]^{(c)}$ for simplicity in the following.
To make the statement clear we introduce following notations:
\begin{align*}
  (S_N^a):&{\left\| {\partial^{\ell}_x{D^a}S(v)} \right\|_{L_t^\infty L_x^2}} \le \varepsilon {s^{ - \ell/2}},\forall 0 \le \ell \le  {K_0} + N  \\
  (S^a_{j,N}):&{2^{\frac{d}
{2}k}}{\left\| {{P_k}{D^a}S(v)} \right\|_{L_t^\infty L_x^2 \cap L_{t,x}^{{p_d}}}} \le  {2^{ - \sigma k}}c_k^{(j)}(\sigma ){(1 + {2^{2k}}s)^{ - L}},\forall 0 \le  L \le {K_0} + N\\
  ({V_{j,N}}):&{2^{\frac{d}
{2}k}}{\left\| {{P_k}v} \right\|_{L_t^\infty L_x^2 \cap L_{t,x}^{{p_d}}}} \le  {2^{ - \sigma k}}c_k^{(j)}(\sigma ){(1 + {2^{2k}}s)^{ - K}},\forall 0 \le  K \le  {K_0} + N \hfill \\
   ({V^{0}_{j}}):&{2^{\frac{d}
{2}k}}{\left\| {{P_k}v} (\upharpoonright_{s=0})\right\|_{L_t^\infty L_x^2 \cap L^{{p_d}}_{t,x}}}
 \lesssim  {2^{ - \sigma k}}c_k^{(j)}(\sigma )\\
 (V^{s}_{j,N}):&{2^{\frac{d}
{2}k}}{\left\| {{P_k}v_s} \right\|_{L_t^\infty L_x^2 \cap L_{t,x}^{{p_d}}}} \le  {2^{ 2k- \sigma k}}c_k^{(j)}(\sigma ){(1 + {2^{2k}}s)^{ - K}},\forall 0 \le  K \le  {K_0} + N.
\end{align*}
These are heat flow quantities. Introduce the following notations for curvature parts:
\begin{align*}
  &(G^b_N):{ \| {\partial^{\ell}_x \mathcal{G}^{(b)}} \|_{L_t^\infty L_x^2}} \lesssim \varepsilon {s^{ - \ell/2}},\forall \ell\in[ 0,{K_0} + N] \hfill \\
  &( {G}^{b}_{j,N}):{2^{\frac{d}
{2}k}}{ \|  {P_k} \mathcal{{G}}^{(b)} \|_{L_t^\infty L_x^2 \cap L_{t,x}^{{p_d}}}}\lesssim {2^{ - \sigma k}}c_k^{(j)}(\sigma ){(1 + {2^{2k}}s)^{ - K}},\forall K\in[ 0,{K_0} + N] \hfill \\
  &( {G}^{b * }_{j,N}):{2^{\frac{d}
{2}k-\frac{dk}{d+2}}}{\| {{P_k} \mathcal{G}^{(b)}} \|_{L_x^{{p_d}}L_t^\infty }} \lesssim {2^{ - \sigma k}}c_k^{(j)}(\sigma ){(1 + {2^{2k}}s)^{ - K}},\forall   K\in[0, {K_0} + N]  \\
  &( {G}_{j,N}):{2^{\frac{d}
{2}k}}{\| {{P_k} \mathcal{G}^{(1)}} \|_{L^{{p_d}}\cap L_t^\infty L^2_x }}+2^{\frac{d}{d+2}k}\| {{P_k} \mathcal{G}^{(1)}} \|_{L^{{p_d}}_x L_t^\infty}\lesssim {2^{ - \sigma k}}c_k^{(j)}(\sigma ){(1 + {2^{2k}}s)^{ - K}},\forall   K\in[0, {K_0} + N]
 \end{align*}
 frame parts
 \begin{align*}
  &(E^c_N): \left\|  \partial^{\ell}_x[d\mathcal{P}]^{(c)}  \right\|_{L_t^\infty L_x^2} \lesssim\varepsilon {s^{ -{\ell}/2}},\forall  {\ell} \in[0, {K_0} + N]\hfill \\
  &(E^c_{j,N}):{2^{\frac{d}
{2}k}}\|  P_k [d\mathcal{P}]^{(c)} \|_{L_t^\infty L_x^2 \cap L_{t,x}^{{p_d}}}  \lesssim {2^{ - \sigma k}}c_k^{(j)}(\sigma ){(1 + {2^{2k}}s)^{ - K}},\forall  K \in[0, {K_0} + N],
\end{align*}
and connection parts:
\begin{align*}
&(A{O_j}){2^{\frac{d}
{2}k - k}}{\left\| {{P_k}{A_x}} \right\|_{S_k^{1/2} \cap {F_k}}} \lesssim {2^{ - \sigma k}}c_{k,s}^{(j)}(\sigma ){(1 + {2^{2k}}s)^{ - 4}}  \\
&( {\Phi^{t}_j}){2^{\frac{d}
{2}k - 2k}}({\left\| {{P_k}{A_t}} (\upharpoonright_{s=0} )\right\|_{L_{t,x}^{{p_d}}}} + {\left\| {{P_k}{\phi _t}}(\upharpoonright_{s=0} ) \right\|_{L_{t,x}^{{p_d}}}}) \lesssim {2^{ - \sigma k}}c_k^{(j)}(\sigma )\\
&( {\Phi^{x}_j}){2^{\frac{d}
{2}k - 2k}}( \left\| {{P_k}{\phi_t}}(\upharpoonright_{s=0} )\right\|_{L_{t,x}^{{p_d}} \bigcap L^{\infty}_{t}L^2_x}) \lesssim  {2^{ - \sigma k}}c_k^{(j)}(\sigma ).
\end{align*}
Now, the induction relation can be written as
 \begin{align*}
(\lambda_1)    &(S^a_N) + (V^s_{0,N}) \Rightarrow (S^{a - 1}_{0,N - 1}) ;\mbox{  }(S^{a - 1}_{0,N - 1}) + (V^s_{1,N - 1} ) \Rightarrow (S_{1,N - 2}^{a - 1})  \\
 &\underline{(S^{a - k}_{j,N - k}) + (V^s_{j,N - 1} ) \Rightarrow (S_{j,N - 2}^{a - k - 1}).\mbox{  }\mbox{  }}\\
(\lambda_2)    &\underline{(S^{0}_{j,N})  + (V_{j,N})  \Rightarrow (V^{s}_{j,N})\mbox{  }\mbox{  } \mbox{  }\mbox{  }\mbox{  }\mbox{  }\mbox{  }\mbox{  }\mbox{  }\mbox{  }\mbox{  }\mbox{  }}\\
(\lambda_3)      & (S^0_N) + (V_{0,N}) \Rightarrow (V^{o}_{0,N} ) \Rightarrow ({V^{*}_{0,N}}) \\
&\underline{(S^{0}_{j - 1,N} ) + (V^{o}_{j,N} ) \Rightarrow (V^{*}_{j,N} ) \Rightarrow ({V_{j,N}})\mbox{  }}\\
(\lambda_4) &(E^0_N) + (V^{s}_{0,N} ) \Rightarrow (\Phi^{s}_{0,N} );
\mbox{ } (E^0_{1,N} ) + (V^s_{1,N} ) \Rightarrow (\Phi^s _{1,N}) \hfill \\
& \underline{(E^0_{j,N} ) + (V^s_{j,N} ) \Rightarrow (\Phi^s _{j,N} ).\mbox{  }\mbox{  }}\\
(\lambda_5)      &(E^c_N ) + (\Phi^s _{0,N} ) \Rightarrow (E^{c-1}_{0,N - 1});
  (E^{c-1}_{0,N - 1} ) + (\Phi^s _{1,N - 1} ) \Rightarrow (E^{c-2}_{1,N - 2} ) \\
&\underline{(E^{c-k}_{j,N - k} ) + (\Phi^s_{j,N - k} ) \Rightarrow (E^{c-k-1}_{j,N - k - 1} )\mbox{  }\mbox{  }}\\
 (\lambda_6)       &(A{O_j}) + (G_{j,N}) \Rightarrow (A{O_{j + 1}})   \\
 &(\Phi^t_j ) + (G^{b + 1}_{j,N}) \Rightarrow (G^{b*}_{j,\frac{1}{2}N })  \\
 &(A{O_j}) + (G_{j,N}) \Rightarrow (\Phi^t_j )   \\
 &\underline{ (G_{j,N})\Rightarrow (A{O_j}).\mbox{  }\mbox{  }}\\
   (\lambda_7)      &({V^{0}_{j}})\Rightarrow({V_{j,N}})\\
 &\underline{(E_{j,0} )+(\Phi^{x}_{j})\Rightarrow ({V^{0}_{j}})\mbox{  }\mbox{  }}\\
   (\lambda_8)      & SL \mbox{  }equation \Rightarrow (\Phi^{x}_{j}).
 \end{align*}
And each time estimates like  (\ref{iteration1}) give additional $c_{k}(\vartheta)$ which inspires the definition of $c^{(j)}_k(\sigma)$.
Thus in order to reach $\sigma \in[j\vartheta,(j+1)\vartheta]$, the top derivative orders and the sufficient decay order  in $(S^{a}_{N}), (G^b_N), (E^c_N)$ we need are
\begin{align}\label{yghpl}
S^{j+2}_{4+2j}, G^{j+2}_{8+j2^{j}}, E^{j+2}_{4+2j}.
\end{align}
The three quantities of (\ref{yghpl}) are purely heat flow related and follow by Section 3 and Section 4 via choosing a large $L'$.
\end{proof}

\subsection{Uniform Sobolev bounds, well-posedness and asymptotic behavior}

By (\ref{03}) and (\ref{YU1}) we see
\begin{align}
2^{\frac{d}{2}k}\|P_{k}v\|_{L^{\infty}_tL^2_x}&\lesssim   c^{(j)}_{k}(\sigma )2^{-\sigma k},
\end{align}
which shows
\begin{align}
\||\nabla|^{\beta}u\|_{L^{\infty}_tL^2_x}&\lesssim  \|u_0\|_{\dot{H}^{\frac{d}{2}}_x\bigcap\dot{H}^{\beta}_x }, \mbox{ }\beta=\sigma+\frac{d}{2}.
\end{align}
Hence considering the conservation of energy we conclude
\begin{align}\label{dfgbnmbn}
\| \partial_xu\|_{L^{\infty}_tH^j_Q}&\lesssim  \|\partial_xu_0\|_{H^{j}_Q },
\end{align}
for all $j\ge [\frac{d}{2}]+1$ if $\epsilon_1$ is sufficiently small depending only on $d,j$.
Applying the local well-posedness result of \cite{huangDWhuang} or \cite{huangMchuang} shows $u$ is global. And the Sobolev bounds (\ref{dfgbnmbn}) hold uniformly in $t\in\Bbb R$.
For  well-posedness, one may copy the arguments of [Section 7.3, \cite{LIZE}] which was based on \cite{huangBIKThuang,huangTataruhuang}.
The the asymptotic behavior (\ref{FFF}) can be proved as and indeed more easily than  [Section 7.4, \cite{LIZE}].
\section*{Appendix}

{\bf Proof of Lemma \ref{EH2}.}
\begin{proof}
The proof is an adaptation of [\cite{huangBIKThuang}, Lemma 8.2]. By Lemma \ref{EH} and symmetry, it suffices to bound the $\|P_{k}(F(v)(P_{k_1}\partial_x vP_{k_2}\partial_xv)\|_{L^{p_d}_{t,x}}$ part with
$k_1\le k_2$. Then we consider three subcases:
\begin{align*}
&\sum_{k_1\le k_2}\|P_{k}(F(v)P_{k_1}\partial_x vP_{k_2}\partial_xv)\|_{L^{p_d}_{t,x}}\\
&\lesssim \sum_{k_1\le k_2,k_3\ge k_2+5}\|P_{k}(P_{k_3}F(v)P_{k_1}\partial_x vP_{k_2}\partial_xv)\|_{L^{p_d}_{t,x}} \\
&+\sum_{k_1\le k_2,|k_2-k_3|\le 4}\|P_{k}(P_{k_3}F(v)P_{k_1}\partial_x vP_{k_2}\partial_xv)\|_{L^{p_d}_{t,x}}\\
& +\sum_{k_1\le k_2,k_2\ge k-4}\|P_{k}(P_{\le k_2-5}F(v)P_{k_1}\partial_x vP_{k_2}\partial_xv)\|_{L^{p_d}_{t,x}}\\
&:=I_1+I_2+I_3.
\end{align*}
$I_1$ is dominated by
\begin{align*}
&\sum_{k_1\le k_2,k_3\ge k_2+5}\|P_{k}(P_{k_3}F(v)P_{k_1}\partial_x vP_{k_2}\partial_xv)\|_{L^{p_d}_{t,x}}\\
&\lesssim \sum_{|k-k_3|\le 4}\| P_{k_3}F(v)\|_{L^{p_{d}}_{t,x}}
\sum_{k_1\le k_2\le k-4}\|P_{k_1}2^{\frac{d}{2}k_1}\partial_x v\|_{L^{\infty}_tL^2_x}2^{\frac{d}{2}k_2}\|P_{k_2}\partial_xv \|_{L^{\infty}_{t}L^2_x}\\
&\lesssim 2^{-\frac{d}{2}k}\widetilde{\alpha}_k(\sum_{k_1\le k}2^{k_1}\widetilde{\beta}_{k_1})^2.
\end{align*}
$I_2$ is bounded by
\begin{align*}
&\sum_{k_1\le k_2-4;|k_2-k_3|\le 4; k_2,k_3\ge k-5}\|P_{k}(P_{k_3}F(v)P_{k_1}\partial_x vP_{k_2}\partial_xv)\|_{L^{p_d}_{t,x}}\\
&+\sum_{|k_1-k_2|\le 4,|k_2-k_3|\le 4;k_1,k_2,k_3\ge k-10}\|P_{k}(P_{k_3}F(v)P_{k_1}\partial_x vP_{k_2}\partial_xv)\|_{L^{p_d}_{t,x}}\\
&\lesssim 2^{\frac{d}{2}k}\sum_{k_2\ge k-5}\| P_{k_2}F(v)\|_{L^{\infty}_tL^{2}_x}\| P_{k_2}\partial_x v\|_{L^{p_{d}}_{t,x}} (\sum_{k_1\le k_2}2^{\frac{d}{2}k_1}\|P_{k_1}\partial_x vP_{k_2}\|_{L^{\infty}_{t}L^2_{x}})\\
&\lesssim 2^{\frac{d}{2}k}\sum_{k_2\ge k-5}2^{-dk_2+k_2}\widetilde{\alpha}_{k_2}\widetilde{\beta}_{k_2}  (\sum_{k_1\le k_2}2^{k_1} \widetilde{\beta}_{k_1}).
\end{align*}
$I_3$ is dominated by
\begin{align*}
&\sum_{k_1\le k_2-4, |k_2-k|\le 4}\| F(v)\|_{L^{\infty}_{t,x}}2^{\frac{d}{2}k_1}\|P_{k_1}\partial_x v\|_{L^{\infty}_{t}L^2_x}\|P_{k_2}\partial_xv)\|_{L^{p_d}_{t,x}}\\
&+\sum_{k_1\le k_2, |k_2-k_1|\le 8,k_1,k_2\ge k-4 }2^{\frac{d}{2}k}\|F(v)\|_{L^{\infty}_{t,x}}
\|P_{k_1}\partial_xv\|_{L^{\infty}_{t}L^2_x}\|P_{k_2}\partial_xv\|_{L^{p_d}_{t,x}}\\
&\lesssim 2^{-\frac{d}{2}k}\widetilde{\beta}_{k}(\sum_{k_1\le k}2^{k_1}\widetilde{\beta}_{k_1})
+2^{\frac{d}{2}k}\sum_{k_1\ge k-4}2^{2k_1-dk_1}\widetilde{\beta}^2_{k_1}.
\end{align*}
\end{proof}

Similar to  [Lemma 3.1, \cite{LIZE}], we have
\begin{Lemma}\label{aaaHeat}
Let $u\in C([-T,T];\mathcal{Q}(\Bbb R^d, \mathcal{N}))$ solve $SL$. And let $v(s,t,x)$ be the solution of heat flow  equation  with initial data $u(t,x)$. Then given any $M\in\Bbb Z_+$, $M\ge 200$, for any $0\le \sigma\le 2M$, there exist  constants $\epsilon_{M}>0$, $C_{M}>0$, $C_{M,T}$, such that if $\|u\|_{L^{\infty}_t{\dot H}^{\frac{d}{2}}_x}\le \epsilon_{M}\ll 1$,
then for any $s\ge 0$, $i=1,2,...,d$, $\rho=0,1$, $m=0,1,...,M$,
 \begin{align*}
 \|\partial^{\rho}_{t}\partial^{m}_x(v-Q)\|_{L^{\infty}_tH^{M}_x}&\le  C_{M,T}(s+1)^{-\frac{m}{2}} \\
(2^{-\frac{1}{2}k}+2^{\sigma k})2^{\frac{d}{2}k-k}\| P_{k} \phi_i\|_{L^{\infty}_tL^2_x}&\le  C_M(2^{2k}s+1)^{-30} \\
(2^{-\frac{1}{2}k} +2^{\sigma k})2^{\frac{d}{2}k-k}\| P_{k} A_i\|_{L^{\infty}_tL^2_x}&\le C_M(2^{2k}s+1)^{-28} \\
2^{m k}2^{\frac{d}{2}k-k}\| P_{k} \partial_tA_i\|_{L^{\infty}_tL^2_x}&\le  C_{M.T}(2^{2k}s+1)^{-24}\\
 2^{m k}2^{\frac{d}{2}k-k}\| P_{k} \partial_t\phi_i\|_{L^{\infty}_tL^2_x}&\le  C_{M,T}(2^{2k}s+1)^{-24} 
\end{align*}
\end{Lemma}


\begin{thebibliography}{999}

 \scriptsize\bibitem{huangBhuang} I. Bejenaru,  \emph{On Schr\"odinger maps}, Amer. J. Math., {\bf 130} (2008), 1033-1065.

\bibitem{huangBIKhuang} I. Bejenaru, A. Ionescu, and C. Kenig. \emph{Global existence and uniqueness of Schr\"odinger maps in dimensions $d\ge4$}. Adv.  Math., {\bf215}(1), 263-291, 2007.

\bibitem{huangBIKThuang} I. Bejenaru, A. Ionescu, C. Kenig, D. Tataru. \emph{Global Schr\"odinger maps in dimensions $d\ge2$: Small data in the critical Sobolev spaces}.  Ann. of Math.,  {\bf173}(3), 1443-1506, 2011.

\bibitem{huangBIKT2huang} I. Bejenaru, A. Ionescu, C. Kenig, D. Tataru. \emph{Equivariant Schr\"odinger maps in two spatial dimensions}. Duke Math. J., {\bf162}(11), 1967-2025, 2013.

\bibitem{huangBIKT3huang} I. Bejenaru, A. Ionescu, C. Kenig, D. Tataru. \emph{Equivariant Schr\"odinger Maps in two spatial dimensions: the $\Bbb H^2$ target}. Kyoto J. Math.  {\bf56}, 283-323, 2016.


\bibitem{huangCSUhuang} N.H. Chang, J. Shatah, and K. Uhlenbeck. \emph{Schr\"odinger maps}. Comm. Pure Appl. Math., {\bf53}(5), 590-602, 2000.

\bibitem{huangDWhuang} W.Y. Ding, and Y.D. Wang. \emph{Local Schr\"odinger flow into K\"ahler manifolds}. Sci. China Math., Ser. A
{\bf44}, 1446-1464, 2001.

\bibitem{huangDShuang} B. Dodson, P. Smith. \emph{A controlling norm for energy-critical Schr\"odinger maps}. Transactions of AMS,   {\bf367}(10), 7193-7220, 2015.

\bibitem{huangFLhuang} {J.L. Ericksen}. \emph{Nilpotent energies in liquid crystal theory}. Archive for Rational Mechanics Analysis. {\bf 10} (1), 189-196, 1962.

\bibitem{Grafakos} L. Grafakos, S. Oh. \emph{The Kato-Ponce Inequality}.  Comm. Partial Differential Equations, {\bf 39}(6), 1128-1157, 2014.

\bibitem{huangGKT1huang} S. Gustafson, K. Kang, T.P. Tsai. \emph{Asymptotic stability of harmonic maps under the Schr\"odinger flow}. Duke Math. J.,  {\bf145}(3), 537-583, 2008.


\bibitem{huangGNThuang} S. Gustafson, K. Nakanishi, T.P. Tsai, \emph{Asymptotic stability, concentration, and oscillation
in harmonic map heat-flow, Landau-ifshitz, and Schr\"oinger maps on $R^2$}. Comm. Math. Phys. 300:1 (2010), 205-242.


\bibitem{Khuanghuang}  H. Koch, D. Tataru, and M. Visan. \emph{Dispersive Equations and Nonlinear Waves. Oberwolfach Seminars Springer Basel}, 2014. Part II, Chapter 5.


\bibitem{huangIK1huang} A. D. Ionescu and C. E. Kenig. \emph{Low-regularity Schrodinger maps}, Differential Integral Equations 19 (2006), 1271-1300.

\bibitem{huangIKhuang} A. D. Ionescu and C. E. Kenig. \emph{ Low-regularity Schrodinger maps. II. Global well-posedness in dimen sions $d\ge 3$}, Comm. Math. Phys. {\bf 271} (2007), 523-559.

\bibitem{huangKPVhuang} C. E. Kenig, G. Ponce, and L. Vega. \emph{Smoothing effects and local existence theory for the generalized non linear Schrodinger equations}, Invent. Math. 134 (1998), 489-545.

\bibitem{huangKMhuang} S. Klainerman and M. Machedon. \emph{Space-time estimates for null forms and the local existence theorem}, Comm. Pure Appl. Math. {\bf46} (1993), 1221-1268.


\bibitem{huangKShuang} S. Klainerman and S. Selberg. \emph{ Remark on the optimal regularity for equations of wave maps type}, Comm. Partial Differential Equations 22 (1997), 901-918.


\bibitem{huangKrhuang} J. Krieger.\emph{  Global regularity of wave maps from $R^{2+1}$ to $H^2$. Small energy}, Comm. Math. Phys. {\bf 250}, 507-580,  2004.

 \bibitem{huangLLhuang} L. Landau, and E. Lifshitz. \emph{On the theory of the dispersion of magnetic permeability in
ferromagnetic bodies}. Phys. Z. Sovietunion, {\bf8}, 153-169, 1935.

\bibitem{huangLawriehuang} {A. Lawrie, S.J. Oh and S. Shahshahani. \emph{The Cauchy problem for wave maps on hyperbolic space in dimensions $d\ge4$}. Int. Math. Res. Not. IMRN 2016£¬rnw272.}

\bibitem{huangLPhuang} F. Linares and G. Ponce. \emph{On the Davey-Stewartson systems}, Ann. Inst. H. Poincare Anal. Non Lineaire 10 (1993), 523-548.

\bibitem{LIZE} Z. Li.  \emph{Global Schr\"odinger map flows to K\"ahler manifolds with small data in critical Sobolev spaces: Energy critical case.} arxiv preprint.

\bibitem{huangMchuang}  H. McGahagan. \emph{An approximation scheme for Schr\"odinger maps}. Comm. Partial Differential Equations , {\bf32}(3), 375-400, 2007.


\bibitem{huangMPRhuang} F. Merle, P. Raphael, I. Rodnianski. \emph{Blowup dynamics for smooth data equivariant solutions to the critical Schr\"odinger map problem.} Invent. Math., {\bf193}(2), 249-365, 2013.



\bibitem{NSU1} A. Nahmod, A. Stefanov, and K. Uhlenbeck. \emph{On Schrodinger maps}, Comm. Pure Appl. Math. {\bf 56} (2003), 114-151.




\bibitem{huangNSU2huang}  A. Nahmod, A. Stefanov, and K. Uhlenbeck. \emph{On the well-posedness of the wave map problem
in high dimensions}. Comm. Anal. Geom.,{\bf 11}(1):49-83, 2003.


\bibitem{huangOhhuang} S. Oh. \emph{Finite energy global well-posedness of the   Yang-Mills equations on $R^{1+3}$: an approach using the Yang-Mills heat flow }. Duke Math. J. {\bf 164}, 1689-1732, 2015.

\bibitem{huangOThuang} S. Oh, D. Tataru. \emph{The threshold conjecture for the energy critical hyperbolic Yang-Mills equation.} arxiv. preprint. 2017.

 \bibitem{huangOT1huang} S. Oh. D. Tataru. \emph{The Yang-Mills heat flow and the caloric gauge}. arxiv. preprint. 2017.

\bibitem{huangPhuang} G. Perelman. \emph{Blow up dynamics for equivariant critical Schr\"odinger maps}. Comm. Math. Phys., 330(1), 69-105, 2014.


\bibitem{huangRRShuang} I. Rodnianski, Y. Rubinstein, G. Staffilani. \emph{On the global well-posedness of the one-dimensional Schr\"odinger map flow}. Anal. PDE, {\bf2}(2), 187-209, 2009.

\bibitem{Sh} J. Shatah. Regularity results for semilinear and geometric wave equations. Banach Center Publications, {\bf 41}, 69-90, 1997.


\bibitem{huangS2huang} P. Smith.   \emph{Conditional global  regularity of  Schr\"odinger maps: sub-threshold dispersed energy}. Anal. PDE,   {\bf 6}(3), 601-686, 2013.

\bibitem{huangSmithhuang} P. Smith. \emph{Geometric renormalization below the ground state}. Int. Math. Res. Not. IMRN, no. 16, 3800-3844, 2012.

\bibitem{huangstruwehuang} M. Struwe.  \emph{On the evolution of harmonic maps in higher dimensions}. J. Differential. Geom., {\bf 28}, 485-502, 1988.

\bibitem{huangSSBhuang} P.L. Sulem, C. Sulem, C. Bardos. \emph{On the continuous limit for a system of classical spins}. Comm. Math. Phys., {\bf107}(3), 431-454, 1986.

\bibitem{huangTaohuang} T. Tao. \emph{Geometric renormalization of large energy wave maps}. Journees equations aux derivees partielles, 1-32, 2004.

\bibitem{huangT3huang}  T. Tao. \emph{Gauges for the schrodinger map}, unpublished. {\tiny http://www. math. ucla.edu/~tao/preprints/Expository}.

\bibitem{huangT4huang} {  T. Tao. \emph{Global regularity of wave maps. III-VII}. arxiv preprint 2008-2009.}


\bibitem{huangTataru3huang} { D. Tataru. \emph{On global existence and scattering for the wave maps equation}. Amer. J. Math.,
{\bf 123}(1), 37-77, 2001.}

\bibitem{huangTataruhuang} D. Tataru. \emph{Rough solutions for the wave maps equation}. Amer. J. Math., {\bf 127}(2), 293-377, 2005.

\end{thebibliography}
\end{document}